\documentclass[11pt,a4paper,reqno]{amsproc}

\input{newheader.sty}

\firsthead{Version of \today\ generated from file \jobname\ at \currenttime}

\title{An enhanced term in the Szeg\H{o}-type asymptotics for the free massless Dirac operator}

\author[L.\ Bollmann]{Leon Bollmann}
\address[Leon Bollmann]{Mathematisches Institut,
  Ludwig-Maximilians-Universit\"at M\"unchen,
  Theresienstra\ss{e} 39,
  80333 M\"unchen, Germany}
\email{bollmann@math.lmu.de}

\begin{document}

\begin{abstract}
We consider a regularised Fermi projection of the Hamiltonian of the massless Dirac equation at Fermi energy zero. The matrix-valued symbol of the resulting operator is discontinuous at the origin.
For this operator, we prove Szeg\H{o}-type asymptotics with the spatial cut-off domains given by $d$-dimensional cubes. For analytic test functions, we obtain a $d$-term asymptotic expansion and provide an upper bound of logarithmic order for the remaining terms. This bound does not depend on the regularisation. In the special case that the test function is given by a polynomial of degree less or equal than three, we prove a $(d+1)$-term asymptotic expansion with an error term of constant order. The additional term is of logarithmic order and its coefficient is independent of the regularisation.
\end{abstract}

\maketitle
\thispagestyle{empty}

\section{Introduction}
The study of Szeg\H{o}-type asymptotics dates back to the analysis of the determinant of large Toeplitz matrices initiated by Szeg\H{o} \cite{Szego1915,Szego1952}. While, in this one-dimensional Toeplitz matrix case, the obtained asymptotic expansion features a leading volume term, the order of the subleading term crucially depends on the continuity of the symbol.
For symbols with jump discontinuities, the subleading term is not of constant order but features a logarithmic enhancement \cite{Fisher.Hartwig.1969, MR0331107, Widom19761, Basor}. 
Continuous analogues for these discrete Szeg\H{o}-type asymptotic expansions, where the Toeplitz matrices are replaced by Wiener--Hopf operators, were soon derived \cite{Kac, Widom1960, Linnik, Widom1974}. In the present article we focus on the multidimensional continuous case and refer to \cite{MR1724795, MR2223704, MR2831118} for further discussion on the Toeplitz matrix case.

In the $d$-dimensional continuum case, a complex-valued symbol $a:\R^d\rightarrow\C$, and the corresponding Wiener--Hopf operator with action 
\begin{equation}\label{Wiener-Hopf-smooth}
(T_Lu)(x)=\frac{1}{(2\pi)^d}1_{\Lambda_L}(x)\int_{\R^d}\int_{\Lambda_L} \e^{\i \xi(x-y)}a(\xi)u(y)\,\d y\,\d\xi, \qquad\qquad x\in\R^d
\end{equation}
on Schwartz functions $u\in\mathcal{S}(\R^d)$ are studied. Here, $\Lambda_L$ is a sufficiently regular domain scaled with the scaling parameter $L>0$.
As in the discrete case, two distinct situations, depending on the smoothness of the symbol, are considered. 

In the case of a symbol without discontinuities, a two-term asymptotic expansion with a leading volume term and a subsequent surface area term, i.e.~of the form
\begin{equation}\label{expansion-non-enhanced}
\tr_{L^2(\R^d)}[g(T_L)]=L^{d}A_0(g;a)+L^{d-1}A_1(g;a)+o(L^{d-1}),\qquad \text{as} \ L\rightarrow\infty,
\end{equation}
holds for sufficiently regular test function $g$ with $g(0)=0$. This expansion can also be extended to matrix-valued symbols \cite{Widom1980} at the cost of a less explicit coefficient $A_1$. For analytic, and in some cases sufficiently smooth, test functions, it is also possible to obtain additional terms in the asymptotic expansion. For domains with $C^1$-boundary, a three-term asymptotic expansion was obtained in \cite{Roccaforte}. In \cite{Widom1985FullExpansion}, for arbitrary $m\in\N$ and domains with $C^\infty$-boundary, an asymptotic expansion
\begin{equation}
\tr_{L^2(\R^d)}[g(T_L)]=\sum_{k=0}^m L^{d-k}A_k(g;a) +o(L^{d-m}), \qquad \text{as} \ L\rightarrow\infty,
\end{equation}
with recursively defined coefficients $A_k(g;a)$ was obtained. More recently, domains with only piece-wise smooth boundary were considered. In \cite{Dietlein18} a complete $(d+1)$-term expansion of the form
\begin{equation}\label{expansion-Dietlein}
\tr_{L^2(\R^d)}[g(T_L)]=\sum_{k=0}^d L^{d-k}A_k(g;a) +o(L^{-\tau}), \qquad \text{as} \ L\rightarrow\infty,
\end{equation}
was established in the case that $\Lambda$ is a $d$-dimensional cube, where $\tau>0$ depends on the rate of decay of the operator kernel of $T_L$. This result holds in the more general setting of $\Z^d$-ergodic operators. The article \cite{Pfirsch} considers two-dimensional polygons and proves a complete three-term expansion of the form \eqref{expansion-Dietlein} for smooth, i.e. $C^\infty$, symbols and with $\tau>0$ arbitrarily large. An expansion similar to \eqref{expansion-Dietlein} has also been established for analytic test functions in the discrete, i.e. Toeplitz matrix, case in \cite{Thorsen}.

The case of a discontinuous symbol, traditionally refers to the symbol having a $(d-1)$-dimensional discontinuity in momentum space at the boundary of a sufficiently regular bounded domain $\Gamma\subset\R^d$. In this case a two-term asymptotic with a leading volume term and a subsequent logarithmically enhanced area term, i.e. of the form
\begin{equation}\label{expansion-enhanced}
\tr_{L^2(\R^d)}[g(T_L)]=L^{d}A_0(g;a)+L^{d-1}\log L \ W_1(g;a)+o(L^{d-1}\log L),\qquad \text{as} \ L\rightarrow\infty,
\end{equation}
is given by the Widom--Sobolev formula, proved by A.~Sobolev in the outstanding work \cite{sobolevlong} and extended in \cite{sobolevpw}. The expansion \eqref{expansion-enhanced}, together with the concrete expression of the coefficient $ W_1(h;a)$, was conjectured by H.~Widom in \cite{Widom1982}, where also the one-dimensional case was proved for both complex and matrix-valued symbols. The higher-dimensional Widom--Sobolev formula was extended to matrix valued symbols in \cite{BM-Widom}. For multidimensional discontinuous symbols no extensions of the Widom--Sobolev formula below the subleading term are known. The question whether it is possible to only obtain an error term of order $L^{d-1}$ in \eqref{expansion-enhanced} is also open.

The study of Szeg\H{o}-type asymptotics in the continuum is closely related to scaling laws for the bipartite entanglement entropy of systems of non-interacting fermions, as observed in \cite{GioevKlich}. The proof of the Widom--Sobolev formula enabled a proof of a logarithmically enhanced area law in the case that the single-particle Hamiltonian is given by the Laplace operator \cite{LeschkeSobolevSpitzerLaplace}. Generalisations from the free Fermi gas to positive temperatures and to electric and magnetic background fields were treated, e.g., in \cite{PhysRevLett.113.150404, ElgartPasturShcherbina2016, LeschkeSobolevSpitzerMultiScale, PfirschSobolev18, MuPaSc19, MuSc20, LeschkeSobolevSpitzerMagnetic, LeschkeSobolevSpitzerTemperature, 10.1063/5.0135006, PfSp22}.
Another natural way of extending the case of a free Fermi gas, is to consider a relativistic system, that is to replace the Laplace, i.e. free Schr\"odinger, operator by the free Dirac operator as the single-particle Hamiltonian. As the spectrum of the free Dirac operator is not bounded from below, a regularised version of the Fermi projection needs to be considered. For the free Dirac operator, enhanced area laws, with coefficients independent of this regularisation, were obtained in \cite{BM-inprep} with the help of the Widom--Sobolev formula for matrix-valued symbols \cite{BM-Widom}, while area laws in dimension three were obtained in \cite{FinsterSob}. The one-dimensional case without scaling is considered in \cite{SpitzerPfeifferFerro}. The study of the entanglement entropy also motived the extension of the expansions \eqref{expansion-non-enhanced} and \eqref{expansion-enhanced} to a more general class of non-smooth test function which include the entropy functions \cite{sobolevschatten, sobolevc2, Sobolev.2019.truncated}. For the expansions going beyond the subleading term, no such extensions to non-smooth test functions are known.

When considering the massless Dirac operator,
a special case occurs. The dispersion relation of the free massless Dirac operator is given by two cones joint at their apex in the origin. When the energy cut-off is chosen to be precisely at this point, the symbol of the corresponding Wiener--Hopf operator features a single zero-dimensional discontinuity at the origin. It was shown in \cite{BM-inprep} that, in the one-dimensional case, this situation can also be treated with the Widom--Sobolev formula and one obtains an enhanced area law of order $\log L$. In the higher-dimensional case at most an area law can be obtained. In \cite{FinsterSob} this area law is obtained via an extension of the result \cite{Widom1980} for matrix-valued symbols without discontinuity.

The goal of the present article is to further study this symbol with a zero-dimensional discontinuity in dimensions $d\geq 2$. As the leading and subleading terms are of the same order as the corresponding terms for continuous symbols, we study the expansion beyond the subleading term. To do so, it is convenient to consider this symbol in terms of the off-diagonal decay of order $|x-y|^{-d}$ of the corresponding integral kernel.
Comparing this decay with the decay required in, e.g., \cite{Widom1980}, suggests that not only the first two but even the first $d$ terms of the asymptotic expansion are of the same order as in the continuous case, and only starting with the $(d+1)$st term the order of the terms changes. Similarly, the zero-dimensional discontinuity suggests a lower-order logarithmically enhanced term of order $\log L$ as in the one-dimensional case \cite[Thm. 3.1]{BM-inprep}. Furthermore, in analogy to the interaction of the two $(d-1)$-dimensional boundaries $\partial\Lambda$ and $\partial\Gamma$ in the enhanced area coefficient of the Widom--Sobolev formula, one might suspect 
an interaction of the zero-dimensional discontinuity with zero-dimensional points in the boundary of $\Lambda$, where the smoothness of the boundary breaks down, e.g. the vertices of a $d$-dimensional cube. 
The occurrence of such a lower-order logarithmic term and its dependence on the corners of the spatial restriction are also discussed in the physics literature \cite{PhysRevLett.97.050404,Casini_2009,CASINI2009594,Kallin_2014,PhysRevB.94.125142,PhysDiracWeyl,Crepel.Hackenbroich.2021}.

Due to the aforementioned interest in the interaction with vertices, we are especially interested in studying domains $\Lambda$ with only piece-wise smooth boundary. We have chosen to restrict ourselves to $d$-dimensional cubes, as they still have a sufficiently simple structure in higher dimensions, which helps us illustrate the general structure of the asymptotic expansion depending on the dimension $d$. Furthermore, we can use the already established $(d+1)$-term expansions for cubes in the continuous case \cite{Dietlein18} as a starting point. We also restrict ourselves to analytic test functions $g$ with $g(0)=0$ as in most of the already obtained higher-order asymptotic expansions for continuous symbols. Although an extension to the entropy functions would be desirable, the strategy to deal with these more general functions, laid out in \cite{sobolevc2}, is based on utilising the quasi-commutator structure of the relevant operators which is incompatible with the reduction strategy we use to obtain the higher-order asymptotic expansion. We are unaware of any way to remedy this problem.
With these restrictions we expect an asymptotic expansion of the following form
\begin{equation}\label{expansion-paper}
\tr_{L^2(\R^d)\otimes\C^n}[g(T_L(D_b))]=\sum_{k=0}^{d-1} L^{d-k}A_k(g;D_b;b) +\log L\ W(g;D)+o(\log L), \qquad \text{as} \ L\rightarrow\infty,
\end{equation}
to hold.
Here, the operator $T_L(D_b)$ depends on the matrix-valued symbol $D_b:\R^d\rightarrow\C^{n\times n}$ which stems from the regularised Fermi projection of the Dirac operator and depends on the parameter $b\geq 0$ determining the regularisation. We note that the coefficient of the logarithmic term $W(g;D)$ is independent of the regularisation parameter $b$. Unfortunately, we are only able to prove the whole expansion \eqref{expansion-paper} for polynomial test functions of sufficiently low degree, for reasons explained in the next paragraph and in the beginning of Section \ref{sec_local_asymptotics}. For general analytic test functions we still manage to obtain the first $d$ terms of the expansion and prove an upper bound, logarithmic in $L$ and independent  of the regularisation parameter $b$, for the remaining term.

In order to prove the desired asymptotic expansion, we combine methods from the analysis of both continuous and discontinuous symbols. We derive the first $d$ terms of the expansion similarly as for continuous symbols. While the algebraic decomposition of the cube largely agrees with the one in \cite{Dietlein18} and uses some methods from \cite{Pfirsch}, the substantially worse decay of the kernel forces us to obtain the required estimates in a different way. Starting with monomials, the trace norm estimates are derived from estimates for a sum of products of appropriately chosen Hilbert-Schmidt norms, which stem from the decay estimate for the integral kernel. In order to still utilise the decay for the Hilbert-Schmidt norm estimates for higher degree monomials, we use a quite technical procedure of constructing neighbourhoods of certain regions in position space whose boundaries locally are graphs of suitably chosen power functions $f(x)=x^q$ with $0<q<1$. This procedure is best illustrated in the proof of Lemma \ref{Lemma-HS-decay} and appears several times in the proofs in Section \ref{sec_commutation}. After determining the first $d$ terms of the expansion, we treat the remaining operator in a similar way as in the case of a discontinuous symbol, cf. e.g. \cite{Widom1982, sobolevlong, sobolevpw}. We first prepare the operator by reducing it to a sum of vertex contributions through localisation and commute out the smooth regularisations of the Fermi projection. As we are only allowed to incur error terms of constant order while doing so, we need to make use of the structure of the remaining operator, leading to the long and technical treatment in Section \ref{sec_commutation}. After this preparation, we further follow the strategy for discontinuous symbols and provide a local asymptotic formula. The integral kernel of the remaining operator is a generalisation of the kernel in the one-dimensional case, where the occurring Hilbert transform is replaced by its higher dimensional analogues, the Riesz transforms.
Therefore, the idea is to extend the original one-dimensional proof in \cite{Widom1982} to higher dimensions. While doing so, we face the challenge that the one-dimensional proof relies on a connection between Hilbert and Mellin transform which allows one to find the spectral representation of the Wiener--Hopf operator in question. At this point, we need to restrict our chosen test functions, as we are not able to find a higher-dimensional analogue to this connection, due to the substantially more complex structure of the higher-dimensional operator and the unclear connection between the higher-dimensional transforms. Still, we are able to prove the required properties of the kernel by hand and a local asymptotic formula in the special case that the test function is given by a polynomial of degree three or less. For general analytic test functions we instead obtain a logarithmic upper bound for the remaining operator in a similar way to \cite[Lem. 3.8]{BM-inprep} respectively \cite[Lem, 4.3]{FinsterSob}.

The article is organised as follows: In Section \ref{sec_prelim_art3} we give an introduction to the free massless Dirac operator and state our main results. In Section \ref{sec_higher_order} we take a closer look at the decay of the integral kernel of the operator and use it to obtain the first $d$ terms of the asymptotic expansion. Section \ref{sec_upper_bound} begins with some additional estimates for pseudo-differential operators with smooth symbols. Utilising these estimates, we find a logarithmic upper bound for the remaining operator from Section \ref{sec_higher_order}. In Section \ref{sec_commutation} we take a closer look at the remaining term from Section \ref{sec_higher_order}. Using the structure of this operator and the estimates for smooth symbols from Section \ref{sec_upper_bound}, we commute all occurrences of the smooth regularisation of the Fermi projection to one side of the operator, up to error terms of constant order. The last Section \ref{sec_local_asymptotics} consists of two parts: In the first part we derive certain properties for the operator obtained through the commutation in Section \ref{sec_commutation}. We only do this for polynomial test functions of degree three or less. The second part consists of a local asymptotic formula under the assumptions that the properties derived in the first part hold. With this and the results from the previous section, we obtain the desired $(d+1)$-term asymptotic expansion.

\section{Preliminaries and main results}
\label{sec_prelim_art3}
As we are only interested in the higher-dimensional case, with the one-dimensional case being established in \cite{BM-inprep}, we restrict ourselves to the case $d\geq 2$ for the remaining part of the article.
\subsection{The free massless Dirac Operator}
\label{subsec_Dirac_massless}
We consider the self-adjoint Hamiltonian of the massless Dirac equation in 
$d\in\N\setminus\{1\}=\{2,3,\dots\}$ space dimensions. This is a special case of the  Dirac equation studied in \cite[Sec.~2.2]{BM-inprep} which offers more details. Here we focus on the aspects most important for the present massless case. The Hamiltonian is given by 
\begin{equation}
	\label{Dirac-op}
	\mathfrak{D}:= -\i\sum_{k=1}^d \alpha_k \partial_k,
\end{equation}
see e.g. \cite{10.1007/bf02754212, 10.1063/1.1367331}. It is densely defined in the product Hilbert space $L^{2}(\R^{d}) \otimes \C^{n}$ of square-integrable
vector-valued functions with spinor dimension
$n:=n_{d}:= 2^{\lfloor (d+1)/2\rfloor}$, where $\lfloor r\rfloor$ denotes the largest integer not 
exceeding $r\in\R$. As usual,
$\i$ is the imaginary unit and $\partial_{k}$, $k=1,\ldots,d$,
denotes partial differentiation with respect to the $k$th Cartesian coordinate. The matrices
$\alpha_{1},\ldots, \alpha_{d} \in\C^{n\times n}$ are Dirac matrices which anti-commute pairwise and whose square gives the $n\times n$-unit matrix $\mathbb{1}_{n}$. As shown by A.~Hurwitz \cite{Hurwitz}, the number $n=2^{\lfloor (d+1)/2\rfloor}$ is the smallest natural number such that $d+1$ such matrices exist (the $(d+1)$st matrix corresponds to the mass term which is absent here). The choice of Dirac matrices is then unique up to similarity transformations. As the trace is invariant under such transformations, the results of this article are, for the given choice of $n$, independent of the choice of Dirac matrices. 
The Dirac matrices can be chosen as (cf.\ \cite{Hurwitz} resp. \cite[Appendix]{10.1063/1.1367331})
\begin{align}\label{def_alpha_k}
\alpha_k:=\begin{pmatrix}
0 & \sigma_k \\
 \sigma_k^* & 0
\end{pmatrix}, \;\; k=1,\ldots,d, 
\end{align}
where the $\frac{n}{2}\times \frac{n}{2}$-matrices $\sigma_{k}$ satisfy the anti-commutation relations
\begin{align}
\label{sigma-anticomm}
\sigma_j \sigma_k^* + \sigma_k \sigma_j^* = 2\delta_{jk}\mathbb{1}_{\frac{n}{2}}
\;\; \text{and} \;\; \sigma_j^* \sigma_k + \sigma_k^* \sigma_j = 2\delta_{jk}\mathbb{1}_{\frac{n}{2}},
\end{align}
for all $j,k=1,\ldots,d$.
By $*$ we denote the adjoint and by $\delta_{jk}$ we denote the Kronecker delta. In particular, we immediately see that the Dirac matrices $\alpha_k$, $k\in\{1,\ldots,d\}$, have vanishing trace.

Via the Fourier transform $\mathcal{F}$ on $L^{2}(\R^{d})$, the Hamiltonian given in \eqref{Dirac-op} is unitarily equivalent
to the operator of multiplication on $L^{2}(\R^{d}) \otimes \C^{n}$ with the matrix-valued symbol 
\begin{align}
\R^{d} \ni \xi = (\xi_{1},\ldots,\xi_{d}) \mapsto D(\xi):= \sum_{k=1}^d \alpha_k \xi_k.
\end{align}
The symbol $D$ is smooth, i.e. we have $D \in C^{\infty}(\R^{d},\C^{n\times n})$.  

We consider the following smoothly truncated version of the Fermi projection at Fermi energy $E_F=0$.
Given an ultraviolet cut-off parameter $b\in [0,\infty[\,$, we define 
\begin{equation}\label{Definition-Chi-E_F}
\chi_{0}^{(b)}:\R\rightarrow [0,1],\quad x\mapsto \chi_{0}^{(b)}(x):=1_{\{y\in\R\,:\, y<0\}}(x)\varphi(x+b),
\end{equation} 
where $1_{\{y\in\R\,:\, y<0\}}$ denotes the corresponding indicator function and the monotone cut-off function $\varphi\in C^\infty(\R)$ obeys $\varphi|_{[0,\infty[}=1$ and $\varphi|_{]-\infty,-1]}=0$. We note that the function $\chi_{0}^{(b)}$ is bounded from above by one and compactly supported.
Applying this function to the Hamiltonian in \eqref{Dirac-op}, we obtain a bounded linear operator on $L^2(\R^d)\otimes \C^n$ which we write as 
\begin{equation}
	\Op\big(\chi_{0}^{(b)}(D)\big) = (\mathcal{F}^{-1} \otimes \mathbb{1}_{n}) \, \big(\chi_{0}^{(b)}(D)\big)(\pmb\cdot) \, (\mathcal{F} \otimes \mathbb{1}_{n}),
\end{equation} 
the operator with matrix-valued symbol $\chi_{0}^{(b)}(D)$. We note that the operator norm $\big\|\Op\big(\chi_{0}^{(b)}(D)\big)\big\|$ is also bounded from above by one. For an introduction on pseudo-differential operators with matrix-valued symbols in this context, we refer to \cite[Sec.~2.3]{BM-Widom}.

The application of measurable functions to the Hamiltonian of the Dirac equation has been studied in \cite{BM-inprep}, in the more general case with non-negative mass $m$ and arbitrary Fermi energy $E_F\in\R$. Here, we cite these results in the special case $E_F=m=0$, considered in the present article.
In this case, the energy, as a function on momentum space, is given by $E(\xi):=|\xi|$ and by \cite[(2.9)]{BM-inprep} the symbol of the operator $\chi_0^{(b)}(D)$ is given by
\begin{align}\label{Operator-m=0-A3}
\big(\chi_0^{(b)}(D)\big)(\xi)= \psi^{(b)}(\xi)\tfrac{1}{2}\big(\mathbb{1}_{n}-\tfrac{D(\xi)}{E(\xi)}\big)= \psi^{(b)}(\xi)\tfrac{1}{2}\Big(\mathbb{1}_{n}-\sum_{k=1}^d\alpha_k\tfrac{\xi_k}{|\xi|}\Big)=:\psi^{(b)}(\xi)\mathcal{D}(\xi), \quad \xi\in\R^d
\end{align}
where $\psi(\xi):=\psi^{(b)}(\xi):=\varphi(-E(\xi)+b)$ for every $\xi\in\R^d$ and we define $\tfrac{D}{E}(0):=0$. The symbol $\psi^{(b)}$ is smooth, i.e.~$\psi^{(b)}\in C^\infty(\R^d)$, and compactly supported in the ball $B_{b+1}(0)$ of radius $b+1$ centred about the origin, and satisfies $\|\psi^{(b)}\|_\infty=1$. The fact that the symbol $\mathcal{D}$ has a discontinuity at the origin, leads to significantly worse decay of the integral kernel $K$ of the corresponding operator. We take a closer look at this in Section \ref{subsec_kernel_decay}. This is the reason why we are no longer able to establish a full asymptotic expansion and instead obtain a logarithmically enhanced constant term. It is notationally convenient to drop the superscript $(b)$ of the symbol $\psi$, beginning with Section \ref{subsec_kernel_decay} of the present article.

\subsection{Main results}
\label{subsec_main_results_3}
In the present article, we focus on entire test functions $g:\C\rightarrow\C$ which satisfy $g(0)=0$, i.e. there exist $\omega_m\in\C$, $m\in\N$, such that
\begin{equation}
g(z)=\sum_{m\in\N}\omega_m z^m,
\end{equation}
for all $z\in\C$. The assumption that the functions are analytic on the whole complex plane is made to simplify notation in some of the results. In fact one could compute a finite radius $R>0$ such that analyticity in the disc $B_R(0)$ of radius $R$ centred about the origin would suffice.
 
Given such a test function $g$ the 
following trace is studied
\begin{equation}\label{trace-studied}
\tr_{L^{2}(\R^{d})\otimes\C^n}\big[g\big(\mathbf{1}_{\Lambda_{L}}\Op(\psi^{(b)}\mathcal{D})\mathbf{1}_{\Lambda_{L}}\big)\big],
\end{equation}
where $\Lambda:=[0,2]^d\subset \R^d$ is a cube of side-length $2$ and, given a scaling parameter $L> 0$, $\Lambda_L:=[0,2L]^d$ is a scaled version of this cube. Moreover, $\mathbf{1}_{\Omega}:=1_\Omega\otimes\mathbb{1}_n$ denotes the operator of multiplication on $L^2(\R^d)\otimes\C^n$ with the indicator function $1_{\Omega}$ of $\Omega\subseteq\R^d$. As $\Lambda$ is bounded and $\psi$ is compactly supported, the operator in \eqref{trace-studied} is trace class, according to \cite[Chap.\ 11, Sect.\ 8, Thm.\ 11]{BirmanSolomjak}.
For the trace \eqref{trace-studied}, we obtain a multi-term asymptotic expansion in the scaling parameter $L$. We start by defining the coefficients occurring in the higher-order terms of this expansion. Using the abbreviation $X_{m,k,b}:=\mathbf{1}_{\R_+^{m-k}\times\R^{d-m+k}}\Op(\psi^{(b)}\mathcal{D})\mathbf{1}_{\R_+^{m-k}\times\R^{d-m+k}}$ with $\R_+:=[0,\infty[\,$, we define
\begin{equation}
\label{definition-A-coeff}
A_{m,g,b}
:=\lim_{L\rightarrow\infty}\tr_{L^{2}(\R^{d})\otimes\C^n}\bigg[\sum_{k=0}^{m}c_{k,m}\mathbf{1}_{[0,L]^{m}\times [0,1]^{d-m}}g\big(X_{m,k,b}\big)\bigg],
\end{equation}
where the constants $c_{k,m}$ are given by $c_{k,m}:=\frac{(-1)^{k}2^{m}d!}{k!(m-k)!(d-m)!}$ for $0\leq k\leq m\leq d$. We will see in Section \ref{sec_higher_order} that these coefficients are well-defined. We note that the limits of the individual operators in the trace in \eqref{definition-A-coeff} are not trace class and therefore the limit can not be interchanged with the trace.

For entire test functions $g$, we obtain a $d$-term asymptotic expansion and show that the remaining term is of logarithmic order in the scaling parameter $L$.
\begin{thm}\label{Theorem-Asymptotics-Analytic}
Let $g:\C\rightarrow\C$ be an entire function satisfying $g(0)=0$. Let $b\in [0,\infty[\,$, then the following asymptotic formula holds 
\begin{equation}\label{Theorem-Asymptotics-Analytic-Statement}
\tr_{L^{2}(\R^{d})\otimes\C^n}\big[g\big(\mathbf{1}_{\Lambda_{L}}\Op(\psi^{(b)}\mathcal{D})\mathbf{1}_{\Lambda_{L}}\big)\big]
=\sum_{m=0}^{d-1}(2L)^{d-m} A_{m,g,b}
+O(\log L), 
\end{equation}
as $L\rightarrow\infty$. The coefficients $A_{m,g,b}$ are defined in \eqref{definition-A-coeff}. The implied constant in the error term of order $\log L$ is independent of the ultraviolet cut-off parameter $b$.
\end{thm}
\begin{rem}
\begin{enumerate}[(a)]
\item The coefficients $A_{m,g,b}$, $m\in\{0,\ldots,d-1\}$, of the $d$ leading terms of the asymptotic expansion correspond to the geometric structure of the cube, where the operators $X_{m,k,b}$, cf.~\eqref{definition-A-coeff}, are localisations to an approximate $(d-m+k)$-face of the cube $\Lambda$. The constant $c_{k,m}$ counts the number $2^m \frac{d!}{m!(d-m)!} $ of $(d-m)$-faces and multiplies it with the number $\frac{m!}{k!(m-k)!}$ of $(d-m+k)$-faces containing a given $(d-m)$-face. While the expression is quite abstract, we do not try to obtain more explicit expressions for the coefficients due to the following two reasons.
\item Due to the matrix nature of the symbol, it is in general not known how to find an explicit integral representation even for the coefficient of the area term, i.e.~$A_{1,g,b}$ in our case. This has been observed multiple times for smooth domains $\Lambda$ and general symbols \cite{Widom1980, Widom1985FullExpansion}, as well as in the special case of the free Dirac operator in dimension three \cite{FinsterSob}.
\item The coefficients $A_{m,g,b}$, $m\in\{0,\ldots,d-1\}$, of the $d$ leading terms of the asymptotic expansion, all depend on the chosen ultraviolet cut-off parameter $b$. In fact, in the case $m\in\{1,\ldots,d-1\}$ one would expect the coefficients to be of order $b^{d-m}$ in the cut-off parameter by a scaling argument, see also the main result of \cite{FinsterSob}. 
\end{enumerate}
\end{rem}
When the test function $g$ is a polynomial of degree less or equal than three, we are able to compute the asymptotic coefficient of the subsequent logarithmic term and obtain the following $(d+1)$-term asymptotic expansion.
\begin{thm}\label{Theorem-Asymptotics-Log}
Let $g$ be a polynomial of degree less or equal than three satisfying $g(0)=0$. Let $b\in [0,\infty[\,$, then the following asymptotic formula holds 
\begin{align}\label{Theorem-Asymptotics-Log-Statement}
\tr_{L^{2}(\R^{d})\otimes\C^n}\big[g\big(\mathbf{1}_{\Lambda_{L}}\Op(\psi^{(b)}\mathcal{D})\mathbf{1}_{\Lambda_{L}}\big)\big]&=\sum_{m=0}^{d-1}(2L)^{d-m}A_{m,g,b} \nonumber \\
& \quad +2^d \, \log L \,\int_{S^{d-1}_+} \tr_{\C^{n}}[K_g(y,y)]\dd y +O(1), 
\end{align}
as $L\rightarrow\infty$. Here, $S^{d-1}_+:=\{y\in \,]0,\infty[\,^d: \ |y|=1 \}\subset S^{d-1}$ is the part of the unit sphere $S^{d-1}$ which lies in the positive $d$-dimensional quadrant, and $K_g(y,y)$, $y\in S^{d-1}_+$ are the well-defined point-wise values of the continuous integral kernel $K_g$ of the operator $\Op(\mathcal{D})_{g}$, which is given by
\begin{equation}
\Op(\mathcal{D})_{g}:=\sum_{k=0}^{d}(-1)^{k}\sum_{\mathcal{M}\subseteq \{1,\ldots,d\}\ : \ |\mathcal{M}|=k}g(\mathbf{1}_{H_{\mathcal{M}}}\Op(\mathcal{D})\mathbf{1}_{H_{\mathcal{M}}}),
\end{equation}
with $H_{\mathcal{M}}:=\{x\in\R^d: \ \forall j\in \{1,\ldots,d\}\setminus \mathcal{M}: \ x_{j} \geq 0\}$.
\end{thm}
\begin{rem}
\begin{enumerate}[(a)]
\item
In contrast to the coefficients of the higher-order terms, the coefficient of the logarithmic term is independent of the cut-off parameter $b$. Therefore, we would indeed be interested to obtain a more explicit expression for it and maybe even relate it to the coefficients $\mathfrak{A}$ respectively $\mathfrak{U}$ which occur in the enhanced area laws obtained in \cite[Theorem 3.1]{BM-inprep}. The reason why we are not able to do so, is again the lack of a full generalisation of the procedure in Section \ref{sec_local_asymptotics}. The following two remarks point out similarities between these coefficients.
\item For polynomials of degree one, the coefficient of the logarithmic term vanishes. This property is shared by the coefficient $W$ in \cite[Theorem 3.1]{BM-inprep}.
\item For the second monomial $g_2$ and third monomial $g_3$, we have $\int_{S^{d-1}_+} \tr_{\C^{n}}[K_{g_3}(y,y)]\dd y=\frac{3}{2}\int_{S^{d-1}_+} \tr_{\C^{n}}[K_{g_2}(y,y)]\dd y$. This is the same ratio as for the coefficient $W$ in \cite[Theorem 3.1]{BM-inprep}, i.e.~in \cite[Theorem 3.1]{BM-inprep} we have $W(g_3,\Lambda,E_F,m)=\frac{3}{2}W(g_2,\Lambda,E_F,m)$, independently of $\Lambda$, $E_F$ and $m$. This is a consequence of the structure of the Dirac matrices, i.e.~their vanishing trace and the anti-commutation relations \eqref{sigma-anticomm}, cf.~Section \ref{subsec_proof_asym_log}. In fact, these properties allow even more insights on the coefficient for higher monomials: Under the assumption that Theorem \ref{Theorem-Asymptotics-Log} holds for all monomials of even degree $g_{2n}$, $n\in\{1,\dots,p\}$ for $p\in\N$, one can then deduce the theorem for the monomial $g_{2p+1}$. If the coefficient of the logarithmic term in \eqref{Theorem-Asymptotics-Log-Statement} also obeys the desired ratio for the monomials $g_{2n}$ and $g_{2n-1}$, $n\in\{2,\dots,p\}$, we would then obtain $\int_{S^{d-1}_+} \tr_{\C^{n}}[K_{g_{2p+1}}(y,y)]\dd y=\big(1+(2p\sum_{k=1}^{2p-1}1/k)^{-1}\big)\int_{S^{d-1}_+} \tr_{\C^{n}}[K_{g_{2p}}(y,y)]\dd y$. This is again the same ratio as for the coefficient $W$ in \cite[Theorem 3.1]{BM-inprep}. This fact could potentially be exploited in order to reduce the question of an explicit expression of the coefficient for the logarithmic term in \eqref{Theorem-Asymptotics-Log-Statement} to just calculating it for the monomial of degree two. However, it is, at the moment, unclear to the author how to derive Theorem \ref{Theorem-Asymptotics-Log} for monomials of degree $2n$, $n\in\N\setminus\{1\}$, and we therefore do not pursue this route further in this article.  
\item Theorems \ref{Theorem-Asymptotics-Analytic} and \ref{Theorem-Asymptotics-Log} can be generalised to symbols $B$ which feature not only a single point-discontinuity given by $\mathcal{D}$ but a finite number of such discontinuities in the compact support of the symbol $B$. The coefficient of the logarithmic term is then the sum of the coefficients corresponding to the points of discontinuity weighted by the values of $B$ at these points.
A proof of this fact is quite straightforward, as it just requires an additional localisation in momentum space, akin to the one used in the Widom--Sobolev formula, cf.~\cite{sobolevpw,BM-Widom}, during the commutation procedure in Section \ref{sec_commutation}. Still, we opt not to make this extension precise in this article, in order to maintain as much readability as possible in the already long Section \ref{sec_commutation}.
\item In order to obtain the first $d$ terms of the expansions \eqref{Theorem-Asymptotics-Analytic-Statement} and \eqref{Theorem-Asymptotics-Log-Statement}, we only make use of the fact that the operator $\Op(\psi^{(b)}\mathcal{D})$ has off-diagonal integral kernel decay of order $|x-y|^{-d}$ to obtain the required bounds and that the operator is translation-invariant with continuous integral kernel in order to evaluate the coefficients, cf.~Theorem \ref{theorem_higher_order_terms}. No other, more specific, properties of the Dirac operator are required.
\item The exact decay of order $|x-y|^{-d}$ required in Theorem \ref{theorem_higher_order_terms} is crucial, in the sense that a higher order decay of order $|x-y|^{-\alpha}$ with $\alpha>d$ would already yield a full $(d+1)$-term asymptotic expansion, with the $(d+1)$st term being of constant order.
\end{enumerate}
\end{rem}

\section{Localisation and higher-order terms}
\label{sec_higher_order}

\subsection{Estimates for the integral kernel decay}
\label{subsec_kernel_decay}
We now take a closer look at the integral kernel $K$ of the operator $\Op(\psi\mathcal{D})$. As a multidimensional analogue of the Hilbert transform, the functions $-\i\tfrac{\xi_k}{|\xi|}$ are the Fourier multipliers corresponding to the $k$th $d$-dimensional Riesz transform, see \cite[p.~57]{Stein} respectively \cite[p.~223f.]{SteinWeiss}. Hence, the kernel of the translation-invariant operator $\Op(\psi\mathcal{D})$ is given by
\begin{align}\label{intro_kernel}
K(x)=\frac{\mathbb{1}_{n}}{2}\check{\psi}(x)+\frac{c_d}{2\i}\sum_{k=1}^{d}\alpha_{k} \lim_{\epsilon\rightarrow 0}\int_{\R^d\setminus B_{\epsilon}(x)}\frac{(x_k-t_k)\check{\psi}(t)}{|x-t|^{d+1}}\dd t,
\end{align}
where $\check{\psi}$ denotes the inverse Fourier transform of the Schwartz function $\psi$ and the constant $c_d$ is given by $c_{d}:=\frac{\Gamma[(d+1)/2]}{\pi^{(d+1)/2}}$.
We want to study the decay of the (matrix) Hilbert-Schmidt norm of this kernel. To do so, let $R>0$ and $|x|> 2R$, then, for every $k\in\{1,\ldots,d\}$, integrating by parts yields the bound
\begin{align}
\lim_{\epsilon\rightarrow 0}&\int_{\R^d\setminus B_{\epsilon}(x)}\frac{(x_k-t_k)\check{\psi}(t)}{|x-t|^{d+1}}\dd t \nonumber\\ 
=&\int_{\R^d\setminus B_{R}(x)}\frac{(x_k-t_k)\check{\psi}(t)}{|x-t|^{d+1}}\dd t+\lim_{\epsilon\rightarrow 0}\int_{B_{R}(x)\setminus B_{\epsilon}(x)}\frac{(x_k-t_k)\check{\psi}(t)}{|x-t|^{d+1}}\dd t \nonumber\\ 
\leq& \ \frac{C_{1}}{R^{d}}+ \lim_{\epsilon\rightarrow 0}\bigg[\int_{\partial (B_{R}(x)\setminus B_{\epsilon}(x))}\frac{\check{\psi}(t)}{|x-t|^{d-1}}e_{k}\cdot \nu(t)\dd S(t)-\int_{B_{R}(x)\setminus B_{\epsilon}(x)}\frac{1}{|x-t|^{d-1}}\partial_k\check{\psi}(t)\dd t\bigg]\nonumber\\
\leq& \ \frac{C_{1}}{R^{d}} +C_{2}\sup_{|t|\geq R}\Big(|\check{\psi}(t)|+R\, \big|\partial_k\check{\psi}(t)\big|\Big)\leq \frac{C_3}{R^{d}},
\end{align}
with constants $C_1, C_2, C_3>0$ independent of $R$, where we used that $\check{\psi}$ is a Schwartz function.
Here, $\partial$ denotes the boundary of a set in $\R^d$, $\nu=\nu_{\partial(B_{R}(x)\setminus B_{\epsilon}(x))}$ denotes the vector field of exterior unit normals in $\R^{d}$ to the boundary of $B_{R}(x)\setminus B_{\epsilon}(x)$, and we write $\d S$ for integration with respect to the $(d-1)$-dimensional surface measure induced by Lebesgue measure in $\R^{d}$. Combining the estimates for every $k\in\{1,\ldots,d\}$, we find a constant $C_4>0$, independent of $x$, such that the (matrix) Hilbert-Schmidt norm
$\|K(x)\|_2$ is bounded from above by $\frac{C_4}{|x|^{d}}$. Being the inverse Fourier transform of the integrable function $\psi\mathcal{D}$, the kernel $K$ is continuous. Interpreting $K$ as a function of two variables $K:\R^d\times\R^d\rightarrow\C^{n\times n}$ with $K(x,y)=K(x-y)$ in terms of the single-variable function $K$ introduced in \eqref{intro_kernel}, the matrix Hilbert-Schmidt norm of its values is then bounded from above by
\begin{align}\label{kernel_bound}
\|K(x,y)\|_2\leq C_{K} \frac{1}{(1+|x-y|^2)^{\tfrac{d}{2}}}, \quad x,y\in\R^d,
\end{align}
for a constant $C_{K}>0$. 

The continuity of the integral kernel, together with the just obtained decay, ensures that the integral kernel of the operator $g\big(\mathbf{1}_{M}\Op(\psi\mathcal{D})\mathbf{1}_{M}\big)$ is also continuous on $M\times M$ for any open set $M\subseteq\R^d$ and entire test function $g$ with $g(0)=0$. As the proof of such a result is fairly standard, it is omitted.

In the remaining part of the present article we will often use the following way of obtaining a Hilbert-Schmidt estimate from the just obtained bound on the decay of the integral kernel.
\begin{lem}\label{Lemma_kernel_set}
Let $X$ be an integral operator on $L^2(\R^d)\otimes\C^n$, whose integral kernel satisfies inequality \eqref{kernel_bound}. Let $M,N\subset\R^d$ be (Borel) measurable and such that there exists a (Borel) measurable function $\varphi:\R^d\rightarrow [0,\infty [\,$ with $\dist(x,N)\geq \varphi(x)$ for every $x\in M$, as well as a constant $C_{\varphi}$ such that 
\begin{align}
\int_{M}\frac{1}{(1+\varphi(x)^2)^{\tfrac{d}{2}}} \dd x< C_{\varphi}.
\end{align}
Then the operator $\mathbf{1}_{M}X\mathbf{1}_{N}$ is Hilbert-Schmidt class with its Hilbert-Schmidt norm on $L^2(\R^d)\otimes\C^n$ satisfying the bound 
\begin{align}
\|\mathbf{1}_{M}X\mathbf{1}_{N}\|_2^2< C_{\varphi}C_{d}C_{K}^2,
\end{align}
where the constant $C_{K}$ is the constant in the bound \eqref{kernel_bound} and the constant $C_{d}$ only depends on the dimension $d$. Here, $\dist$ denotes the Euclidian distance in $\R^d$.
\end{lem}
\begin{proof}
It suffices to show that 
\begin{align}
\int_{M}\int_{N} \|K(x,y)\|_2^{2} \dd y \dd x < C_{\varphi}C_{d}C_{K}^2.
\end{align}
We fix $x\in M$ and estimate the integral over the set $N$ in the following way
\begin{align}
\int_{N} \|K(x,y)\|_2^{2} \dd y &\leq \int_{N} C_{K}^2 \frac{1}{(1+|x-y|^2)^{d}} \dd y  \leq C_{K}^2 \int_{\R^d\setminus B_{\varphi(x)}(x)} \frac{1}{(1+|x-y|^2)^{d}} \dd y \nonumber \\ &= C_{K}^2\int_{\R^d\setminus B_{\varphi(x)}(0)} \frac{1}{(1+|y|^2)^{d}} \dd y =C_{K}^2|S^{d-1}| \int_{\varphi(x)}^{\infty} \frac{r^{d-1}}{(1+r^2)^{d}} \dd r \nonumber \\ & \leq C_{K}^2|S^{d-1}| \int_{\varphi(x)}^{\infty} \frac{r}{(1+r^2)^{\tfrac{d}{2}+1}} \dd r = \frac{C_{K}^2|S^{d-1}|}{d} \frac{1}{(1+\varphi(x)^2)^{\tfrac{d}{2}}}.
\end{align}
Here, $|S^{d-1}|$ denotes the $(d-1)$-dimensional surface area of the unit sphere $S^{d-1}$ induced by Lebesgue measure. It follows that
\begin{align}
\int_{M}\int_{N} \|K(x,y)\|_2^{2} \dd y \dd x \leq \frac{C_{K}^2|S^{d-1}|}{d} \int_{M}\frac{1}{(1+\varphi(x)^2)^{\tfrac{d}{2}}}\dd x < C_{\varphi}C_{d}C_{K}^2,
\end{align}
with $C_{d}:=\frac{|S^{d-1}|}{d}$. 
\end{proof}

In Sections \ref{sec_commutation} and \ref{sec_local_asymptotics} we frequently deal with the operator $\Op(\mathcal{D})$ which no longer has an integrable symbol. The integral kernel of the operator $\Op(\mathcal{D})$ is singular at the origin and only exists as a tempered distribution. Nonetheless, away from the origin we can treat it as a function on $\R^d$ and it is clearly homogeneous of degree $-d$.
Hence, we obtain the estimate
\begin{align}\label{kernel_bound_distribution}
\|K(x)\|_2\leq C_{K} \frac{1}{|x|^{d}}, \qquad |x|>c>0.
\end{align}
In this situation, we also derive an estimate similar to the one in Lemma \ref{Lemma_kernel_set}.
\begin{lem}\label{Lemma_kernel_set_distribution}
Let $X$ be a translation-invariant integral operator, whose integral kernel satisfies inequality \eqref{kernel_bound_distribution}. Let $M,N\subset\R^d$ be measurable such that there exists a measurable function $\varphi:\R^d\rightarrow \,]c,\infty [\,$ with $\dist(x,N)\geq \varphi(x)$ for every $x\in M$ and a constant $C_{\varphi}$ such that 
\begin{align}
\int_{M}\frac{1}{\varphi(x)^{d}} \dd x< C_{\varphi}.
\end{align}
Then the operator $\mathbf{1}_{M}X\mathbf{1}_{N}$ is Hilbert-Schmidt class with its Hilbert-Schmidt norm on $L^2(\R^d)\otimes\C^n$ satisfying the bound
\begin{align}
\|\mathbf{1}_{M}X\mathbf{1}_{N}\|_2^2 < C_{\varphi}C_{d}C_{K}^2,
\end{align}
where the constant $C_{K}$ is the constant in the bound \eqref{kernel_bound_distribution} and the constant $C_{d}$ only depends on the dimension $d$. 
\end{lem}
As the proof of Lemma \ref{Lemma_kernel_set_distribution} mirrors the proof of Lemma \ref{Lemma_kernel_set}, it is omitted.

We now remark upon the following symmetry properties of the operator $\Op(\psi\mathcal{D})$ which will be useful in the following sections.
They rely on the fact that $\psi$ is spherically symmetric.
\begin{rem}\label{rem_symmetry}
\begin{enumerate}[(a)]
\item Symmetry of spatial directions: Given a permutation $\pi\in\mathcal{S}_d$, in the symmetric group of the set $\{1,\ldots,d\}$, and the corresponding unitary operator $U_\pi$ on $L^2(\R^d)\otimes\C^n$ with action $(U_\pi u)(x):=u(\pi(x)):=u(x_{\pi(1)},\dots, x_{\pi(d)})$, we have that
\begin{equation}
U_\pi\Op(\psi\mathcal{D})U_\pi^{-1}=\Op(\psi\mathcal{D}_\pi),
\end{equation}
where the symbol $\mathcal{D}_\pi=\tfrac{1}{2}\Big(\mathbb{1}_{n}-\sum_{k=1}^d\alpha_{\pi(k)}\tfrac{\xi_k}{|\xi|}\Big)$ agrees with the symbol $\mathcal{D}$ up to a relabelled choice of Dirac matrices.
Therefore, for a measurable set $\Omega\subset\R^d$ and the corresponding set $\Omega_\pi:=\{x\in\R^d:\ \pi(x)\in\Omega\}$, we have
\begin{equation}\label{consequence_trace_relabeling}
\tr_{L^{2}(\R^{d})\otimes\C^n}\big[\mathbf{1}_{\Omega}\Op(\psi\mathcal{D})\mathbf{1}_{\Omega}\big]=\tr_{L^{2}(\R^{d})\otimes\C^n}\big[\mathbf{1}_{\Omega_\pi}\Op(\psi\mathcal{D})\mathbf{1}_{\Omega_\pi}\big].
\end{equation}
\item Symmetry under reflections: Given $\sigma\in\{0,1\}^d$ and the corresponding unitary operator $U_\sigma$ on $L^2(\R^d)\otimes\C^n$ with action $(U_\sigma u)(x):=u(\sigma(x)):=u((-1)^{\sigma_1}x_{1},\dots, (-1)^{\sigma_d}x_{d})$, we have that
\begin{equation}
U_\sigma\Op(\psi\mathcal{D})U_\sigma^{-1}=\Op(\psi\mathcal{D}_\sigma),
\end{equation}
where the symbol $\mathcal{D}_\sigma=\tfrac{1}{2}\Big(\mathbb{1}_{n}-\sum_{k=1}^d(-1)^{\sigma_k}\alpha_{k}\tfrac{\xi_k}{|\xi|}\Big)$ agrees with the symbol $\mathcal{D}$ up to the signs of the chosen Dirac matrices. The corresponding consequence for the trace, as in \eqref{consequence_trace_relabeling}, follows.
\end{enumerate}
These properties can also be checked with the corresponding properties of the integral kernel, see \cite[p.~57]{Stein}. They agree with the symmetry properties $(\mathcal{A}_2)$ and $(\mathcal{A}_3)$ in \cite{Dietlein18}.
\end{rem}

\subsection{Higher-order terms}
\label{subsec_higher_order_thm__stat}
We now begin to study the higher-order terms of the expansion. The required analysis relies entirely on the fact that the integral kernel of the operator in question satisfies the bound \eqref{kernel_bound} and does not use any additional properties of the Dirac operator. Therefore, we study the trace
\begin{equation}
\tr_{L^{2}(\R^{d})\otimes\C^n}\big[g\big(\mathbf{1}_{\Lambda_{L}}X\mathbf{1}_{\Lambda_{L}}\big)\big],
\end{equation}
in the more general setting that $X$ is a bounded translation-invariant integral operator on $L^{2}(\R^{d})\otimes\C^n$, $n\in\N$, with continuous integral kernel $K:\R^d\times\R^d\rightarrow\C^{n\times n}$ satisfying the bound \eqref{kernel_bound}. We note that in this more general setting the matrix dimension $n\times n$ can be chosen independently of the spatial dimension $d$.

In this section we decompose the $d$-dimensional cube $\Lambda=[0,2]^d$ (and its scaled version $\Lambda_L=[0,2L]^d$) in a similar way as in \cite{Dietlein18}. To do so, we first fix some notation. We call $H\subset\R^d$ a half-space if there exists a $(d-1)$-dimensional (affine) subspace $A$ of $\R^d$, i.e. a (translated) hyperplane, dividing $\R^d$ into two parts and such that $H$ is one of these parts. We consider half-spaces to be closed, i.e. $A\subseteq H$.
For $k\in\{-1,\ldots,d\}$, we say that $F\subseteq\R^d$ is a $k$-face (of $\Lambda$) if $\dim(F)=k$ and there exists a half-space $H\subseteq \R^d$ with $\Lambda^{\circ}\cap \partial H=\emptyset$ such that $F=\Lambda \cap H$. Here, $^{\circ}$ denotes the interior of a set, $\dim(F)$ the dimension of the smallest affine subspace of $\R^d$ which contains $F$, $\emptyset$ the empty set and we define $\dim(\emptyset):=-1$. We note that the choice of half-space $H$ is only unique in the case that $\dim(F)=d-1$. 

For each $k\in\{-1,\ldots,d\}$, we denote by $\mathcal{F}^{(k)}$ the set of all $k$-faces of $\Lambda$. For a given $j$-face $F$, we define the sets $\mathcal{F}_{F}^{(k)}:=\{G\in\mathcal{F}^{(k)} : F\subseteq G\}$ of all $k$-faces containing $F$ for $k\in\{j,\ldots,d\}$. We now associate to each $k$-face $F\in\mathcal{F}^{(k)}$ a unique set $H_F$. This set is only a half-space in the case that $k=d-1$.
To the only $d$-face $F=\Lambda$ we associate the set $H_F=\R^d$. To each $(d-1)$-face $F\in\mathcal{F}^{(d-1)}$, we associate the unique half-space $H_{F}\subseteq\R^{d}$ such that $F\subset \partial H_{F}$ and $\Lambda\subset H_F$. For $k\in\{0,\ldots,d-2\}$, we associate to a given $k$-face $F\in\mathcal{F}^{(k)}$ the set $H_{F}:=\bigcap_{G\in\mathcal{F}_{F}^{(d-1)}}H_{G}$. 

For a given $k$-Face $F$, $k\in\{1,\ldots,d-1\}$, the smallest affine subspace $W_F'$ of $\R^d$ with $F\subset W_F'$ is given by $W_F'=\bigcap_{G\in\mathcal{F}_{F}^{(d-1)}}\partial H_{G}$. We then find $u_F\in\R^d$ and a $k$-dimensional subspace $W_F$ of $\R^d$ such that $W_F'=u_F+W_F$. Writing $\R^d\ni x=x_1+x_2\in (u_F+W_F)\oplus W_F^{\intercal}=:W_F'\oplus W_F^{\intercal}$, we define the semi-finite space $H_{F,\infty}:=\{x_1+x_2\in H_F\subseteq W_F'\oplus W_F^{\intercal}:x_1\in F\}$ as well as its finite scaled version $H_{F,L}:=\{x_1+x_2\in H_F\subseteq W_F'\oplus W_F^{\intercal}:x_1\in F, \ \|x_2\|_\infty\leq L\}$, where $\|\cdot\|_\infty$ denotes the maximum norm in $W_F^{\intercal}$.

We start by decomposing the scaled cube $\Lambda_{L}$ into $2^{d}$ parts; each associated to one of the vertices (or $0$-faces) of $\Lambda$. Let $V\in\mathcal{F}^{(0)}$ be any of these vertices. We begin by considering the finite scaled version of the space $H_{V}$ associated to $V$, which we define as the cube $H_{V,L}:=H_{LV}\cap \big(LV+[-L,L]^{d}\big)$ containing the vertex $LV$ of $\Lambda_L$. As $g(0)=0$, we obtain
\begin{equation}\label{decomp_corners}
g\big(\mathbf{1}_{\Lambda_L}X\mathbf{1}_{\Lambda_L}\big)=\mathbf{1}_{\Lambda_L}g\big(\mathbf{1}_{\Lambda_L}X\mathbf{1}_{\Lambda_L}\big)=\sum_{V\in\mathcal{F}^{(0)}}\mathbf{1}_{H_{V,L}}g\big(\mathbf{1}_{\Lambda_L}X\mathbf{1}_{\Lambda_L}\big).
\end{equation}
Let $m\in\{0,\ldots,d\}$. For a given $(d-m)$-face $F$ (of $\Lambda$ respectively $\Lambda_L$), we define the operator
\begin{equation}\label{def-face-operator}
X_{F,g}:=\sum_{k=0}^{m}(-1)^{k}\sum_{G\in\mathcal{F}_{F}^{(k+(d-m))}}g(\mathbf{1}_{H_G}X\mathbf{1}_{H_G})
\end{equation}
and note that, for a given vertex $V$ (of $\Lambda$ respectively $\Lambda_L$), we have 
\begin{equation}
\sum_{m=0}^{d}\sum_{F\in\mathcal{F}_{V}^{(d-m)}}X_{F,g}=\sum_{m=0}^{d}\sum_{F\in\mathcal{F}_{V}^{(d-m)}}\sum_{k=0}^{m}(-1)^{k}\sum_{G\in\mathcal{F}_{F}^{(k+(d-m))}}g(\mathbf{1}_{H_G}X\mathbf{1}_{H_G})=g(\mathbf{1}_{H_{V}}X\mathbf{1}_{H_{V}}),
\end{equation}
as all other terms cancel out. Using this, we rewrite the right-hand side of \eqref{decomp_corners} as
\begin{equation}\label{decomp_corners_2}
\sum_{V\in\mathcal{F}^{(0)}}\mathbf{1}_{H_{V,L}}\bigg[g\big(\mathbf{1}_{\Lambda_L}X\mathbf{1}_{\Lambda_L}\big)-g(\mathbf{1}_{H_{LV}}X\mathbf{1}_{H_{LV}})+\sum_{m=0}^{d}\sum_{F\in\mathcal{F}_{V}^{(d-m)}}X_{LF,g}\bigg].
\end{equation}
The first step will be to localise the operator up to a term of constant order, i.e. to show that for each vertex $V$, replacing $\mathbf{1}_{H_{V,L}}g\big(\mathbf{1}_{\Lambda_L}X\mathbf{1}_{\Lambda_L}\big)$ with $\mathbf{1}_{H_{V,L}}g(\mathbf{1}_{H_{LV}}X\mathbf{1}_{H_{LV}})$ in the trace of \eqref{decomp_corners_2} only yields an error term of constant order in $L$.
We will do this in Section \ref{subsec_loc}. After the localisation we obtain 
\begin{equation}\label{decomp_corners_result}
\tr_{L^{2}(\R^{d})\otimes\C^n}\Big[g\big(\mathbf{1}_{\Lambda_L}X\mathbf{1}_{\Lambda_L}\big)\Big]=\tr_{L^{2}(\R^{d})\otimes\C^n}\bigg[\sum_{V\in\mathcal{F}^{(0)}}\mathbf{1}_{H_{V,L}}\bigg(X_{LV,g}+\sum_{m=0}^{d-1}\sum_{F\in\mathcal{F}_{V}^{(d-m)}}X_{LF,g}\bigg)\bigg],
\end{equation} 
up to an error term of constant order.
As we will show in Section \ref{sec_upper_bound}, the contributions to the higher-order terms stem exclusively from the sum
\begin{equation}\label{decomp_corners_higher_order}
\sum_{V\in\mathcal{F}^{(0)}}\mathbf{1}_{H_{V,L}}\sum_{m=0}^{d-1}\sum_{F\in\mathcal{F}_{V}^{(d-m)}}X_{LF,g},
\end{equation}
while the terms containing the operators $X_{LV,g}$ are of logarithmic order in the scaling parameter $L$.
The largest part of the present section is devoted to determining these higher-order terms, starting from the trace of \eqref{decomp_corners_higher_order}. The results are summarised in the following

\begin{thm}\label{theorem_higher_order_terms}
Let $X$ be a bounded translation-invariant integral operator on $L^{2}(\R^{d})\otimes\C^n$ with continuous integral kernel $K:\R^d\times\R^d\rightarrow\C^n$ satisfying the bound \eqref{kernel_bound}. Let $g:\C\rightarrow\C$ be an entire function with $g(0)=0$. Further let $\Lambda=[0,2]^{d}\subset\R^{d}$. Then
\begin{multline}\label{theorem_higher_order_terms_general}
\tr_{L^{2}(\R^{d})\otimes\C^n}\bigg[g\big(\mathbf{1}_{\Lambda_{L}}X\mathbf{1}_{\Lambda_{L}}\big)-\sum_{V\in\mathcal{F}^{(0)}}\mathbf{1}_{H_{V,L}}X_{LV,g}\bigg]\\
=\sum_{m=0}^{d-1}L^{d-m} \lim_{L\rightarrow\infty}\sum_{F\in\mathcal{F}^{(d-m)}}\tr_{L^{2}(\R^{d})\otimes\C^n}\big[\mathbf{1}_{H_{F,L}}X_{F,g}\big] 
+O(1),
\end{multline} 
as $L\rightarrow\infty$.
If the operator $X$ additionally fulfils the symmetry properties in Remark \ref{rem_symmetry}, this reduces to
\begin{align}
\label{theorem_higher_order_terms_special}
&\tr_{L^{2}(\R^{d})\otimes\C^n}\bigg[g\big(\mathbf{1}_{\Lambda_{L}}X\mathbf{1}_{\Lambda_{L}}\big)-\sum_{k=0}^{d}c_{k,d}\mathbf{1}_{[0,L]^{d}}g\big(\mathbf{1}_{\R_+^{d-k}\times\R^{k}}X\mathbf{1}_{\R_+^{d-k}\times\R^{k}}\big) \bigg] \nonumber\\
  &\qquad =\sum_{m=0}^{d-1}(2L)^{d-m}\lim_{L\rightarrow\infty}\tr_{L^{2}(\R^{d})\otimes\C^n}\bigg[\sum_{k=0}^{m}c_{k,m}\mathbf{1}_{[0,L]^{m}\times [0,1]^{d-m}}g\big(\mathbf{1}_{\R_+^{m-k}\times\R^{d-m+k}}X\mathbf{1}_{\R_+^{m-k}\times\R^{d-m+k}}\big)\bigg]  \nonumber\\
& \qquad\qquad\ \ +O(1),
\end{align}
as $L\rightarrow\infty$, with the constants $c_{k,m}:=\frac{(-1)^{k}2^{m}d!}{k!(m-k)!(d-m)!}$ for $0\leq k\leq m\leq d$. The limits of the operators in the traces on the right-hand side of both \eqref{theorem_higher_order_terms_general} and \eqref{theorem_higher_order_terms_special} are only trace class in the case $m=0$. Therefore, the limit can not be exchanged with the trace.
\end{thm}
The proof of this theorem is divided into several steps. As mentioned above, we begin by localising the remaining operator in Section \ref{subsec_loc}. Then we establish several required estimates for the operator $X$ in Section \ref{subsec_HS_higher_order}. This will be the most challenging part of the proof. Lastly, we finish the proof of Theorem \ref{theorem_higher_order_terms} in Section \ref{subsec_higher_order_proof}.

Throughout the proof of Theorem \ref{theorem_higher_order_terms} we will often choose a designated vertex $V$ and face $F$ in the following way:
Up to suitable rotation and translation, we can assume that the cube is still given by $\Lambda=[0,2]^{d}$, the vertex is given by $V=V_{0}=\{0\}\in\R^d$ and the $(d-m)$-face is given by $F=\{0\}^{m}\times [0,2]^{d-m}$. This will considerably simplify the notation in the following analysis. The sets, already defined in this section, are then given by
\begin{align}
H_{LV}=H_{V}=\R_+^{d}, \quad H_{V,L}=[0,L]^{d},  \quad H_{LF}=H_{F}=\R_+^{m}\times \R^{d-m},  \quad H_{F,L}=[0,L]^{m}\times [0,2]^{d-m}.
\end{align}

\subsection{Localisation in position space}
\label{subsec_loc}
We now study the absolute value of the trace of the operator
\begin{equation}
\sum_{V\in\mathcal{F}^{(0)}}\mathbf{1}_{H_{V,L}}\big[g(\mathbf{1}_{H_{LV}}X\mathbf{1}_{H_{LV}})-g\big(\mathbf{1}_{\Lambda_L}X\mathbf{1}_{\Lambda_L}\big)\big]
\end{equation}
and show that it is bounded from above by a constant independently of the scaling parameter $L$. By the triangle inequality, it suffices to show this for the designated vertex $V=V_0$ chosen above, i.e. for the operator
\begin{equation}\label{loc_op}
\mathbf{1}_{[0,L]^{d}}\Big(g\big(\mathbf{1}_{\R_+^{d}}\mathbf{1}_{\R_+^{d}}\big)-g\big(\mathbf{1}_{\Lambda_L}X\mathbf{1}_{\Lambda_L}\big)\Big).
\end{equation}
The proof is given in the following two lemmas. The first one gives a Hilbert-Schmidt estimate, in a more general setting for monomial test functions, and the second one applies this to the operator in question and yields an estimate for the absolute value of the trace of \eqref{loc_op}.

\begin{lem}\label{Lemma-HS-distance-L}
Let $L\geq 1$. Let $p\in\N$, $\Omega_1,\ldots,\Omega
_p\subseteq\R^d$ be measurable and $X_1,\ldots, X_p$ be bounded translation-invariant integral operators on $L^{2}(\R^{d})\otimes\C^n$, each satisfying estimate \eqref{kernel_bound_distribution}. Further let $M,N\subset\R^{d}$ be measurable and such that there exist constants $\rho,R>0$, independent of $L$, with $M\subset B_{L\rho}(0)$ and $\dist(M,N)>RL$. Then there exists a constant $C>0$, independent of $L$ and $p$, such that
\begin{equation}
\big\|\mathbf{1}_{M}\Big(\prod_{k=1}^{p}\mathbf{1}_{\Omega_{k}}X_{k}\Big)\mathbf{1}_{N}\big\|_2\leq C p^{\tfrac{d+2}{2}}\Big(\max_{1\leq k\leq p}\|X_k\|\Big)^{p-1}.
\end{equation}
Here, $\|\cdot\|$ denotes the operator norm on the space of bounded linear operators on $L^{2}(\R^{d})\otimes\C^n$.
\end{lem}
\begin{proof}
We start with the case $p=1$. The sets $M\cap \Omega_{1}$ and $N$ satisfy the requirements of Lemma \ref{Lemma_kernel_set_distribution} with $\varphi(x)=RL$. As 
\begin{equation}
\int_{M\cap \Omega_{1}}\frac{1}{\varphi(x)^{d}} \dd x\leq \int_{B_{L\rho}(0)}\frac{1}{(RL)^d}\dd x = |B_1(0)|\frac{\rho^{d}}{R^d},
\end{equation}
Lemma \ref{Lemma_kernel_set_distribution} yields a constant $C>0$, independent of $L$, such that
\begin{equation}
\big\|\mathbf{1}_{M}\mathbf{1}_{\Omega_{1}}X_{1}\mathbf{1}_{N}\big\|_2<C.
\end{equation} 
In order to prove the case $p\geq 2$, we construct $p-1$ sets "in between" the sets $M$ and $N$. Namely, we define the sets $M_k:= \{x\in \Omega_k : \dist(x,M)\leq \tfrac{(k-1)R}{p}L\}$, for $k\in\{1,\ldots,p\}$. We note that $M_1=M$. Clearly, we have $\dist\big(M_k,\Omega_{k+1}\setminus M_{k+1}\big)> \tfrac{R}{p}L$, for all $k\in\{1,\ldots,p-1\}$. Defining $\varphi(x):=\tfrac{R}{p}L$, we calculate
\begin{equation}
\int_{M_k}\frac{1}{\varphi(x)^{d}} \dd x\leq \int_{B_{L(\rho+R)}(0)}\frac{p^{d}}{(RL)^d}= p^{d}|B_1(0)|\frac{(\rho+R)^{d}}{R^d},
\end{equation}
for $k\in\{1,\ldots,p-1\}$, and Lemma \ref{Lemma_kernel_set_distribution} yields a constant $C_{k}>0$, independent of $L$ and $p$, such that
\begin{equation}\label{Lemma-HS-distance-L-estimate-1}
\big\|\mathbf{1}_{M_k}X_{k}\mathbf{1}_{\Omega_{k+1}\setminus M_{k+1}}\big\|_2\leq C_{k}\sqrt{p^{d}}.
\end{equation} 
As we also have $\dist\big(M_{p},N)> \tfrac{R}{p}L$, Lemma \ref{Lemma_kernel_set_distribution} yields a constant $C_{p}>0$, independent of $L$ and $p$, with
\begin{equation}\label{Lemma-HS-distance-L-estimate-2}
\big\|\mathbf{1}_{M_p}X_{p}\mathbf{1}_{N}\big\|_2\leq C_{p}\sqrt{p^{d}},
\end{equation}
in the same way. Repeated use of the triangle inequality, followed by H\"older's inequality and application of the bounds \eqref{Lemma-HS-distance-L-estimate-1} and \eqref{Lemma-HS-distance-L-estimate-2}, yields
\begin{align}
\big\|\mathbf{1}_{M}\Big(\prod_{k=1}^{p}\mathbf{1}_{\Omega_{k}}X_{k}\Big)&\mathbf{1}_{N}\big\|_2 \nonumber \\ 
\leq& \sum_{k=1}^{p-1}\Big\|\Big(\prod_{j=1}^{k}\mathbf{1}_{M_{j}}X_{j}\Big)\mathbf{1}_{\Omega_{k+1}\setminus M_{k+1}}X_{k+1}\Big(\prod_{j=k+2}^{p}\mathbf{1}_{\Omega_{j}}X_{j}\Big)\mathbf{1}_{N}\Big\|_2
+ \Big\|\Big(\prod_{j=1}^{p}\mathbf{1}_{M_{j}}X_{j}\Big)\mathbf{1}_{N}\Big\|_2 \nonumber \\
\leq&  \sum_{k=1}^{p-1}\big\|\mathbf{1}_{M_k}X_{k}\mathbf{1}_{\Omega_{k+1}\setminus M_{k+1}}\big\|_2\Big(\max_{1\leq k\leq p}\|X_k\|\Big)^{p-1}+\big\|\mathbf{1}_{M_p}X_{p}\mathbf{1}_{N}\big\|_2 \Big(\max_{1\leq k\leq p}\|X_k\|\Big)^{p-1} \nonumber\\
\leq& \ C p^{\tfrac{d+2}{2}}\Big(\max_{1\leq k\leq p}\|X_k\|\Big)^{p-1},
\end{align}
with a constant $C>0$, independent of $L$ and $p$.
This proves the lemma for $p\geq 2$.
\end{proof}

\begin{lem}\label{Lemma-localisation}
Let $L\geq 1$, $g:\C\rightarrow\C$ be an entire function with $g(0)=0$ and $X$ be a bounded translation-invariant integral operator on $L^{2}(\R^{d})\otimes\C^n$ whose integral kernel satisfies estimate \eqref{kernel_bound_distribution}. Then there exists a constant $C>0$, independent of $L$, such that
\begin{equation}\label{Lemma-localisation-estimate}
\Big|\tr_{L^{2}(\R^{d})\otimes\C^n}\Big[\mathbf{1}_{[0,L]^{d}}\Big(g\big(\mathbf{1}_{\R_+^{d}}X\mathbf{1}_{\R_+^{d}}\big)-g\big(\mathbf{1}_{\Lambda_L}X\mathbf{1}_{\Lambda_L}\big)\Big)\Big]\Big|\leq C.
\end{equation}
\end{lem}
\begin{proof}
We start with a proof for monomial test functions $g$. Let $p\in\N$ and set $P_{0}:=\mathbf{1}_{\Lambda_L}$ as well as $P_{1}:=\mathbf{1}_{\R_+^{d}}$. Using the fact that $\mathbf{1}_{[0,L]^{d}}P_{0}=\mathbf{1}_{[0,L]^{d}}P_{1}=\mathbf{1}_{[0,L]^{d}}$ and the cyclic property of the trace, we rewrite the operator in \eqref{Lemma-localisation-estimate} as
\begin{equation}\label{Lemma-localisation-proof-step-1}
\mathbf{1}_{[0,L]^{d}}\big[\big(XP_{0}\big)^{p-1}-\big(XP_{1}\big)^{p-1}\big]X\mathbf{1}_{[0,L]^{d}}
\end{equation}
and see that in the case $p=1$ there is nothing to show. If $p\geq 2$, we define $P_{2}=P_{1}-P_{0}=\mathbf{1}_{\R_+^{d}\setminus \, [0,2L]^{d}}$ and write
\begin{align}
\mathbf{1}_{[0,L]^{d}}\big(&XP_{1}\big)^{p-1}X\mathbf{1}_{[0,L]^{d}}\nonumber\\
&=\mathbf{1}_{[0,L]^{d}}\big(XP_{0}+XP_{2}\big)^{p-1}X\mathbf{1}_{[0,L]^{d}}\nonumber\\
&=\sum_{\pi=(\pi_{1},\ldots,\pi_{p-1})\in \{0,2\}^{p-1}}\mathbf{1}_{[0,L]^{d}}\Big(\prod_{j=1}^{p-1}XP_{\pi(j)}\Big)X\mathbf{1}_{[0,L]^{d}}.
\end{align}
With this at hand, we estimate the trace norm of \eqref{Lemma-localisation-proof-step-1} by 
\begin{equation}\label{Lemma-localisation-proof-step-2}
\sum_{\pi\in \{0,2\}^{p-1}: \ \pi\neq 0} \Big\|\mathbf{1}_{[0,L]^{d}}\Big(\prod_{j=1}^{p-1}XP_{\pi(j)}\Big)X\mathbf{1}_{[0,L]^{d}}\Big\|_1.
\end{equation}
Each of these trace norms contains at least one factor $P_{2}$. Setting $p_{1,\pi}:=\min \{j\in\{1,\ldots,p-1\}: \ \pi(j)=2\}$ and $p_{2,\pi}:=\max \{j\in\{1,\ldots,p-1\}: \ \pi(j)=2\}$ for a given $\pi\in\{0,2\}^{p-1}$ with $\pi\neq 0$, we estimate \eqref{Lemma-localisation-proof-step-2} by
\begin{equation}\label{Lemma-localisation-proof-step-3}
\|X\|^{p-1-p_{1,\pi}-p_{2,\pi}}\sum_{\pi\in \{0,2\}^{p-1}: \ \pi\neq 0} \Big\|\mathbf{1}_{[0,L]^{d}}\big(P_{0}X\big)^{p_{1,\pi}}P_{2}\Big\|_2\Big\|P_{2}\big(XP_{0}\big)^{p_{2,\pi}}\mathbf{1}_{[0,L]^{d}}\Big\|_2.
\end{equation}
By Lemma \ref{Lemma-HS-distance-L}, there exists a constant $C_1>0$, independent of $L$ and $p$, such that \eqref{Lemma-localisation-proof-step-3} is bounded by
\begin{equation}\label{est_loc_poly}
C_1\sum_{\pi\in \{0,2\}^{p-1}: \ \pi\neq 0}(p_{1,\pi})^{\tfrac{d+2}{2}}(p_{2,\pi})^{\tfrac{d+2}{2}}\|X\|^{p-1}\leq C_1(p-1)^{d+2}(2\|X\|)^{p-1}.
\end{equation}
This concludes the proof of the Lemma for monomials. It remains to extend it to entire functions $g$. With estimate \eqref{est_loc_poly} we write
\begin{multline}
\Big|\tr_{L^{2}(\R^{d})\otimes\C^n}\Big[\mathbf{1}_{[0,L]^{d}}\Big(g\big(\mathbf{1}_{\R_+^{d}}X\mathbf{1}_{\R_+^{d}}\big)-g\big(\mathbf{1}_{\Lambda_L}X\mathbf{1}_{\Lambda_L}\big)\Big)\Big]\Big|  \\
= \Big|\sum_{p=1}^{\infty}\tr_{L^{2}(\R^{d})\otimes\C^n}\Big[\omega_{p} \mathbf{1}_{[0,L]^{d}}\Big(\big(\mathbf{1}_{\R_+^{d}}X\mathbf{1}_{\R_+^{d}}\big)^p-\big(\mathbf{1}_{\Lambda_L}X\mathbf{1}_{\Lambda_L}\big)^p\Big)\Big]\Big| \\
\leq  C_{2}\sum_{p=1}^{\infty}|\omega_{p}|(p-1)^{d+2}(2\|X\|)^{p-1}\leq C,
\end{multline}
with constants $C, C_2>0$ independent of $L$. This concludes the proof of the Lemma.
\end{proof}

\subsection{Hilbert-Schmidt and trace norm estimates}
\label{subsec_HS_higher_order}
We continue with the proof of Theorem \ref{theorem_higher_order_terms}. We start by discussing the further strategy of the proof and by introducing some additional notation, before we prove the required bounds.

For each $k$-face $F$, we divide the sets $H_{F,\infty}$ and $H_{F,L}$ into $2^{k}$ parts, each associated to one of the $2^{k}$ vertices $V\in F$, given by $H_{F,L,V}:=\{x_1+x_2\in H_F\subseteq W_F'\oplus W_F^{\intercal}:x_1\in F\cap(V+[-1,1]^{d}),\ \|x_2\|_\infty\leq L\}$ for $L\in [1,\infty]$. Clearly, we have $H_{F,L}=\sum_{\{V\in\mathcal{F}^{(0)}:V\in F\}}H_{F,L,V}$. With this, \eqref{decomp_corners_result} and the localisation in Section \ref{subsec_loc}, we see that the proof of \eqref{theorem_higher_order_terms_general} reduces to showing that
\begin{multline}\label{theorem_higher_order_terms_goal_1}
\sum_{m=0}^{d-1}\sum_{V\in\mathcal{F}^{(0)}}\sum_{F\in\mathcal{F}_{V}^{(d-m)}}\tr_{L^{2}(\R^{d})\otimes\C^n}\big[\mathbf{1}_{H_{V,L}}X_{LF,g}\big]\\
=\sum_{m=0}^{d-1}L^{d-m}\lim_{L\rightarrow\infty}\sum_{V\in\mathcal{F}^{(0)}}\sum_{F\in\mathcal{F}_{V}^{(d-m)}}\tr_{L^{2}(\R^{d})\otimes\C^n}\big[\mathbf{1}_{H_{F,L,V}}X_{F,g}\big]+O(1),
\end{multline}
as $L\rightarrow\infty$. Therefore, for a given vertex $V\in\mathcal{F}^{(0)}$ and a given $(d-m)$-face $F$ we need to show that the limit $L^{d-m}\lim_{L\rightarrow\infty}\tr_{L^{2}(\R^{d})\otimes\C^n}\big[\mathbf{1}_{H_{F,L,V}}X_{F,g}\big]$ exists and agrees with the trace $\tr_{L^{2}(\R^{d})\otimes\C^n}\big[\mathbf{1}_{H_{V,L}}X_{LF,g}\big]$ up to an error term of constant order.

It will again be convenient to only deal with the case, where the vertex is given by $V=\{0\}\in\R^d$ and the $(d-m)$-face is given by $F=\{0\}^m\times [0,2]^{d-m}$. The other cases reduce to this case after suitable rotation and translation. We recall that in this case we have
\begin{align}
H_{LV}=H_{V}=\R_+^{d}, \quad H_{V,L}=[0,L]^{d},  \quad H_{LF}=H_{F}=\R_+^{m}\times \R^{d-m}
\end{align}
and note that also $X_{LF,g}=X_{F,g}$ in this case.
We further note that in this case we also have $H_{F,L,V}=[0,L]^{m}\times[0,1]^{d-m}$ and
$H_{F,\infty,V}=\R_+^{m}\times[0,1]^{d-m}$ for the recently defined sets, and we define a scaled version of the latter set
by $H_{F,\infty,V,L}:=(H_{F,\infty,V})_{L}=\R_+^{m}\times[0,L]^{d-m}$.

For the remaining part of the present section, we only deal with the case $m\in\{1,\ldots,d-1\}$ and exclude the easier case $m=0$. In order to understand the structure of the operator $X_{F,g}$, it will be convenient to relabel the occurring projections.
There is a one to one correspondence between the faces $G\in\mathcal{F}_{F}^{(k)}$, with $k\geq d-m$, and the sets $\mathcal{M}_G\subseteq\{1,\ldots,m\}$ with $|\mathcal{M}_G|=k-(d-m)$. For a given set $\mathcal{M}\subseteq\{1,\ldots,m\}$, we write $H_{\mathcal{M}}:=H_{\mathcal{M}_G}:=\{x\in\R^d: \ \forall j\in \{1,\ldots,m\}\setminus \mathcal{M}_G: \ x_{j}\geq 0\}=H_{G}$. We note that the set $H_{\mathcal{M}}$ implicitly depends on the face $F$. Still, it is convenient to not reflect this dependence in the notation.
With this, the operator $X_{F,g}$ is given by
\begin{equation}\label{X_F,g-alternative}
X_{F,g}=\sum_{k=0}^{m}(-1)^{k}\sum_{\mathcal{M}\subseteq \{1,\ldots,m\}\ : \ |\mathcal{M}|=k}g(\mathbf{1}_{H_{\mathcal{M}}}X\mathbf{1}_{H_{\mathcal{M}}}).
\end{equation}

Before we state the desired estimates, we subdivide the sets $H_{V,L}$ and $H_{F,\infty,V,L}$ into $m$ different parts. Let $k\in\{1,\ldots,m\}$, then we define the sets 
\begin{equation}
H_{L,m,k}:=\{x\in H_{V,L}=[0,L]^d: \ \forall j\in \{1,\ldots,m\}\setminus \{k\}: \ x_{j}\leq x_{k}\} 
\end{equation}
and
\begin{equation}
H_{\infty,L,m,k}:=\{x\in H_{F,\infty,V,L}=\R_+^{m}\times[0,L]^{d-m}: \ \forall j\in \{1,\ldots,m\}\setminus \{k\}: \ x_{j}\leq x_{k}\}. 
\end{equation}
We note that we drop the dependence on the vertex $V$ and the face $F$ in these sets, as we will only use the notation for the designated vertex $V=\{0\}$ and $(d-m)$-face $F=\{0\}^m\times [0,2]^{d-m}$.  
Rearranging the terms of the operator $X_{F,g}$, we obtain
\begin{equation}\label{theorem_higher_order_terms_goal_2}
\mathbf{1}_{H_{V,L}}X_{F,g}=\sum_{k=1}^{m}\mathbf{1}_{H_{L,m,k}}\sum_{\mathcal{M}\subseteq \{1,\ldots,m\}\setminus\{k\}}(-1)^{|\mathcal{M}|}\big[g\big(\mathbf{1}_{H_{\mathcal{M}}}X\mathbf{1}_{H_{\mathcal{M}}}\big)-g\big(\mathbf{1}_{H_{\mathcal{M}\cup\{k\}}}X\mathbf{1}_{H_{\mathcal{M}\cup\{k\}}}\big)\big]
\end{equation}
as well as
\begin{equation}\label{theorem_higher_order_terms_goal_3}
\mathbf{1}_{H_{F,\infty,V,L}}X_{F,g}=\sum_{k=1}^{m}\mathbf{1}_{H_{\infty,L,m,k}}\sum_{\mathcal{M}\subseteq \{1,\ldots,m\}\setminus\{k\}}(-1)^{|\mathcal{M}|}\big[g\big(\mathbf{1}_{H_{\mathcal{M}}}X\mathbf{1}_{H_{\mathcal{M}}}\big)-g\big(\mathbf{1}_{H_{\mathcal{M}\cup\{k\}}}X\mathbf{1}_{H_{\mathcal{M}\cup\{k\}}}\big)\big].
\end{equation}
We now want to show that the operator
\begin{equation}
\mathbf{1}_{H_{\infty,L,m,k}}\big[g\big(\mathbf{1}_{H_{\mathcal{M}}}X\mathbf{1}_{H_{\mathcal{M}}}\big)-g\big(\mathbf{1}_{H_{\mathcal{M}\cup\{k\}}}X\mathbf{1}_{H_{\mathcal{M}\cup\{k\}}}\big)\big]\mathbf{1}_{H_{\infty,L,m,k}}
\end{equation}
is trace class and analyse the operator difference
\begin{equation}
\big(\mathbf{1}_{H_{\infty,L,m,k}}-\mathbf{1}_{H_{L,m,k}}\big)\big[g\big(\mathbf{1}_{H_{\mathcal{M}}}X\mathbf{1}_{H_{\mathcal{M}}}\big)-g\big(\mathbf{1}_{H_{\mathcal{M}\cup\{k\}}}X\mathbf{1}_{H_{\mathcal{M}\cup\{k\}}}\big)\big]\big(\mathbf{1}_{H_{\infty,L,m,k}}-\mathbf{1}_{H_{L,m,k}}\big)
\end{equation}
for given $k,m,\mathcal{M}$ and $g$. We do this in several lemmas in the present section. 

We start with monomial test functions $g$ and extend the result to analytic test functions in Lemma \ref{higher-order-analytic}. As Lemma \ref{higher-order-monom} shows, the expressions above can, for a given monomial $g$, be rewritten as a sum of terms which always contain at least one occurrence of the projection $\mathbf{1}_{H_{\mathcal{M}\cup\{k\}}}-\mathbf{1}_{H_{\mathcal{M}}}$. The following very technical Lemma \ref{Lemma-HS-decay} uses Lemma \ref{Lemma_kernel_set} to estimate the Hilbert-Schmidt norm of such terms. This constitutes the most challenging part of the proof of Theorem \ref{theorem_higher_order_terms}. 

\begin{lem}\label{Lemma-HS-decay}
Let $X$ be a bounded integral operator on $L^{2}(\R^{d})\otimes\C^n$ whose kernel satisfies the bound \eqref{kernel_bound}. Let $1\leq k \leq m < d$ and $\mathcal{M}\subseteq\{1,\ldots,m\}$ such that $k\notin\mathcal{M}$. Let $L\geq 1$ and $p\in\N$. Then the operator 
\begin{equation}\label{operator-HS-proof}
\mathbf{1}_{H_{\infty,L,m,k}}(\mathbf{1}_{H_{\mathcal{M}}}X)^{p}(\mathbf{1}_{H_{\mathcal{M}\cup\{k\}}}-\mathbf{1}_{H_{\mathcal{M}}})
\end{equation}
is Hilbert-Schmidt class and there exists a constant $C>0$, independent of $L$ and $p$, such that
\begin{equation}\label{operator-HS-estimate}
\|\mathbf{1}_{H_{\infty,L,m,k}}(\mathbf{1}_{H_{\mathcal{M}}}X)^{p}(\mathbf{1}_{H_{\mathcal{M}\cup\{k\}}}-\mathbf{1}_{H_{\mathcal{M}}})\|_2\leq C \|X\|^{p-1}d^{\tfrac{p}{2}}p^{d+1}L^{\tfrac{d-m}{2}}.
\end{equation}
Furthermore, there exists a constant $C'>0$ which is independent of $L$ and $p$ such that
\begin{equation}\label{operator-difference-HS-estimate}
\|\big(\mathbf{1}_{H_{\infty,L,m,k}}-\mathbf{1}_{H_{L,m,k}}\big)(\mathbf{1}_{H_{\mathcal{M}}}X)^{p}(\mathbf{1}_{H_{\mathcal{M}\cup\{k\}}}-\mathbf{1}_{H_{\mathcal{M}}})\|_2\leq C' \|X\|^{p-1}d^{\tfrac{p}{2}}p^{d+1}.
\end{equation}
Here, $\|\cdot\|$ denotes the operator norm on the space of bounded linear operators on $L^{2}(\R^{d})\otimes\C^n$.
\end{lem}
\begin{proof} 
Let $L\geq 1$. We start by proving that the operator \eqref{operator-HS-proof} is Hilbert-Schmidt. Due to the fact that $\mathbf{1}_{H_{\infty,L,m,k}}\mathbf{1}_{H_{\mathcal{M}}}=\mathbf{1}_{H_{\infty,L,m,k}}$, there are total of $p-1$ occurrences of the projection $\mathbf{1}_{H_{\mathcal{M}}}$ in \eqref{operator-HS-proof}. In particular the case $p=1$ features no such occurrence. We will start with a proof in this easier case. The idea is to apply Lemma \ref{Lemma_kernel_set} with $M=H_{\infty,L,m,k}$ and $N=H_{\mathcal{M}\cup\{k\}}\setminus H_{\mathcal{M}}=H_{\mathcal{M}\cup\{k\}}\cap\{x\in\R^d:\ x_k<0\}$. We note that these sets satisfy the conditions of Lemma \ref{Lemma_kernel_set} with the function $\varphi$ given by $\varphi(x):=x_{k}$. Therefore, it remains to estimate the integral 
\begin{align}
\int_{M}\frac{1}{(1+\varphi(x)^2)^{\tfrac{d}{2}}}\dd x&=\int_{\R_+^m\times [0,L]^{d-m}\cap \{x\in\R^d:\ \forall j\in\{1,\ldots,m\}:\ x_k\geq x_j\}}\frac{1}{(1+x_{k}^2)^{\tfrac{d}{2}}}\dd x \nonumber \\
&=L^{d-m}\int_0^\infty \frac{x_{k}^{m-1}}{(1+x_{k}^2)^{\tfrac{d}{2}}} \dd x_{k} \nonumber \\ 
&\leq L^{d-m} \Big(\int_0^{1} x_{k}^{m-1}\dd x_{k}+\int_{1}^\infty x_{k}^{m-d-1}\dd x_{k}\Big)=L^{d-m}\Big(\frac{1}{m}+\frac{1}{d-m}\Big),
\end{align}
where we used that $\varphi$ is non-negative and measurable as well as Tonelli's theorem.
Therefore, Lemma \ref{Lemma_kernel_set} yields
\begin{equation}
\|\mathbf{1}_{H_{\infty,L,m,k}}X\mathbf{1}_{H_{\mathcal{M}\cup\{k\}}\setminus H_{\mathcal{M}}}\|_2\leq CL^{\tfrac{d-m}{2}},
\end{equation}
where the constant $C>0$ does not depend on $L$. This proves \eqref{operator-HS-estimate} for $p=1$.

We now turn to the case $p\geq 2$.
For each of the $p-1$ occurrences of the projection $\mathbf{1}_{H_{\mathcal{M}}}$, we will subdivide the corresponding subspace $H_{\mathcal{M}}$ into two parts. To do so, we define the following set, for a given measurable function $\phi:[0,\infty[\,\rightarrow [0,\infty[\,$,
\begin{align}
H_{\phi}:=H_{L,m,k,\mathcal{M},\phi}:=&\{x\in\R^d:\ \forall j\in\{1,\ldots,m\}\setminus\{k\}:\phi(|x_k|)+|x_k|\geq |x_j|\}\nonumber \\&\cap\{x\in\R^d:\ \forall j\in\{m+1,\ldots,d\}:\phi(|x_k|)+L\geq |x_j|\}\cap H_{\mathcal{M}}.
\end{align}
For the remaining part of the proof, we set $\epsilon:=\tfrac{1}{d^p}$. For $l\in\{1,\ldots,p-1\}$, we choose the functions $\phi_{l}:[0,\infty[\,\rightarrow [0,\infty[\,$ given by
\begin{equation}\label{def_phi_k}
\phi_{l}(x):=\tfrac{l}{p}x^{1-\big(\tfrac{d-m}{d}\big)^l+\epsilon}
\end{equation}
and note that we have $1-(\tfrac{d-m}{d})^l+\epsilon<1$ for all $l\in\{1,\ldots,p-1\}$. We start with the first occurrence of $\mathbf{1}_{H_{\mathcal{M}}}$ (from the left) and write $\mathbf{1}_{H_{\mathcal{M}}}=\mathbf{1}_{H_{\phi_{1}}}+\mathbf{1}_{H_{\mathcal{M}}\setminus H_{\phi_{1}}}$. We first prove that the operator $\mathbf{1}_{H_{\infty,L,m,k}}X\mathbf{1}_{H_{\mathcal{M}}\setminus H_{\phi_{1}}}$ is Hilbert-Schmidt. To do so, we apply Lemma \ref{Lemma_kernel_set} with $M=H_{\infty,L,m,k}$ and $N=H_{\mathcal{M}}\setminus H_{\phi_{1}}$. We check that these sets satisfy the conditions of Lemma \ref{Lemma_kernel_set} with the function $\varphi$ given by $\varphi(x)=1_{\{x_k\geq 1\}}\frac{1}{3}\phi_{1}(x_{k})$. To see this, let $x\in M$ with $x_k\geq 1$ be fixed and let $y\in B_{\varphi(x)}(x)\cap H_{\mathcal{M}}$ be arbitrary. It suffices to show that $y\in H_{\phi_{1}}$. As $y\in B_{\varphi(x)}(x)$, we obtain the following bound for $y_{k}$
\begin{equation}
y_k\geq x_k-\tfrac{1}{3p}x_{k}^{1-\tfrac{d-m}{d}+\epsilon}\geq x_k \tfrac{3p-1}{3p}.
\end{equation}
Together with the monotonicity of $\phi_{1}$, a bound for $\phi_{1}(y_k)$ follows
\begin{equation}\label{phi_k+1_bound_first_term}
\phi_{1}(y_k)\geq \phi_{1}\big(x_k \tfrac{3p-1}{3p}\big)\geq \tfrac{1}{p}\tfrac{3p-1}{3p}x_{k}^{1-\tfrac{d-m}{d}+\epsilon}= \tfrac{3p-1}{3p^2}x_{k}^{1-\tfrac{d-m}{d}+\epsilon}\geq \tfrac{2}{3}\phi_{1}(x_{k}).
\end{equation}
For each $j\in\{1,\ldots,m\}\setminus\{k\}$, we obtain
\begin{equation}
|y_{j}|\leq |x_{j}|+|y_{j}-x_{j}|\leq x_{k}+\tfrac{1}{3}\phi_{1}(x_{k})\leq y_{k}+\tfrac{2}{3}\phi_{1}(x_{k})\leq y_{k}+\phi_{1}(y_k),
\end{equation}
where we used \eqref{phi_k+1_bound_first_term} in the last inequality. For $j\in\{m+1,\ldots,d\}$, we obtain
\begin{equation}
|y_{j}|\leq |x_{j}|+|y_{j}-x_{j}|\leq L+\tfrac{1}{3}\phi_{1}(x_{k})\leq L+ \phi_{1}(y_{k}),
\end{equation}
where we again used \eqref{phi_k+1_bound_first_term} in the last inequality. 
Therefore, in order to apply Lemma \ref{Lemma_kernel_set}, it remains to evaluate the integral
\begin{align}
\int_{M}\frac{1}{(1+\varphi(x)^2)^{\tfrac{d}{2}}}\dd x&=\int_{\R_+^m\times [0,L]^{d-m}\cap \{x\in\R^d:\ \forall j\in\{1,\ldots,m\}:\ x_k\geq x_j\}}\frac{1}{\big(1+\big(1_{\{x_k\geq 1\}}\frac{1}{3}\phi_{1}(x_{k})\big)^2\big)^{\tfrac{d}{2}}}\dd x \nonumber \\
&=L^{d-m}\int_0^\infty \frac{x_{k}^{m-1}}{\big(1+\big(1_{\{x_k\geq 1\}}\frac{1}{3}\phi_{1}(x_{k})\big)^2\big)^{\tfrac{d}{2}}} \dd x_{k} \nonumber \\ &\leq L^{d-m} \Big(\int_0^{1} x_{k}^{m-1}\dd x_{k}+\int_{1}^\infty (3p)^{d}x_{k}^{m-m-1-d\epsilon}\dd x_{k}\Big)\leq C_{1} L^{d-m}\frac{p^{d}}{\epsilon},
\end{align}
where the constant $C_{1}>0$ is independent of $L$ and $p$.
Therefore, Lemma \ref{Lemma_kernel_set} yields 
\begin{equation}\label{operator-HS-estimate-1}
\|\mathbf{1}_{H_{\infty,L,m,k}}X\mathbf{1}_{H_{\mathcal{M}}\setminus H_{\phi_{1}}}\|_2\leq C_{1}'\sqrt{\frac{L^{d-m}p^{d}}{\epsilon}},
\end{equation}
with a constant $C_{1}'>0$ independent of $L$ and $p$.

We now turn to the intermediate terms, i.e. we show that the operators $\mathbf{1}_{H_{\phi_{l}}}X\mathbf{1}_{H_{\mathcal{M}}\setminus H_{\phi_{l+1}}}$ are Hilbert-Schmidt for $l\in\{1,\ldots,p-2\}$. We again apply Lemma \ref{Lemma_kernel_set} with the sets $M=H_{\phi_{l}}$ and $N=H_{\mathcal{M}}\setminus H_{\phi_{l+1}}$. We show that these sets satisfy the requirements of Lemma \ref{Lemma_kernel_set} with the function $\varphi(x)=1_{\{x_k\geq 1\}}\frac{1}{2(l+1)p}\phi_{l+1}(x_{k})$. To see this, let $x\in H_{\phi_{l}}$ with $x_k\geq 1$ be fixed and let $y\in B_{\varphi(x)}(x)\cap H_{\mathcal{M}}$ be arbitrary. It suffices to show that $y\in H_{\phi_{l+1}}$. As we have $y\in B_{\varphi(x)}(x)$, we obtain the following bound for $y_{k}$
\begin{equation}
y_k\geq x_k-\tfrac{1}{2p^2}x_{k}^{1-\big(\tfrac{d-m}{d}\big)^{l+1}+\epsilon}\geq x_k \tfrac{2p^2-1}{2p^2}.
\end{equation}
Together with the monotonicity of $\phi_{l+1}$, we see that
\begin{equation}
\phi_{l+1}(y_k)\geq \phi_{l+1}\big(x_k \tfrac{2p^2-1}{2p^2}\big)\geq \tfrac{l+1}{p}\tfrac{2p^2-1}{2p^2}\, x_{k}^{1-\big(\tfrac{d-m}{d}\big)^{l+1}+\epsilon} = \tfrac{2lp^2+2p^2-l-1}{2p^3}\, x_{k}^{1-\big(\tfrac{d-m}{d}\big)^{l+1}+\epsilon}.
\end{equation}
As we have $\tfrac{d-m}{d} \leq 1$ and $2p^2-l-1\geq 2p$,
this yields the following bound
\begin{equation}\label{phi_k+1_bound}
\phi_{l+1}(y_k) \geq \tfrac{l}{p}\,x_{k}^{1-\big(\tfrac{d-m}{d}\big)^{l}+\epsilon}+\tfrac{1}{p^2}\,x_{k}^{1-\big(\tfrac{d-m}{d}\big)^{l+1}+\epsilon}
= \phi_{l}(x_k)+\tfrac{1}{(l+1)p}\phi_{l+1}(x_{k}).
\end{equation}
For each $j\in\{1,\ldots,m\}\setminus\{k\}$, this yields
\begin{align}
|y_{j}|\leq& \, |x_{j}|+|y_{j}-x_{j}|\leq x_{k}+\phi_{l}(x_{k})+\tfrac{1}{2(l+1)p}\phi_{l+1}(x_{k})\nonumber \\
\leq& \, y_{k}+\phi_{l}(x_{k})+\tfrac{1}{(l+1)p}\phi_{l+1}(x_{k})
\leq y_{k}+ \phi_{l+1}(y_{k}),
\end{align}
where we used \eqref{phi_k+1_bound} in the last inequality. For $j\in\{m+1,\ldots,d\}$, we obtain
\begin{equation}
|y_{j}|\leq |x_{j}|+|y_{j}-x_{j}|\leq L+\phi_{l}(x_{k})+\tfrac{1}{2(l+1)p}\phi_{l+1}(x_{k})\leq L+ \phi_{l+1}(y_{k}),
\end{equation}
where we again used \eqref{phi_k+1_bound} in the last inequality. Therefore, we have $y\in H_{\phi_{l+1}}$. In order to apply Lemma \ref{Lemma_kernel_set}, it remains to estimate the integral
\begin{equation}\label{HS_higher_order_int_1}
\int_{M}\frac{1}{(1+\varphi(x)^2)^{\tfrac{d}{2}}}\dd x 
=\int_{H_{\phi_{l}}}\frac{1}{\big(1+\big(1_{\{x_k\geq 1\}}\tfrac{1}{2(l+1)p}\phi_{l+1}(x_{k})\big)^2\big)^{\tfrac{d}{2}}}\dd x.
\end{equation}
We note that $x_k+\phi_l(x_k)\leq 2x_k$. By the definition of the set $H_{\phi_l}$ the right-hand side of \eqref{HS_higher_order_int_1} is bounded from above by
\begin{multline}\label{HS_higher_order_int_2}
2^{d}\int_{[0,2x_{k}]^{k-1}\times \R_+\times [0,2x_{k}]^{m-k}\times [0,L+\phi_{l}(x_{k})]^{d-m}}\frac{1}{\big(1+\big(1_{\{x_k\geq 1\}}\tfrac{1}{2(l+1)p}\phi_{l+1}(x_{k})\big)^2\big)^{\tfrac{d}{2}}}\dd x \\ 
\leq 2^{d}\int_{0}^{\infty}\frac{2^{m-1}x_{k}^{m-1}\Big(L+x_{k}^{1-\tfrac{(d-m)^{l}}{d^{l}}+\epsilon}\Big)^{d-m}}{\big(1+\big(1_{\{x_k\geq 1\}}\tfrac{1}{2(l+1)p}\phi_{l+1}(x_{k})\big)^2\big)^{\tfrac{d}{2}}}\dd x_{k}.
\end{multline}
Splitting the integral on the right-hand side of \eqref{HS_higher_order_int_2} into two parts and applying the definition \eqref{def_phi_k} of $\phi_{l+1}$, we see that the right-hand side of \eqref{HS_higher_order_int_2} is bounded from above by
\begin{align}
&2^{d+m-1}\Bigg[\int_{0}^{1}x_{k}^{m-1}2^{d-m}L^{d-m}\dd x_k \nonumber \\ 
&\qquad\qquad\ +\sum_{j=0}^{d-m}\frac{(d-m)!}{j!(d-m-j)!}\int_{1}^{\infty}\frac{2^{d}p^{2d}x_{k}^{m-1+j-j\tfrac{(d-m)^{l}}{d^{l}}+j\epsilon}L^{d-m-j}}{x_{k}^{d-\tfrac{(d-m)^{l+1}}{d^{l}}+d\epsilon}}\dd x_k \Bigg]\nonumber\\
&=  2^{d+m-1}\Bigg[\frac{2^{d-m}L^{d-m}}{m}\nonumber \\ 
&\qquad\qquad\ \quad +2^{d}p^{2d}\sum_{j=0}^{d-m}\frac{(d-m)!}{j!(d-m-j)!}L^{d-m-j}\int_{1}^{\infty}x_{k}^{-1-(d-m-j)+(d-m-j)\tfrac{(d-m)^{l}}{d^{l}}-(d-j)\epsilon}\dd x_k\Bigg]\nonumber \\ 
&\leq 2^{d+m-1}\bigg[\frac{2^{d-m}L^{d-m}}{m}+2^{d}p^{2d}\sum_{j=0}^{d-m}\frac{(d-m)!}{j!(d-m-j)!}\frac{L^{d-m-j}}{\epsilon}\bigg]\leq 
C_{l}\frac{p^{2d}L^{d-m}}{\epsilon},
\end{align}
where the constant $C_{l}>0$ is independent of $L$ and $p$. Therefore, Lemma \ref{Lemma_kernel_set} yields 
\begin{equation}\label{operator-HS-estimate-2}
\|\mathbf{1}_{H_{\phi_{l}}}X\mathbf{1}_{H_{\mathcal{M}}\setminus H_{\phi_{l+1}}}\|_2\leq C_{l}'\sqrt{\frac{p^{2d}L^{d-m}}{\epsilon}},
\end{equation}
with a constant $C_{l}'>0$ independent of $L$ and $p$.

It remains to show that the remaining operator $\mathbf{1}_{H_{\phi_{p-1}}}X\mathbf{1}_{H_{\mathcal{M}\cup\{k\}}\setminus H_{\mathcal{M}}}$ is Hilbert-Schmidt. As before, we do this by an application of Lemma \ref{Lemma_kernel_set}. We note that the sets $M=H_{\phi_{p-1}}$ and $N=H_{\mathcal{M}\cup\{k\}}\setminus H_{\mathcal{M}}=H_{\mathcal{M}\cup\{k\}}\cap\{x\in\R^d:\ x_k<0\}$ satisfy the requirements of Lemma \ref{Lemma_kernel_set} with the function $\varphi(x)=x_{k}$. Again we need to evaluate the corresponding integral
\begin{multline}\label{HS_higher_order_int_3}
\int_{H_{\phi_{p-1}}}\frac{1}{(1+x_{k}^2)^{\tfrac{d}{2}}}\dd x  \\
= 2^{d-m+|\mathcal{M}|}\int_{[0,x_{k}+\phi_{p-1}(x_{k})]^{k-1}\times \R_+\times [0,x_{k}+\phi_{p-1}(x_{k})]^{m-k}\times [0,L+\phi_{p-1}(x_{k})]^{d-m}}\frac{1}{(1+x_{k}^2)^{\tfrac{d}{2}}}\dd x.
\end{multline}
The right-hand side of \eqref{HS_higher_order_int_3} is bounded from above by
\begin{align}
 2&^{d}\int_{0}^{\infty}\frac{2^{m-1}x_{k}^{m-1}\Big(L+x_{k}^{1-\big(\tfrac{d-m}{d}\big)^{p-1}+\epsilon}\Big)^{d-m}}{(1+x_{k}^2)^{\tfrac{d}{2}}}\dd x_{k} \nonumber\\ 
&\leq 2^{d+m-1}\Bigg[\int_{0}^{1}x_{k}^{m-1}2^{d-m}L^{d-m}\dd x_k \nonumber \\ 
&\qquad\qquad\qquad +\sum_{j=0}^{d-m}\frac{(d-m)!}{j!(d-m-j)!}L^{d-m-j}\int_{1}^{\infty}x_{k}^{-1-(d-m-j)-j\big(\tfrac{d-m}{d}\big)^{p-1}+j\epsilon}\dd x_k \Bigg]\nonumber\\
&\leq  2^{d+m-1}\Big[\frac{2^{d-m}L^{d-m}}{m}+\sum_{j=0}^{d-m}\frac{(d-m)!}{j!(d-m-j)!}\frac{L^{d-m-j}}{\epsilon}\Big]\leq 
C_{p-1}\sqrt{\frac{L^{d-m}}{\epsilon}},
\end{align}
with a constant $C_{p-1}>0$ independent of $L$ and $p$.
Therefore, Lemma \ref{Lemma_kernel_set} yields 
\begin{equation}\label{operator-HS-estimate-3}
\|\mathbf{1}_{H_{\phi_{p-1}}}\mathbf{1}_{H_{\mathcal{M}\cup\{k\}}\setminus H_{\mathcal{M}}}\|_2\leq C_{p-1}'\sqrt{\frac{L^{d-m}}{\epsilon}},
\end{equation}
with a constant $C_{p-1}'>0$ independent of $L$ and $p$.
Repeated use of the triangle inequality, followed by H\"older's inequality, yields
\begin{align}
\|\mathbf{1}_{H_{\infty,L,m,k}}&(\mathbf{1}_{H_{\mathcal{M}}}X)^{p}(\mathbf{1}_{H_{\mathcal{M}\cup\{k\}}}-\mathbf{1}_{H_{\mathcal{M}}})\|_2 \nonumber \\ 
\leq& \, \sum_{l=0}^{p-2}\Big\|\mathbf{1}_{H_{\infty,L,m,k}}\Big(\prod_{j=1}^{l}X\mathbf{1}_{H_{\phi_{j}}}\Big)X\mathbf{1}_{H_{\mathcal{M}}\setminus H_{\phi_{l+1}}}(X\mathbf{1}_{H_{\mathcal{M}}})^{p-2-l}X(\mathbf{1}_{H_{\mathcal{M}\cup\{k\}}}-\mathbf{1}_{H_{\mathcal{M}}})\Big\|_2 \nonumber \\ 
&+ \Big\|\mathbf{1}_{H_{\infty,L,m,k}}\Big(\prod_{j=1}^{p-1}X\mathbf{1}_{H_{\phi_{j}}}\Big)X(\mathbf{1}_{H_{\mathcal{M}\cup\{k\}}}-\mathbf{1}_{H_{\mathcal{M}}})\Big\|_2 \nonumber \\
\leq& \, \|X\|^{p-1}\Big( \|\mathbf{1}_{H_{\infty,L,m,k}}X\mathbf{1}_{H_{\mathcal{M}}\setminus H_{\phi_{1}}}\|_2+\sum_{l=1}^{p-2}\|\mathbf{1}_{H_{\phi_{l}}}X\mathbf{1}_{H_{\mathcal{M}}\setminus H_{\phi_{l+1}}}\|_2+\|\mathbf{1}_{H_{\phi_{p-1}}}X\mathbf{1}_{H_{\mathcal{M}\cup\{k\}}\setminus H_{\mathcal{M}}}\|_2 \Big) \nonumber\\
\leq& \, C \|X\|^{p-1}p\sqrt{\frac{p^{2d}L^{d-m}}{\epsilon}}\leq C \|X\|^{p-1}d^{\tfrac{p}{2}}p^{d+1}L^{\tfrac{d-m}{2}},
\end{align}
where we combined estimates \eqref{operator-HS-estimate-1}, \eqref{operator-HS-estimate-2} and \eqref{operator-HS-estimate-3} and used the definition of $\epsilon$ in the last line. This proves \eqref{operator-HS-estimate} for $p\geq 2$.

We continue with the proof of \eqref{operator-difference-HS-estimate}. The proof works in a similar fashion to the one of \eqref{operator-HS-estimate}, although it is a bit more involved. We require different partitions of the set $H_{\mathcal{M}}$ as well as different functions $\phi_{l}$, $l\in\{1,\ldots,p-1\}$. Before we define them, we again give a proof in the easier case $p=1$. We apply Lemma \ref{Lemma_kernel_set} with the sets $M=H_{\infty,L,m,k}\setminus H_{L,m,k}$ and $N=H_{\mathcal{M}\cup\{k\}}\setminus H_{\mathcal{M}}=\{x\in H_{\mathcal{M}\cup\{k\}}:\ x_k<0\}$. These sets fulfil the requirements of Lemma \ref{Lemma_kernel_set} with the function $\varphi(x)=x_{k}$ and we compute
\begin{align}
\int_{M}\frac{1}{(1+\varphi(x)^2)^{\tfrac{d}{2}}}\dd x&=\int_{[0,x_{k}[^{k-1}\times[L,\infty[\,\times[0,x_{k}[^{m-k}\times [0,L]^{d-m}}\frac{1}{(1+x_{k}^2)^{\tfrac{d}{2}}}\dd x \nonumber \\ 
&= L^{d-m}\int_{L}^\infty \frac{x_{k}^{m-1}}{(1+x_{k}^2)^{\tfrac{d}{2}}} \dd x_{k} \leq L^{d-m}\int_{L}^\infty x_{k}^{m-d-1}\dd x_{k}=\frac{1}{d-m}.
\end{align}
Then Lemma \ref{Lemma_kernel_set} yields the bound \eqref{operator-difference-HS-estimate} in the case $p=1$.

In the case $p\geq 2$, we define the sets
\begin{align}
H_{\phi,l}&:=H_{L,m,k,\mathcal{M},p,\phi,l}\nonumber \\
&:= \{x\in\R^d:\ \forall j\in\{1,\ldots,m\}\setminus\{k\}:\phi(|x_k|)+|x_k|\geq |x_j|\}\nonumber \\
& \ \cap \{x\in\R^d:\ \forall j\in\{m+1,\ldots,d\}:\phi(|x_k|)+L\geq |x_j|\}\nonumber \\
& \ \cap H_{\mathcal{M}} \cap \{x\in\R^d:\ x_{k}\geq \tfrac{p-l}{p}L\}.
\end{align}
for a given measurable function $\phi:[0,\infty[\,\rightarrow [0,\infty[\,$.
We recall $\epsilon=\tfrac{1}{d^p}$. For $l\in\{1,\ldots,p-1\}$, we choose the functions $\phi_{l,L}:[0,\infty[\,\rightarrow [0,\infty[\,$ with
\begin{equation}
\phi_{l,L}(x):=\tfrac{l}{p}L^{\big(\tfrac{d-m}{d}\big)^l-\epsilon}\big(\tfrac{p}{p+1-l} x\big)^{1-\big(\tfrac{d-m}{d}\big)^l+\epsilon}.
\end{equation}
We again begin with the first occurrence of $\mathbf{1}_{H_{\mathcal{M}}}$ (from the left) and write $\mathbf{1}_{H_{\mathcal{M}}}=\mathbf{1}_{H_{\phi_{1,L},1}}+\mathbf{1}_{H_{\mathcal{M}}\setminus H_{\phi_{1,L},1}}$. We first prove that the operator $\big(\mathbf{1}_{H_{\infty,L,m,k}}-\mathbf{1}_{H_{L,m,k}}\big)X\mathbf{1}_{H_{\mathcal{M}}\setminus H_{\phi_{1,L},1}}$ has its Hilbert-Schmidt norm bounded independently of $L$. To do so, we apply Lemma \ref{Lemma_kernel_set} with $M=H_{\infty,L,m,k}\setminus H_{L,m,k}$ and $N=H_{\mathcal{M}}\setminus H_{\phi_{1,L},1}$. We check that these sets satisfy the conditions of Lemma \ref{Lemma_kernel_set} with the function $\varphi$ given by $\varphi(x)=\frac{1}{3}\phi_{1,L}(x_{k})$. To see this, let $x\in M$ be fixed and let $y\in B_{\varphi(x)}(x)\cap H_{\mathcal{M}}$ be arbitrary. It suffices to show that $y\in H_{\phi_{1,L},1}$. As $y\in B_{\varphi(x)}(x)$, we obtain the following bound for $y_{k}$
\begin{equation}
y_k\geq x_k-\tfrac{1}{3p}L^{\tfrac{d-m}{d}-\epsilon}x_{k}^{1-\tfrac{d-m}{d}+\epsilon}\geq x_{k}-\tfrac{1}{3p}x_{k}= x_{k} \tfrac{3p-1}{3p},
\end{equation}
where we used $x_{k}\geq L \geq 1$ (as $x\in M$) in the last inequality. In particular we have
\begin{equation}
y_k\geq x_{k} \tfrac{3p-1}{3p}\geq \tfrac{3p-1}{3p}L\geq \tfrac{p-1}{p}L.
\end{equation}
With the monotonicity of $\phi_{1,L}$, we obtain the following bound for $\phi_{1,L}(y_k)$
\begin{equation}\label{phi_k+1_bound_first_term_L}
\phi_{1,L}(y_k)\geq \phi_{1,L}\big(x_k \tfrac{3p-1}{3p}\big)\geq \tfrac{1}{p}\tfrac{3p-1}{3p}L^{\tfrac{d-m}{d}-\epsilon}x_{k}^{1-\tfrac{d-m}{d}+\epsilon}= \tfrac{3p-1}{3p}\phi_{1,L}(x_{k})\geq \tfrac{2}{3}\phi_{1,L}(x_{k}).
\end{equation}
For each $j\in\{1,\ldots,m\}\setminus\{k\}$, we obtain
\begin{equation}
|y_{j}|\leq |x_{j}|+|y_{j}-x_{j}|\leq x_{k}+\tfrac{1}{3}\phi_{1,L}(x_{k})\leq y_{k}+\tfrac{2}{3}\phi_{1,L}(x_{k})\leq y_{k}+\phi_{1,L}(y_k),
\end{equation}
where we used \eqref{phi_k+1_bound_first_term_L} in the last inequality. For $j\in\{m+1,\ldots,d\}$, we obtain
\begin{equation}
|y_{j}|\leq |x_{j}|+|y_{j}-x_{j}|\leq L+\tfrac{1}{3}\phi_{1,L}(x_{k})\leq L+ \phi_{1,L}(y_{k}),
\end{equation}
where we again used \eqref{phi_k+1_bound_first_term_L} in the last inequality. 
Therefore, in order to apply Lemma \ref{Lemma_kernel_set}, it remains to evaluate the integral
\begin{align}
\int_{M}\frac{1}{(1+\varphi(x)^2)^{\tfrac{d}{2}}}\dd x&=\int_{[0,x_{k}[^{k-1}\times[L,\infty[\,\times[0,x_{k}[^{m-k}\times[0,L]^{d-m}}\frac{1}{\big(1+\big(\frac{1}{3}\phi_{1,L}(x_{k})\big)^2\big)^{\tfrac{d}{2}}}\dd x \nonumber \\
&=L^{d-m}\int_{L}^{\infty} \frac{x_{k}^{m-1}}{\big(1+\big(\frac{1}{3}\phi_{1,L}(x_{k})\big)^2\big)^{\tfrac{d}{2}}} \dd x_{k} \nonumber \\ 
&\leq L^{d-m} \int_{L}^{\infty} (3p)^{d}L^{m-d+d\epsilon}x_{k}^{m-m-1-d\epsilon}\dd x_{k}\leq C_{1} \frac{L^{d\epsilon}}{L^{d\epsilon}}\frac{p^{d}}{\epsilon}=C_{1}\frac{p^{d}}{\epsilon},
\end{align}
where the constant $C_{1}>0$ is independent of $L$ and $p$.
Therefore, Lemma \ref{Lemma_kernel_set} yields 
\begin{equation}\label{operator-HS-estimate-1-L}
\|\big(\mathbf{1}_{H_{\infty,L,m,k}}-\mathbf{1}_{H_{L,m,k}}\big)X\mathbf{1}_{H_{\mathcal{M}}\setminus H_{\phi_{1,L},1}}\|_2\leq C_{1}'\sqrt{\frac{p^{d}}{\epsilon}}
\end{equation}
with a constant $C_{1}'>0$ independent of $L$ and $p$.

We again turn towards the intermediate terms, i.e. we show that the operators $\mathbf{1}_{H_{\phi_{l,L},l}}X\mathbf{1}_{H_{\mathcal{M}}\setminus H_{\phi_{l+1,L},l+1}}$ have their Hilbert-Schmidt norm bounded independently of $L$, for $l\in\{1,\ldots,p-2\}$. We again apply Lemma \ref{Lemma_kernel_set} with the sets $M=H_{\phi_{l,L},l}$ and $N=H_{\mathcal{M}}\setminus H_{\phi_{l+1,L},l+1}$. We show that these sets satisfy the requirements of Lemma \ref{Lemma_kernel_set} with the function $\varphi(x)=\frac{1}{2(l+1)p}\phi_{l+1,L}(x_{k})$. To see this, let $x\in H_{\phi_{l,L},l}$ be fixed and let $y\in B_{\varphi(x)}(x)\cap H_{\mathcal{M}}$ be arbitrary. It suffices to show that $y\in H_{\phi_{l+1,L},l+1}$. As $y\in B_{\varphi(x)}(x)$, we get the following bound for $y_{k}$
\begin{multline}
y_k\geq x_k-\tfrac{1}{2p^2}L^{\big(\tfrac{d-m}{d}\big)^{l+1}-\epsilon}\big(\tfrac{p}{p+1-l-1} x_{k}\big)^{1-\big(\tfrac{d-m}{d}\big)^{l+1}+\epsilon}  \\
\geq x_{k}-\tfrac{1}{2p^2}\big(\tfrac{p}{p-l} x_{k}\big)^{\big(\tfrac{d-m}{d}\big)^{l+1}-\epsilon}\big(\tfrac{p}{p-l} x_{k}\big)^{1-\big(\tfrac{d-m}{d}\big)^{l+1}+\epsilon}= x_k \tfrac{2p(p-l)-1}{2p(p-l)},
\end{multline}
where we used $\frac{p}{p-l}x_{k}\geq  L$ (as $x\in M$) in the last inequality. In particular, we obtain
\begin{equation}
y_k\geq x_{k} \tfrac{2p(p-l)-1}{2p(p-l)}\geq \tfrac{p-l}{p}\tfrac{2p(p-l)-1}{2p(p-l)}L 
=\tfrac{2p^2-2pl-1}{2p^2}L\geq \tfrac{p-(l+1)}{p}L.
\end{equation}
With the monotonicity of $\phi_{l+1,L}$, we have
\begin{equation}
\phi_{l+1}(y_k)\geq \phi_{l+1}\big(x_k \tfrac{2p(p-l)-1}{2p(p-l)}\big)\geq \tfrac{l+1}{p}\tfrac{2p(p-l)-1}{2p(p-l)}L^{\big(\tfrac{d-m}{d}\big)^{l+1}-\epsilon}\big(\tfrac{p}{p-l}x_{k}\big)^{1-\big(\tfrac{d-m}{d}\big)^{l+1}+\epsilon},
\end{equation}
which is bounded from below by
\begin{multline}\label{phi_k+1_bound_L_step_1}
\tfrac{2lp(p-l)+2(p-l)}{2p^2(p-l)}L^{\big(\tfrac{d-m}{d}\big)^{l+1}-\epsilon}\big(\tfrac{p}{p-l}x_{k}\big)^{1-\big(\tfrac{d-m}{d}\big)^{l+1}+\epsilon}  \\
= \tfrac{l}{p}L^{\big(\tfrac{d-m}{d}\big)^{l+1}-\epsilon}\big(\tfrac{p}{p-l}x_{k}\big)^{1-\big(\tfrac{d-m}{d}\big)^{l}+\epsilon}\big(\tfrac{p}{p-l}x_{k}\big)^{\big(\tfrac{d-m}{d}\big)^{l}-\big(\tfrac{d-m}{d}\big)^{l+1}}+\tfrac{1}{(l+1)p}\phi_{l+1,L}(x_{k}).
\end{multline}
Using the fact that $\frac{p}{p-l}x_{k}\geq  L$, as $x\in H_{\phi_{l,L},l}$, we see that the first term on the right-hand side of 
\eqref{phi_k+1_bound_L_step_1} is bounded from below by 
\begin{equation}
\tfrac{l}{p}L^{\big(\tfrac{d-m}{d}\big)^{l+1}-\epsilon}\big(\tfrac{p}{p+1-l}x_{k}\big)^{1-\big(\tfrac{d-m}{d}\big)^{l}+\epsilon}L^{\big(\tfrac{d-m}{d}\big)^{l}-\big(\tfrac{d-m}{d}\big)^{l+1}} 
=\phi_{l,L}(x_k).
\end{equation}
This in turn yields the bound
\begin{equation}\label{phi_k+1_bound_L}
\phi_{l+1}(y_k)\geq\phi_{l,L}(x_k)+\tfrac{1}{(l+1)p}\phi_{l+1,L}(x_{k}).
\end{equation}
Therefore, for each $j\in\{1,\ldots,m\}\setminus\{k\}$, we obtain
\begin{multline}
|y_{j}|\leq |x_{j}|+|y_{j}-x_{j}|\leq x_{k}+\phi_{l,L}(x_{k})+\tfrac{1}{2(l+1)p}\phi_{l+1,L}(x_{k}) \\
\leq y_{k}+\phi_{l,L}(x_{k})+\tfrac{1}{(l+1)p}\phi_{l+1,L}(x_{k})
\leq y_{k}+ \phi_{l+1,L}(y_{k}),
\end{multline}
where we used \eqref{phi_k+1_bound_L} in the last inequality. For $j\in\{m+1,\ldots,d\}$, we obtain
\begin{equation}
|y_{j}|\leq |x_{j}|+|y_{j}-x_{j}|\leq L+\phi_{l,L}(x_{k})+\tfrac{1}{2(l+1)p}\phi_{l+1,L}(x_{k})\leq L+ \phi_{l+1,L}(y_{k}),
\end{equation}
where we again used \eqref{phi_k+1_bound_L} in the last inequality. Before we are ready to evaluate the integral for Lemma \ref{Lemma_kernel_set}, we note that for every $x\in H_{\phi_{l,L},l}$ and $l\in\{1,\ldots,p-2\}$ we have
\begin{multline}\label{auxiliary_estimate_phi}
x_{k}+\phi_{l,L}(x_{k})=x_{k}+\tfrac{l}{p}L^{\big(\tfrac{d-m}{d}\big)^l-\epsilon}\big(\tfrac{p}{p+1-l} x_{k}\big)^{1-\big(\tfrac{d-m}{d}\big)^l+\epsilon}\\
\leq x_{k}+\tfrac{l}{p}\big(\tfrac{p}{p-l}x_{k}\big)^{\big(\tfrac{d-m}{d}\big)^l-\epsilon}\big(\tfrac{p}{p+1-l} x_{k}\big)^{1-\big(\tfrac{d-m}{d}\big)^l+\epsilon}\leq x_{k}+\tfrac{l}{p-l}x_{k}= \tfrac{p}{p-l}x_{k},
\end{multline}
where we used $\frac{p}{p-l}x_{k}\geq  L$. With this at hand, it remains to bound the integral
\begin{equation}\label{HS_higher_order_int_4}
\int_{M}\frac{1}{(1+\varphi(x)^2)^{\tfrac{d}{2}}}\dd x 
=\int_{H_{\phi_{l,L},l}}\frac{1}{\big(1+\big(\frac{1}{2(l+1)p}\phi_{l+1,L}(x_{k})\big)^2\big)^{\tfrac{d}{2}}}\dd x
\end{equation}
in order to apply Lemma \ref{Lemma_kernel_set}.
By the definition of $H_{\phi_{l,L},l}$ and \eqref{auxiliary_estimate_phi}, the right-hand side of \eqref{HS_higher_order_int_4} is bounded from above by
\begin{multline}
2^d\int_{\big[0,\tfrac{p}{p-l}x_{k}\big]^{k-1}\times \big[\tfrac{p-l}{p} L,\infty\big[\,\times \big[0,\tfrac{p}{p-l}x_{k}\big]^{m-k}\times [0,L+\phi_{l,L}(x_{k})]^{d-m}}\frac{1}{\big(\frac{1}{2(l+1)p}\phi_{l+1,L}(x_{k})\big)^{d}}\dd x  \\ 
= 2^d\int_{\tfrac{p-l}{p} L}^{\infty}\frac{\big(\tfrac{p}{p-l}x_{k}\big)^{m-1}\Big(L+\tfrac{l}{p}L^{\big(\tfrac{d-m}{d}\big)^{l}-\epsilon} \big(\tfrac{p}{p+1-l}x_{k}\big)^{1-\big(\tfrac{d-m}{d}\big)^{l}+\epsilon}\Big)^{d-m}}{\big(\frac{1}{2(l+1)p}\phi_{l+1,L}(x_{k})\big)^{d}}\dd x_{k}.
\end{multline}
As $\frac{p}{p-l}x_{k}\geq  L$, this is in turn bounded from above by
\begin{equation}
2^{2d-m}\int_{\tfrac{p-l}{p} L}^{\infty}\frac{\big(\tfrac{p}{p-l}x_{k}\big)^{m-1}\Big(L^{\big(\tfrac{d-m}{d}\big)^{l}-\epsilon} \big(\tfrac{p}{p-l}x_{k}\big)^{1-\big(\tfrac{d-m}{d}\big)^{l}+\epsilon}\Big)^{d-m}}{\big(\frac{1}{2(l+1)p}\phi_{l+1,L}(x_{k})\big)^{d}}\dd x_{k}.
\end{equation}
Using the definition of $\phi_{l+1,L}$, this is equal to
\begin{multline}
2 ^{3d-m}p^{2d}\int_{\tfrac{p-l}{p} L}^{\infty}\frac{\big(\tfrac{p}{p-l}x_{k}\big)^{m-1}L^{\tfrac{(d-m)^{l+1}}{d^{l}}-(d-m)\epsilon} \big(\tfrac{p}{p-l}x_{k}\big)^{d-m-\tfrac{(d-m)^{l+1}}{d^{l}}+(d-m)\epsilon}}{L^{\tfrac{(d-m)^{l+1}}{d^{l}}-d\epsilon}\big(\tfrac{p}{p-l} x_{k}\big)^{d-\tfrac{(d-m)^{l+1}}{d^{l}}+d\epsilon}}\dd x_{k} \\
=2 ^{3d-m}p^{2d}\int_{\tfrac{p-l}{p} L}^{\infty}L^{m\epsilon}\big(\tfrac{p}{p-l}x_{k}\big)^{-1-m\epsilon}\dd x_{k}.
\end{multline}
Carrying out the integration, we see that the right-hand side of \eqref{HS_higher_order_int_4} bounded from above by
\begin{equation}
2 ^{3d-m}p^{2d}L^{m\epsilon-m\epsilon}\tfrac{1}{m\epsilon}\leq C_l\frac{p^{2d}}{\epsilon},
\end{equation}
where the constant $C_{l}>0$ is independent of $L$ and $p$. Therefore, Lemma \ref{Lemma_kernel_set} yields 
\begin{equation}\label{operator-HS-estimate-2-L}
\|\mathbf{1}_{H_{\phi_{l}}}X\mathbf{1}_{H_{\mathcal{M}}\setminus H_{\phi_{l+1}}}\|_2\leq C_{l}'\sqrt{\frac{p^{2d}}{\epsilon}}
\end{equation}
with a constant $C_{l}'>0$ independent of $L$ and $p$.

It again remains to show that the remaining operator $\mathbf{1}_{H_{\phi_{p-1,L},p-1}}X\mathbf{1}_{H_{\mathcal{M}\cup\{k\}}\setminus H_{\mathcal{M}}}$ has Hilbert-Schmidt norm bounded independently of $L$. As before, we do this by an application of Lemma \ref{Lemma_kernel_set}. We note that the sets $M=H_{\phi_{p-1,L},p-1}$ and $N=H_{\mathcal{M}\cup\{k\}}\setminus H_{\mathcal{M}}=H_{\mathcal{M}\cup\{k\}}\cap\{x\in\R^d:\ x_k<0\}$ satisfy the requirements of Lemma \ref{Lemma_kernel_set} with the function $\varphi(x)=x_{k}$. With the definition of $H_{\phi_{p-1,L},p-1}$, the corresponding integral reads
\begin{align}\label{HS_higher_order_int_5}
\int_{M}&\frac{1}{(1+\varphi(x)^2)^{\tfrac{d}{2}}}\dd x 
=\int_{H_{\phi_{p-1,L},p-1}}\frac{1}{(1+x_{k}^2)^{\tfrac{d}{2}}}\dd x \nonumber \\ 
&= 2^{d-m+|\mathcal{M}|}\int_{[0,x_{k}+\phi_{p-1,L}(x_{k})]^{k-1}\times \big[\tfrac{L}{p},\infty\big[\,\times [0,x_{k}+\phi_{p-1,L}(x_{k})]^{m-k}\times [0,L+\phi_{p-1,L}(x_{k})]^{d-m}}\frac{1}{(1+x_{k}^2)^{\tfrac{d}{2}}}\dd x.
\end{align}
With the help of \eqref{auxiliary_estimate_phi}, the right-hand side of \eqref{HS_higher_order_int_5} is bounded from above by  
\begin{multline}
2^{2d-m}\int_{\tfrac{L}{p}}^{\infty}\frac{(px_{k})^{m-1}\Big(L^{\big(\tfrac{d-m}{d}\big)^{p-1}-\epsilon} (px_{k})^{1-\big(\tfrac{d-m}{d}\big)^{p-1}+\epsilon}\Big)^{d-m}}{x_{k}^d}\dd x_{k} \\
\leq 2^{2d-m}p^{d-1}\int_{\tfrac{L}{p}}^{\infty}L^{\tfrac{(d-m)^p}{d^{p-1}}-(d-m)\epsilon}x_k^{-1-\tfrac{(d-m)^p}{d^{p-1}}+(d-m)\epsilon}\dd x_{k}.
\end{multline}
Carrying out the integration, we obtain the bound
\begin{equation}
\int_{M}\frac{1}{(1+\varphi(x)^2)^{\tfrac{d}{2}}}\dd x\leq C_{p-1}\frac{p^{d-1}}{\epsilon}, 
\end{equation}
with a constant $C_{p-1}>0$ independent of $L$ and $p$.
Therefore, Lemma \ref{Lemma_kernel_set} yields 
\begin{equation}\label{operator-HS-estimate-3-L}
\|\mathbf{1}_{H_{\phi_{p+1}}}X\mathbf{1}_{H_{\mathcal{M}\cup\{k\}}\setminus H_{\mathcal{M}}}\|_2\leq C_{p-1}'\sqrt{\frac{p^{d-1}}{\epsilon}},
\end{equation}
with a constant $C_{p-1}'>0$ independent of $L$ and $p$.
The proof of \eqref{operator-difference-HS-estimate} follows in the same manner as the proof \eqref{operator-HS-estimate}: We first estimate 
\begin{align}
\big\|\big(&\mathbf{1}_{H_{\infty,L,m,k}}-\mathbf{1}_{H_{L,m,k}}\big)(\mathbf{1}_{H_{\mathcal{M}}}X)^{p}(\mathbf{1}_{H_{\mathcal{M}\cup\{k\}}}-\mathbf{1}_{H_{\mathcal{M}}})
\big\|_2 \nonumber \\ 
\leq& \, \sum_{l=0}^{p-2}\Big\|\big(\mathbf{1}_{H_{\infty,L,m,k}}-\mathbf{1}_{H_{L,m,k}}\big)\Big(\prod_{j=1}^{l}X\mathbf{1}_{H_{\phi_{j,L},j}}\Big)X\mathbf{1}_{H_{\mathcal{M}}\setminus H_{\phi_{l+1,L},l+1}}(X\mathbf{1}_{H_{\mathcal{M}}})^{p-2-l}X(\mathbf{1}_{H_{\mathcal{M}\cup\{k\}}}-\mathbf{1}_{H_{\mathcal{M}}})\Big\|_2 \nonumber \\ 
&+ \Big\|\big(\mathbf{1}_{H_{\infty,L,m,k}}-\mathbf{1}_{H_{L,m,k}}\big)\Big(\prod_{j=1}^{p-1}X\mathbf{1}_{H_{\phi_{j,L},j}}\Big)X(\mathbf{1}_{H_{\mathcal{M}\cup\{k\}}}-\mathbf{1}_{H_{\mathcal{M}}})\Big\|_2
\end{align}
which in turn is bounded from above by
\begin{align}
\|X\|^{p-1}\Big[& \big\|\big(\mathbf{1}_{H_{\infty,L,m,k}}-\mathbf{1}_{H_{L,m,k}}\big)X\mathbf{1}_{H_{\mathcal{M}}\setminus H_{\phi_{1,L},1}}\big\|_2 \nonumber \\
&  +\sum_{l=1}^{p-2}\big\|\mathbf{1}_{H_{\phi_{l,L},l}}X\mathbf{1}_{H_{\mathcal{M}}\setminus H_{\phi_{l+1,L},l+1}}\big\|_2+\big\|\mathbf{1}_{H_{\phi_{p-1,L},p-1}}X\mathbf{1}_{H_{\mathcal{M}\cup\{k\}}\setminus H_{\mathcal{M}}}\big\|_2 \Big] \nonumber\\
& \qquad \leq \,  C' \|X\|^{p-1}p\sqrt{\frac{p^{2d}}{\epsilon}}\leq C' \|X\|^{p-1}d^{\tfrac{p}{2}}p^{d+1},
\end{align}
where we combined estimates \eqref{operator-HS-estimate-1-L}, \eqref{operator-HS-estimate-2-L} and \eqref{operator-HS-estimate-3-L} and used the definition of $\epsilon$ in the last line. This proves \eqref{operator-difference-HS-estimate} for $p\geq 2$ and concludes the proof of the lemma.
\end{proof}
We now use the just obtained Hilbert-Schmidt estimate, to derive a trace norm estimate for monomial test functions.
\begin{lem}\label{higher-order-monom}
Let $X$ be a bounded integral operator on $L^{2}(\R^{d})\otimes\C^n$ whose kernel satisfies the bound \eqref{kernel_bound}. Let $1\leq k \leq m < d$ and $\mathcal{M}\subseteq\{1,\ldots,m\}$ such that $k\notin\mathcal{M}$. Let $L\geq 1$ and $p\in\N$. Then the operator 
\begin{equation}\label{operator-trace-proof}
\mathbf{1}_{H_{\infty,L,m,k}}\big[\big(\mathbf{1}_{H_{\mathcal{M}}}X\mathbf{1}_{H_{\mathcal{M}}}\big)^{p}-\big(\mathbf{1}_{H_{\mathcal{M}\cup\{k\}}}X\mathbf{1}_{H_{\mathcal{M}\cup\{k\}}}\big)^{p}\big]\mathbf{1}_{H_{\infty,L,m,k}}
\end{equation}
is trace class and there exists a constant $C>0$, independent of $L$ and $p$, such that
\begin{multline}\label{operator-trace-estimate}
\big\|\mathbf{1}_{H_{\infty,L,m,k}}\big[\big(\mathbf{1}_{H_{\mathcal{M}}}X\mathbf{1}_{H_{\mathcal{M}}}\big)^{p}-\big(\mathbf{1}_{H_{\mathcal{M}\cup\{k\}}}X\mathbf{1}_{H_{\mathcal{M}\cup\{k\}}}\big)^{p}\big]\mathbf{1}_{H_{\infty,L,m,k}}\big\|_1 \\
\leq C L^{d-m}\|X\|^{p-1}\big(2d\big)^{p-1}(p-1)^{2d+2}.
\end{multline}
Furthermore, there exists a constant $C'>0$ which is independent of $L$ and $p$ such that
\begin{multline}\label{operator-difference-trace-estimate}
\big\|\big(\mathbf{1}_{H_{\infty,L,m,k}}-\mathbf{1}_{H_{L,m,k}}\big)\big[\big(\mathbf{1}_{H_{\mathcal{M}}}X\mathbf{1}_{H_{\mathcal{M}}}\big)^{p}-\big(\mathbf{1}_{H_{\mathcal{M}\cup\{k\}}}X\mathbf{1}_{H_{\mathcal{M}\cup\{k\}}}\big)^{p}\big]\big(\mathbf{1}_{H_{\infty,L,m,k}}-\mathbf{1}_{H_{L,m,k}}\big)\big\|_1 \\
\leq C'\|X\|^{p-1}\big(2d\big)^{p-1}(p-1)^{2d+2}.
\end{multline}
\end{lem}
\begin{proof}
We start by proving \eqref{operator-trace-estimate}. The proof works in a similar way as the proof of Lemma \ref{Lemma-localisation}.
Set $P_{0}:=\mathbf{1}_{H_{\mathcal{M}}}$ and $P_{1}:=\mathbf{1}_{H_{\mathcal{M}\cup\{k\}}}$. Using the fact that $\mathbf{1}_{H_{\infty,L,m,k}}P_{0}=\mathbf{1}_{H_{\infty,L,m,k}}P_{1}=\mathbf{1}_{H_{\infty,L,m,k}}$, we rewrite the operator \eqref{operator-trace-proof} in the following way
\begin{equation}\label{operator-trace-proof-step-1}
\mathbf{1}_{H_{\infty,L,m,k}}\big[\big(XP_{0}\big)^{p-1}-\big(XP_{1}\big)^{p-1}\big]X\mathbf{1}_{H_{\infty,L,m,k}}
\end{equation}
and see that in the case $p=1$ there is nothing to show. If $p\geq 2$, we define $P_{2}=P_{1}-P_{0}$ and write
\begin{align}
\mathbf{1}_{H_{\infty,L,m,k}}\big(XP_{1}\big)^{p-1}X\mathbf{1}_{H_{\infty,L,m,k}}
&=\mathbf{1}_{H_{\infty,L,m,k}}\big(XP_{0}+XP_{2}\big)^{p-1}X\mathbf{1}_{H_{\infty,L,m,k}} \nonumber\\
&=\sum_{\pi=(\pi_{1},\ldots,\pi_{p-1})\in \{0,2\}^{p-1}}\mathbf{1}_{H_{\infty,L,m,k}}\Big(\prod_{j=1}^{p-1}XP_{\pi(j)}\Big)X\mathbf{1}_{H_{\infty,L,m,k}}.
\end{align}
With this at hand, we estimate the trace norm of \eqref{operator-trace-proof-step-1} by 
\begin{equation}\label{operator-trace-proof-step-2}
\sum_{\pi\in \{0,2\}^{p-1}: \ \pi\neq 0} \Big\|\mathbf{1}_{H_{\infty,L,m,k}}\Big(\prod_{j=1}^{p-1}XP_{\pi(j)}\Big)X\mathbf{1}_{H_{\infty,L,m,k}}\Big\|_1.
\end{equation}
Each of these trace norms contains at least one factor $P_{2}$. Setting $p_{1,\pi}:=\min \{j\in\{1,\ldots,p-1\}: \ \pi(j)=2\}$ and $p_{2,\pi}:=\max \{j\in\{1,\ldots,p-1\}: \ \pi(j)=2\}$ for a given $\pi\in\{0,2\}^{p-1}$ with $\pi\neq 0$, we estimate \eqref{operator-trace-proof-step-2} by
\begin{equation}\label{operator-trace-proof-step-3}
\sum_{\pi\in \{0,2\}^{p-1}: \ \pi\neq 0} \|X\|^{p-1-p_{1,\pi}-p_{2,\pi}}\Big\|\mathbf{1}_{H_{\infty,L,m,k}}\big(P_{0}X\big)^{p_{1,\pi}}P_{2}\Big\|_2\Big\|P_{2}\big(XP_{0}\big)^{p_{2,\pi}}\mathbf{1}_{H_{\infty,L,m,k}}\Big\|_2.
\end{equation}
By Lemma \ref{Lemma-HS-decay}, there exists a constant $C>0$, independent of $L$ and $p$, such that \eqref{operator-trace-proof-step-3} is bounded by
\begin{multline}
C L^{d-m}\|X\|^{p-1}\sum_{\pi\in \{0,2\}^{p-1}: \ \pi\neq 0}d^{\tfrac{p_{1,\pi}+p_{2,\pi}}{2}}(p_{1,\pi})^{d+1}(p_{2,\pi})^{d+1} \\
\leq C L^{d-m} \|X\|^{p-1}\big(2d\big)^{p-1}(p-1)^{2d+2}.
\end{multline}
This concludes the proof of \eqref{operator-trace-estimate}. In order to prove \eqref{operator-difference-trace-estimate}, we use an analogous argument to estimate the left-hand side of \eqref{operator-difference-trace-estimate} by
\begin{multline}\label{operator-trace-proof-step-4}
\sum_{\pi\in \{0,2\}^{p-1}: \ \pi\neq 0} \|X\|^{p-1-p_{1,\pi}-p_{2,\pi}}\Big\|\big(\mathbf{1}_{H_{\infty,L,m,k}}-\mathbf{1}_{H_{L,m,k}}\big)\big(P_{0}X\big)^{p_{1,\pi}}P_{2}\Big\|_2  \\
\times \quad \Big\|P_{2}\big(XP_{0}\big)^{p_{2,\pi}}\big(\mathbf{1}_{H_{\infty,L,m,k}}-\mathbf{1}_{H_{L,m,k}}\big)\Big\|_2.
\end{multline}
Again Lemma \ref{Lemma-HS-decay} yields a constant $C'>0$, independent of $L$ and $p$, such that \eqref{operator-trace-proof-step-4} is bounded by
\begin{equation}
C'\|X\|^{p-1}\sum_{\pi\in \{0,2\}^{p-1}: \ \pi\neq 0}d^{\tfrac{p_{1,\pi}+p_{2,\pi}}{2}}(p_{1,\pi})^{d+1}(p_{2,\pi})^{d+1}\leq C' \|X\|^{p-1}\big(2d\big)^{p-1}(p-1)^{2d+2}.
\end{equation}
This concludes the proof of \eqref{operator-difference-trace-estimate}.
\end{proof}
It remains to extend this estimate to analytic functions, which we do in the following 
\begin{lem}\label{higher-order-analytic}
Let $X$ be a bounded integral operator on $L^{2}(\R^{d})\otimes\C^n$ whose kernel satisfies the bound \eqref{kernel_bound}. Let $1\leq k \leq m < d$ and $\mathcal{M}\subseteq\{1,\ldots,m\}$ such that $k\notin\mathcal{M}$. Let $L\geq 1$ and $g$ be an entire function such that $g(0)=0$. Then the operator 
\begin{equation}\label{analytic-operator-trace-proof}
\mathbf{1}_{H_{\infty,L,m,k}}\big[g\big(\mathbf{1}_{H_{\mathcal{M}}}X\mathbf{1}_{H_{\mathcal{M}}}\big)-g\big(\mathbf{1}_{H_{\mathcal{M}\cup\{k\}}}X\mathbf{1}_{H_{\mathcal{M}\cup\{k\}}}\big)\big]\mathbf{1}_{H_{\infty,L,m,k}}
\end{equation}
is trace class and there exists a constant $C>0$, which is independent of $L$, such that
\begin{multline}\label{analytic-operator-difference-trace-estimate}
\big\|\big(\mathbf{1}_{H_{\infty,L,m,k}}-\mathbf{1}_{H_{L,m,k}}\big)\big[g\big(\mathbf{1}_{H_{\mathcal{M}}}X\mathbf{1}_{H_{\mathcal{M}}}\big)-g\big(\mathbf{1}_{H_{\mathcal{M}\cup\{k\}}}X\mathbf{1}_{H_{\mathcal{M}\cup\{k\}}}\big)\big]\big(\mathbf{1}_{H_{\infty,L,m,k}}-\mathbf{1}_{H_{L,m,k}}\big)\big\|_1 
\leq C.
\end{multline}
\end{lem}
\begin{proof}
We begin by proving \eqref{analytic-operator-trace-proof}. There exists a natural number $N_{0}\in\N$ such that $p^{2d+2}\leq 2^{p}$ for all $p\geq N_{0}$. Using Lemma \ref{higher-order-monom}, we write 
\begin{multline}
\big\|\mathbf{1}_{H_{\infty,L,m,k}}\big[g\big(\mathbf{1}_{H_{\mathcal{M}}}X\mathbf{1}_{H_{\mathcal{M}}}\big)-g\big(\mathbf{1}_{H_{\mathcal{M}\cup\{k\}}}X\mathbf{1}_{H_{\mathcal{M}\cup\{k\}}}\big)\big]\mathbf{1}_{H_{\infty,L,m,k}}\big\|_1 \\
=\big\|\sum_{p=1}^{\infty}\omega_{p}\big\{\mathbf{1}_{H_{\infty,L,m,k}}\big[\big(\mathbf{1}_{H_{\mathcal{M}}}X\mathbf{1}_{H_{\mathcal{M}}}\big)^{p}-\big(\mathbf{1}_{H_{\mathcal{M}\cup\{k\}}}X\mathbf{1}_{H_{\mathcal{M}\cup\{k\}}}\big)^{p}\big]\mathbf{1}_{H_{\infty,L,m,k}}\big\}\big\|_1 \\
\leq  C_{1}L^{d-m}+C_{2}L^{d-m}\sum_{p=N_{0}}^{\infty}|\omega_{p}|\big(4d\|X\|\big)^{p-1}\leq C_{3}L^{d-m},
\end{multline}
with constants $C_{1}, C_2, C_3>0$ independent of $L$.
In order to prove \eqref{analytic-operator-difference-trace-estimate}, we again use Lemma \ref{higher-order-monom} to estimate
\begin{multline}
\big\|\big(\mathbf{1}_{H_{\infty,L,m,k}}-\mathbf{1}_{H_{L,m,k}}\big)\big[g\big(\mathbf{1}_{H_{\mathcal{M}}}X\mathbf{1}_{H_{\mathcal{M}}}\big)-g\big(\mathbf{1}_{H_{\mathcal{M}\cup\{k\}}}X\mathbf{1}_{H_{\mathcal{M}\cup\{k\}}}\big)\big]\big(\mathbf{1}_{H_{\infty,L,m,k}}-\mathbf{1}_{H_{L,m,k}}\big)\big\|_1  \\
\leq  C_{4}+C_{5}\sum_{p=N_{0}}^{\infty}|\omega_{p}|\big(4d\|X\|\big)^{p-1}\leq C,
\end{multline}
with constants $C_{4}, C_5, C>0$ independent of $L$. This concludes the proof of the lemma.
\end{proof}
\subsection{Proof of Theorem \ref{theorem_higher_order_terms}}
\label{subsec_higher_order_proof}

With this at hand, we are now ready to prove Theorem \ref{theorem_higher_order_terms}.
\begin{proof}[Proof of Theorem \ref{theorem_higher_order_terms}]
We start by proving \eqref{theorem_higher_order_terms_general}. The asymptotic expansion \eqref{theorem_higher_order_terms_special} follows by counting the number of occurring terms. After the localisation in Lemma \ref{Lemma-localisation}, the expansion \eqref{theorem_higher_order_terms_general} reduces to analysing the trace
\begin{equation}
\sum_{V\in\mathcal{F}^{(0)}}\mathbf{1}_{H_{V,L}}\sum_{m=0}^{d-1}\sum_{F\in\mathcal{F}_{V}^{(d-m)}}X_{LF,g}.
\end{equation}
As explained in the strategy of the proof at the beginning of Section \ref{subsec_HS_higher_order}, this analysis, and with it the proof of \eqref{theorem_higher_order_terms_general}, reduces to showing  that
\begin{multline}\label{theorem-higher-order-step-1}
\sum_{m=0}^{d-1}\sum_{V\in\mathcal{F}^{(0)}}\sum_{F\in\mathcal{F}_{V}^{(d-m)}}\tr_{L^{2}(\R^{d})\otimes\C^n}\big[\mathbf{1}_{H_{V,L}}X_{LF,g}\big]\\
=\sum_{m=0}^{d-1}L^{d-m}\lim_{L\rightarrow\infty}\sum_{V\in\mathcal{F}^{(0)}}\sum_{F\in\mathcal{F}_{V}^{(d-m)}}\tr_{L^{2}(\R^{d})\otimes\C^n}\big[\mathbf{1}_{H_{F,L,V}}X_{F,g}\big]+O(1),
\end{multline}
as $L\rightarrow\infty$ (cf. \eqref{theorem_higher_order_terms_goal_1}). Therefore, we begin by showing that, for a given vertex $V$ and a given $(d-m)$-face $F$, we have
\begin{equation}\label{theorem-higher-order-step-2}
\tr_{L^{2}(\R^{d})\otimes\C^n}\big[\mathbf{1}_{H_{V,L}}X_{LF,g}\big]=L^{d-m}\lim_{L\rightarrow\infty}\tr_{L^{2}(\R^{d})\otimes\C^n}\big[\mathbf{1}_{H_{F,L,V}}X_{F,g}\big]+O(1),
\end{equation}
as $L\rightarrow\infty$. After suitable rotation and translation we can again consider the situation with the designated vertex $V=\{0\}$ and $(d-m)$-face $F=\{0\}^{m}\times [0,2]^{d-m}$. We recall that in this case we have
\begin{equation}
H_{V,L}=[0,L]^{d},   \quad  H_{F,L,V}=[0,L]^{m}\times[0,1]^{d-m} \ \ \text{and} \ \ X_{LF,g}=X_{F,g}.
\end{equation}

In the easy case $m=0$, the operator $X_{F,g}=g(X)$ is translation invariant. Then \eqref{theorem-higher-order-step-2} follows by the continuity of the kernel of $X$. In the other cases we apply the results of Section \ref{subsec_HS_higher_order}. 
Let $k\in\{1,\ldots,m\}$ and $\mathcal{M}\subseteq\{1,\ldots,m\}\setminus\{k\}$ be arbitrary. Comparing the operators \eqref{theorem_higher_order_terms_goal_2} and \eqref{theorem_higher_order_terms_goal_3}, we first show that
\begin{multline}\label{theorem-higher-order-step-3}
\tr_{L^{2}(\R^{d})\otimes\C^n}\Big[\mathbf{1}_{H_{L,m,k}}\Big(g\big(\mathbf{1}_{H_{\mathcal{M}}}X\mathbf{1}_{H_{\mathcal{M}}}\big)-g\big(\mathbf{1}_{H_{\mathcal{M}\cup\{k\}}}X\mathbf{1}_{H_{\mathcal{M}\cup\{k\}}}\big)\Big)\Big] \\
= \, \tr_{L^{2}(\R^{d})\otimes\C^n}\Big[\mathbf{1}_{H_{\infty,L,m,k}}\Big(g\big(\mathbf{1}_{H_{\mathcal{M}}}X\mathbf{1}_{H_{\mathcal{M}}}\big)-g\big(\mathbf{1}_{H_{\mathcal{M}\cup\{k\}}}X\mathbf{1}_{H_{\mathcal{M}\cup\{k\}}}\big)\Big)\mathbf{1}_{H_{\infty,L,m,k}}\Big]  +O(1),
\end{multline}
as $L\rightarrow\infty$. The operator 
\begin{equation}
\mathbf{1}_{H_{\infty,L,m,k}}\Big(g\big(\mathbf{1}_{H_{\mathcal{M}}}X\mathbf{1}_{H_{\mathcal{M}}}\big)-g\big(\mathbf{1}_{H_{\mathcal{M}\cup\{k\}}}X\mathbf{1}_{H_{\mathcal{M}\cup\{k\}}}\big)\Big)\mathbf{1}_{H_{\infty,L,m,k}}
\end{equation} 
is trace class by Lemma \ref{higher-order-analytic}. By the cyclic property of the trace and the fact that $H_{L,m,k}\subset H_{\infty,L,m,k}$, we see that the operators 
\begin{multline}
\mathbf{1}_{H_{\infty,L,m,k}}\Big(g\big(\mathbf{1}_{H_{\mathcal{M}}}X\mathbf{1}_{H_{\mathcal{M}}}\big)-g\big(\mathbf{1}_{H_{\mathcal{M}\cup\{k\}}}X\mathbf{1}_{H_{\mathcal{M}\cup\{k\}}}\big)\Big)\mathbf{1}_{H_{\infty,L,m,k}}  \\
-\mathbf{1}_{H_{L,m,k}}\Big(g\big(\mathbf{1}_{H_{\mathcal{M}}}X\mathbf{1}_{H_{\mathcal{M}}}\big)-g\big(\mathbf{1}_{H_{\mathcal{M}\cup\{k\}}}X\mathbf{1}_{H_{\mathcal{M}\cup\{k\}}}\big)\Big)
\end{multline}
and
\begin{equation}
\big(\mathbf{1}_{H_{\infty,L,m,k}}-\mathbf{1}_{H_{L,m,k}}\big)\Big(g\big(\mathbf{1}_{H_{\mathcal{M}}}X\mathbf{1}_{H_{\mathcal{M}}}\big)-g\big(\mathbf{1}_{H_{\mathcal{M}\cup\{k\}}}X\mathbf{1}_{H_{\mathcal{M}\cup\{k\}}}\big)\Big)\big(\mathbf{1}_{H_{\infty,L,m,k}}-\mathbf{1}_{H_{L,m,k}}\big)
\end{equation}
have the same trace.
Now, Lemma \ref{higher-order-analytic} yields a constant $C>0$, independent of $L$, such that for all $L\geq 1$ we have 
\begin{equation}
\big\|\big(\mathbf{1}_{H_{\infty,L,m,k}}-\mathbf{1}_{H_{L,m,k}}\big)\Big(g\big(\mathbf{1}_{H_{\mathcal{M}}}X\mathbf{1}_{H_{\mathcal{M}}}\big)-g\big(\mathbf{1}_{H_{\mathcal{M}\cup\{k\}}}X\mathbf{1}_{H_{\mathcal{M}\cup\{k\}}}\big)\Big)\big(\mathbf{1}_{H_{\infty,L,m,k}}-\mathbf{1}_{H_{L,m,k}}\big)\big\|_1 \leq C.
\end{equation}

As the operators $g\big(\mathbf{1}_{H_{\mathcal{M}}}X\mathbf{1}_{H_{\mathcal{M}}}\big)$ and $g\big(\mathbf{1}_{H_{\mathcal{M}\cup\{k\}}}X\mathbf{1}_{H_{\mathcal{M}\cup\{k\}}}\big)$
are translation invariant in the last $(d-m)$ variables with continuous kernel, we obtain
\begin{multline}\label{theorem-higher-order-step-4}
\tr_{L^{2}(\R^{d})\otimes\C^n}\Big[\mathbf{1}_{H_{\infty,L,m,k}}\Big(g\big(\mathbf{1}_{H_{\mathcal{M}}}X\mathbf{1}_{H_{\mathcal{M}}}\big)-g\big(\mathbf{1}_{H_{\mathcal{M}\cup\{k\}}}X\mathbf{1}_{H_{\mathcal{M}\cup\{k\}}}\big)\Big)\mathbf{1}_{H_{\infty,L,m,k}}\Big] \\
= L^{d-m} \tr_{L^{2}(\R^{d})\otimes\C^n}\Big[\mathbf{1}_{H_{\infty,1,m,k}}\Big(g\big(\mathbf{1}_{H_{\mathcal{M}}}X\mathbf{1}_{H_{\mathcal{M}}}\big)-g\big(\mathbf{1}_{H_{\mathcal{M}\cup\{k\}}}X\mathbf{1}_{H_{\mathcal{M}\cup\{k\}}}\big)\Big)\mathbf{1}_{H_{\infty,1,m,k}}\Big] \\
= L^{d-m} \lim_{L\rightarrow\infty} \tr_{L^{2}(\R^{d})\otimes\C^n}\Big[\mathbf{1}_{H_{L,1,m,k}}\Big(g\big(\mathbf{1}_{H_{\mathcal{M}}}X\mathbf{1}_{H_{\mathcal{M}}}\big)-g\big(\mathbf{1}_{H_{\mathcal{M}\cup\{k\}}}X\mathbf{1}_{H_{\mathcal{M}\cup\{k\}}}\big)\Big)\Big], 
\end{multline}
where $H_{\infty,1,m,k}=\{x\in \R_+^m\times [0,1]^{d-m}: \ \forall j\in \{1,\ldots,m\}\setminus \{k\}: \ x_{j}\leq x_{k}\}$, $H_{L,1,m,k}:=\{x\in [0,L]^m\times [0,1]^{d-m}: \ \forall j\in \{1,\ldots,m\}\setminus \{k\}: \ x_{j}\leq x_{k}\}$ and we again used the cyclic property of the trace. The operator 
\begin{equation}
\mathbf{1}_{H_{\infty,1,m,k}}\Big(g\big(\mathbf{1}_{H_{\mathcal{M}}}X\mathbf{1}_{H_{\mathcal{M}}}\big)-g\big(\mathbf{1}_{H_{\mathcal{M}\cup\{k\}}}X\mathbf{1}_{H_{\mathcal{M}\cup\{k\}}}\big)\Big)\mathbf{1}_{H_{\infty,1,m,k}}
\end{equation}
is trace class by Lemma \ref{higher-order-analytic} with $L=1$.

Combining \eqref{theorem-higher-order-step-3} and \eqref{theorem-higher-order-step-4}, adding the contributions for every $k\in\{1,\ldots,m\}$ and $\mathcal{M}\subseteq\{1,\ldots,m\}\setminus\{k\}$, as well as reordering the terms, we obtain
\begin{multline}\label{trace-translation-operator}
\tr_{L^{2}(\R^{d})\otimes\C^n}\Bigg[\sum_{k=0}^{m}(-1)^{k}\sum_{\mathcal{M}\subseteq \{1,\ldots,m\}\ : \ |\mathcal{M}|=k}\mathbf{1}_{H_{V,L}}g(\mathbf{1}_{H_{\mathcal{M}}}X\mathbf{1}_{H_{\mathcal{M}}})\Bigg] \\
=L^{d-m} \lim_{L\rightarrow\infty}\tr_{L^{2}(\R^{d})\otimes\C^n}\Bigg[\sum_{k=0}^{m}(-1)^{k}\sum_{\mathcal{M}\subseteq \{1,\ldots,m\}\ : \ |\mathcal{M}|=k}\mathbf{1}_{H_{F,L,V}}g(\mathbf{1}_{H_{\mathcal{M}}}X\mathbf{1}_{H_{\mathcal{M}}})\Bigg] 
+O(1),
\end{multline}
as $L\rightarrow\infty$. Together with \eqref{X_F,g-alternative} this yields \eqref{theorem-higher-order-step-2} for the designated vertex $V=\{0\}$ and $(d-m)$-face $F=\{0\}^{m}\times [0,2]^{d-m}$.
Rotating and translating back, we obtain \eqref{theorem-higher-order-step-2} for an arbitrary vertex $V$ and $(d-m)$-face $F$. Adding up the contributions of all $V$ and $F$ yields \eqref{theorem-higher-order-step-1}. Recollecting the contributions of all vertices in $F$ via the equality $H_{F,L}=\sum_{\{V\in\mathcal{F}^{(0)}:V\in F\}}H_{F,L,V}$, finishes the proof of \eqref{theorem_higher_order_terms_general}.

In order to prove \eqref{theorem_higher_order_terms_special}, it just remains to count the number of contributions in \eqref{theorem_higher_order_terms_general}. Clearly we have $\big|\mathcal{F}^{(d-m)}\big|=2^m\frac{d!}{m!(d-m)!}$ and $\big|\mathcal{F}_{F}^{(k+(d-m))}\big|=\frac{m!}{k!(m-k)!}$ for a given $(d-m)$-face $F$ and $k\in\{0,\ldots,m\}$. The latter corresponds to the number of operators occurring in the definition \eqref{def-face-operator} of $X_{F,g}$. Therefore, we obtain the constants 
$c_{k,m}=2^m\frac{d!}{m!(d-m)!}(-1)^{k}\frac{m!}{k!(m-k)!}=\frac{(-1)^{k}2^{m}d!}{k!(m-k)!(d-m)!}$ for $0\leq k\leq m\leq d$. This concludes the proof of \eqref{theorem_higher_order_terms_special} and of Theorem \ref{theorem_higher_order_terms}.
\end{proof}


\section{Estimates for smooth symbols and an upper bound}
\label{sec_upper_bound}
In this section we further study the trace of the remaining operator 
\begin{equation}\label{rem_op}
\sum_{V\in\mathcal{F}^{(0)}}\mathbf{1}_{H_{V,L}}X_{LV,g}.
\end{equation}
From here on, we only consider the special case of the free massless Dirac operator, given by our application, i.e. we set $X=\Op(\psi\mathcal{D})$. As in Section \ref{sec_higher_order} we choose a designated vertex $V=V_0=\{0\}$ such that $H_{V,L}=H_{V_0,L}=[0,L]^{d}$ and
\begin{equation}\label{def_op_designated_corner}
X_g:=X_{LV_0,g}=\sum_{k=0}^{d}(-1)^{k}\sum_{\mathcal{M}\subseteq \{1,\ldots,d\}\ : \ |\mathcal{M}|=k}g(\mathbf{1}_{H_{\mathcal{M}}}\Op(\psi\mathcal{D})\mathbf{1}_{H_{\mathcal{M}}}),
\end{equation}
where we recall $H_{\mathcal{M}}=\{x\in\R^d: \ \forall j\in \{1,\ldots,d\}\setminus \mathcal{M}: \ x_{j}\geq 0\}$ from the last section. The other vertices reduce to this case after suitable rotation and translation. In order to derive a logarithmic upper bound for the trace norm of \eqref{rem_op}, we make use of the fact that the symbol $\psi\mathcal{D}$ is discontinuous at a single point, the origin, and is smooth away from it, as well as compactly supported. We follow the usual strategy, for example laid out in \cite{sobolevschatten}, to first establish "non-enhanced" bounds for symbols which are both smooth and compactly supported and then to obtain a logarithmically enhanced bound for the discontinuous symbol from there.

The estimate for smooth and compactly supported symbols utilised in this article is based on the estimates in \cite{sobolevschatten} and is similar, in the choice of sets and proof, to the result \cite[Cor. 3.7]{Pfirsch}.  
\begin{lem}\label{Lemma-trace-class-sob-prep}
Let $\psi\in\C_{c}^\infty(\R^d)$ be supported in a ball of radius $\rho\geq 1$ centred about some $\xi_0\in\R^d$. Let $M,N\subset\R^d$ be such that there exists $\beta\geq 0$ and a constant $C_\beta>0$ with the property that, for all $r>0$, we have
\begin{equation}\label{Lemma-trace-class-sob-prep-req}
|\{x\in \Z^{d}: \ Q_x\cap M_\rho\neq \emptyset \ \text{and} \ \dist(x,N_\rho)\leq r\}|\leq C_{\beta}(1+r^2)^{\tfrac{\beta}{2}},
\end{equation}
where $Q_x$ is the closed unit cube centred about $x$, and $M_\rho$ respectively $N_\rho$ are the scaled versions of the respective sets.
Then the operator $\mathbf{1}_M\Op(\psi\otimes\mathbb{1}_{n}) \mathbf{1}_N$ is trace class and there exists a constant $C>0$, which is independent of $\rho$, $M$ and $N$, such that
\begin{equation}
\|\mathbf{1}_M\Op(\psi\otimes\mathbb{1}_{n}) \mathbf{1}_N\|_1\leq C_{\beta}C\max_{|\alpha|\leq 2d+\beta+1}\sup_{\xi\in\R^d}\rho^{|\alpha|}|\partial^{\alpha}_{\xi}\psi(\xi)|,
\end{equation}
where $|\alpha| := \sum_{k=1}^{d}\alpha_{k}$ for a given multi-index $\alpha\in\N_0^d$.
\end{lem}
\begin{proof}
By the scaling argument used in the proof of \cite[Thm.~3.1]{sobolevschatten} and the fact that the bound \eqref{Lemma-trace-class-sob-prep-req} holds for the scaled sets $M_\rho$, $N_\rho$, it suffices to prove the lemma in the case $\rho=1$.
For $x\in\Z^{d}$ and arbitrary $m\in\N$, we infer from \cite[Thm.~3.1,~3.2]{sobolevschatten} the bound 
\begin{equation}
\|1_{Q_{x}\cap M} \Op(\psi)1_N\|_1 \leq C' (1+\dist(x,N)^2)^{-\tfrac{m}{2}} \max_{|\alpha|\leq m+d}\sup_{\xi\in\R^d}|\partial^{\alpha}_{\xi}\psi(\xi)|,
\end{equation}
where the
constant $C'>0$ is independent of $x$, $M$ and $N$. The bound easily extends to the matrix-valued symbol $\psi\otimes \mathbb{1}_{n}$. Choosing $m=d+\beta+1$, we estimate
\begin{align}
\|\mathbf{1}_M\Op(\psi\otimes\mathbb{1}_{n}) \mathbf{1}_N\|_1 &\leq \sum_{a\in\Z^{d}}\|\mathbf{1}_{Q_{a}\cap M}\Op(\psi\otimes\mathbb{1}_{n}) \mathbf{1}_N\|_1 \nonumber \\
&\leq C'\max_{|\alpha|\leq 2d+\beta+1}\sup_{\xi\in\R^d}|\partial^{\alpha}_{\xi}\psi(\xi)|\sum_{k=0}^\infty \quad \sum_{\substack{a\in\Z^{d}: \ Q_{a}\cap M\neq \emptyset \\ k\leq \dist(a,N)\leq k+1}} (1+k^2)^{\tfrac{-d-\beta-1}{2}} \nonumber \\
&\leq  C_{\beta}C'\max_{|\alpha|\leq 2d+\beta+1}\sup_{\xi\in\R^d}|\partial^{\alpha}_{\xi}\psi(\xi)|\sum_{k=0}^\infty (1+(k+1)^2)^{\tfrac{\beta}{2}}(1+k^2)^{\tfrac{-d-\beta-1}{2}} \nonumber \\
&=:C_{\beta}C\max_{|\alpha|\leq 2d+\beta+1}\sup_{\xi\in\R^d}|\partial^{\alpha}_{\xi}\psi(\xi)|.
\end{align}
The constant $C$ is independent of $M$ and $N$.
\end{proof}
In Section \ref{sec_commutation} it will often be necessary to obtain bounds for operators with symbol given by a power of $\psi$ and control the dependence on the corresponding exponent. This is the content of the following
\begin{cor}\label{Lemma-trace-class-sob}
Let $\psi\in\C_{c}^\infty(\R^d)$ be supported in a ball of radius $b+1$ centred about the origin and let $p\in\N$. Let $M,N\subseteq\R^d$ be as in Lemma \ref{Lemma-trace-class-sob-prep}.
Then the operator $\mathbf{1}_M\Op(\psi^p) \mathbf{1}_N$ is trace class and there exists a constant $C>0$, which is independent of $p$, $M$ and $N$, such that
\begin{equation}
\|\mathbf{1}_M\Op(\psi^{p}\otimes\mathbb{1}_{n}) \mathbf{1}_N\|_1\leq C_{\beta}Cp^{2d+\beta+1}.
\end{equation}
\end{cor}
\begin{proof}
For the function $\psi^p$, we compute
\begin{equation}
\max_{|\alpha|\leq 2d+\beta+1}\sup_{\xi\in\R^d}(b+1)^{|\alpha|}|\partial^{\alpha}_{\xi}\psi^p(\xi)| \leq C'p^{2d+\beta+1}
\end{equation}
with a constant $C'>0$ independent of $p$, $M$ and $N$. Therefore, the desired bound follows.
\end{proof}

In order to prove an upper bound for the remaining operator, we first need to obtain an estimate for the non-smooth Dirac symbol from the just obtained estimate for smooth symbols. The proof follows the ideas in \cite{sobolevschatten} and overlaps to a large extent with the proof of \cite[Lemma 3.8]{BM-inprep}. Therefore, the proof given here is kept a bit shorter. See \cite[Lemma 3.8]{BM-inprep} for a more detailed version of some of the arguments.
\begin{lem}\label{log-L-Estimate}
Let $\psi\in\C_{c}^\infty(\R^d)$ with $\supp(\psi) \subset B_{b+1}(0)$. Let $M,N$ be as in Lemma \ref{Lemma-trace-class-sob-prep} with arbitrary $\beta\geq 0$, $C_\beta$ independent of $L$ and $b$, and such that
\begin{equation}\label{log-L-Estimate-Req}
|\{x\in \Z^{d}: \ Q_x\cap M\neq \emptyset\}|\leq C_{M}L^d,
\end{equation}
where $Q_x$ denotes the closed unit cube centred about $x$ and $C_M>0$ is a constant independent of $L$ and $b$.
Then there exists a constant $C>0$, which is independent of $L$, $b$, $M$ and $N$,
such that for every $L\geq \max (b+1,2)$ we have
\begin{align}
\big\|\mathbf{1}_M \Op \big(\psi\mathcal{D}\big)\mathbf{1}_N \big\|_1 \leq C (C_\beta+C_M)\log L,
\end{align}
where $C_\beta$ is the corresponding constant from Lemma \ref{Lemma-trace-class-sob-prep}.
\end{lem}
\begin{proof}
Let $L\geq \max (b+1,2)$. We define the Lipschitz continuous function $\tau:\R^d\rightarrow\R_+$ by 
\begin{equation}
\label{tau-xi-def-A3}
\tau(\xi):=\frac{1}{4}\sqrt{\frac{1}{L^2}+|\xi|^2}.
\end{equation}
By \cite[Thm.1.4.10]{Hoermander} we obtain a sequence of centres $(\xi_j)_{j\in\N}\subset \R^d$ such that the balls of radius $\tau_j:=\tau(\xi_j)$ about the points $\xi_j$ cover $\R^d$, i.e. $\bigcup_{j\in\N}B_{\tau_{j}}(\xi_j)=\R^d$ and such that at most $N_0<\infty$ balls intersect in any given point, where $N_0$ only depends on the dimension $d$. Furthermore, there is a partition of unity $(\psi_j)_{j\in\N}$ subordinate to this covering, such that for every multi-index $\alpha\in\N_{0}^d$
we have  
\begin{align}\label{Psi-Derivative-A3}
\sup_{\;j\in\N^{\phantom{d}}} \!\sup_{\xi\in\R^d}|\partial^{\alpha}_{\xi}\psi_j(\xi)|\leq \tilde C_{|\alpha|}\tau(\xi)^{-|\alpha|},
\end{align}
with constants $\tilde C_{|\alpha|}>0$.
By the Lipschitz continuity of $\tau$ we obtain 
the inequality
\begin{align}\label{Lipschitz-Property-A3}
\frac{4}{5}\tau(\xi)\leq \tau_j \leq \frac{4}{3} \tau(\xi), \qquad \xi\in B_{\tau_j}(\xi_j)
\end{align}
for every $j\in \N$.

As $\supp(\psi) \subset B_{b+1}(0)$, there is a finite, $b$-dependent index set $J\subset\N$ with $\bigcup_{j\in J}B_{\tau_{j}}(\xi_j)\supseteq \supp( \psi\mathcal{D})$. We divide $J$ into the two parts
\begin{align}
J_1:=\{j\in J : B_{\tau_{j}}(\xi_j)\cap \{0\}\neq \emptyset\} \qquad \text{and} \qquad J_2:= J\setminus J_1.
\end{align}
The finite-intersection property implies the upper bound $|J_{1}| \le N_0$.
We estimate
\begin{equation}
\label{J1-J2-decomp-A3}
	\big\|\mathbf{1}_M \Op \big(\psi\mathcal{D}\big)\mathbf{1}_N \big\|_1  \\
	\leq \sum_{j\in J_1} \big\|\mathbf{1}_M \Op \big(\psi_j\psi\mathcal{D}\big)\mathbf{1}_N 
		\big\|_1 
	+ \sum_{j\in J_2} \big\|\mathbf{1}_M \Op \big(\psi_j\psi\mathcal{D}\big)\mathbf{1}_N 
		\big\|_1.
\end{equation}
If $j\in J_1$, we subdivide the set $M$ along $\Z^d$ and estimate
\begin{equation}\label{Estimate-J1-A3}
\big\|\mathbf{1}_M \Op \big(\psi_j\psi\mathcal{D}\big)\mathbf{1}_N \big\|_1
\leq \sum_{x\in\Z^{d}: \ Q_{x}\cap M\neq \emptyset} \big\| \Op^l \big(\phi_x\psi_j\otimes\mathbb{1}_{n})\big\|_1 \;
\big\|\Op \big(\psi\mathcal{D}\big)\big\| ,
\end{equation}
where the functions $\phi_x\in C_c^\infty(\R^d)$ with $\phi_x|_{Q_x}=1$ are supported in balls $B_{r_1}(x)$ for some radius $r_1>0$ and $\Op^l$ denotes the standard left-quantisation functor, cf.\ \cite[Sec.~2.2,~2.3]{BM-Widom}. As we have
$\big\|\Op \big(\psi\mathcal{D}\big)\big\| = \|\psi\|_\infty$ and 
\begin{align}\label{Estimate-J1-2-A3}
\big\| \Op^l \big(\phi_x\psi_j\otimes \mathbb{1}_{n})\big\|_1=n\big\| \Op^l \big(\phi_x\psi_j\big)\big\|_1,
\end{align}
it remains to estimate the trace-norm for the operator on the right-hand side of \eqref{Estimate-J1-2-A3}.
Using that $0 \in B_{\tau_j}(\xi_j)$ for all $j\in J_1$, property \eqref{Lipschitz-Property-A3} guarantees that $\phi_x\psi_j$ is compactly supported in $B_{r_1}(x)\times B_{\frac{1}{3L}}(\xi_j)$. Therefore, by an  application of \cite[Thm 3.1]{sobolevschatten}, the right-hand side of \eqref{Estimate-J1-2-A3} is bounded from above by $\frac{C_{1}}{L^d}$ for a constant $C_{1}>0$, independently of $L$ and $b$. This, in turn, provides, in combination with \eqref{log-L-Estimate-Req}, the bound
\begin{equation}
	\label{J1-sum-A3}
	\sum_{j\in J_1} \big\|\mathbf{1}_M \Op \big(\psi_j\psi\mathcal{D}\big)\mathbf{1}_N 
		\big\|_1
	\le N_0\, n \, C_{1} C_M\|\psi\|_\infty,
\end{equation}
where the right-hand side is independent of $L$ and $b$.

If instead $j\in J_2$, the symbol $\psi_j\psi\mathcal{D}$ is smooth with $\supp \psi_j\psi\mathcal{D} \subset B_{\tau_j}(\xi_j)$. We apply Lemma \ref{Lemma-trace-class-sob-prep} with $\psi=\psi_j\psi\mathcal{D}$ and obtain
\begin{equation}
\big\|\mathbf{1}_M \Op \big(\psi_j\psi\mathcal{D}\big)\mathbf{1}_N \big\|_1 \leq C_{\beta}C_2\max_{|\alpha|\leq 2d+\beta+1}\sup_{\xi\in\R^d}\tau_j^{|\alpha|}|\partial^{\alpha}_{\xi}\psi_j\psi\mathcal{D}(\xi)| \leq C_\beta C_{3},
\end{equation}
where we used \eqref{Psi-Derivative-A3} as well as \eqref{Lipschitz-Property-A3} and the constants $C_{2}, C_3>0$ are independent of $j\in J_{2}$, as well as of $L$, $b$, $M$ and $N$.
Therefore, 
\begin{equation}\label{Estimate-Finite-Intersection-A3}
\sum_{j\in J_2} \big\|\mathbf{1}_M \Op \big(\psi_j\psi\mathcal{D}\big)\mathbf{1}_N \big\|_1 \leq C_\beta C_{3} |J_2|\leq C_\beta C_4\int_{B_{2(b+1)}(0)} \tau(\xi)^{-d} \d\xi,
\end{equation}
where the constant $C_{4}>0$ is independent of $L$, $b$, $M$ and $N$. The last inequality is a consequence of the finite-intersection property and \eqref{Lipschitz-Property-A3}. We abbreviate $\rho:= 2(b+1)$ and write the right-hand side of 
\eqref{Estimate-Finite-Intersection-A3} as
\begin{align}
C_\beta C_4 |S^{d-1}| \int_0^{L\rho}  \frac{r^{d-1}}{\big(\sqrt{1+r^2}\big)^d} \; \d r \leq C_\beta C_{5} \log(\rho L),
\end{align}
for some constant $C_{5} >0$ that is independent of $L$, $b$, $M$ and $N$.
This estimate, together with \eqref{Estimate-Finite-Intersection-A3}, \eqref{J1-sum-A3} and \eqref{J1-J2-decomp-A3} 
concludes the proof of the Lemma, as $L \geq \max (b+1,2)$ and $C_\beta$ is independent of $L$ and $b$.
\end{proof}
In order to utilise this estimate, we need to further study the remaining operator $\mathbf{1}_{H_{V,L}}X_{g}$.
As in Section \ref{sec_higher_order} (cf. \eqref{theorem_higher_order_terms_goal_2}) we split the set $H_{V,L}=[0,L]^{d}$ into $d$ parts and rewrite the operator as
\begin{equation}\label{Recall-split-n}
\mathbf{1}_{H_{V,L}}X_{g}=\sum_{n=1}^{d}\mathbf{1}_{H_{L,d,k}}\sum_{\mathcal{M}\subseteq \{1,\ldots,d\}\setminus\{k\}}(-1)^{|\mathcal{M}|}\big[g\big(\mathbf{1}_{H_{\mathcal{M}}}X\mathbf{1}_{H_{\mathcal{M}}}\big)-g\big(\mathbf{1}_{H_{\mathcal{M}\cup\{k\}}}X\mathbf{1}_{H_{\mathcal{M}\cup\{k\}}}\big)\big],
\end{equation}
where we recall
\begin{equation}
H_{L,d,k}=\{x\in H_{V,L}: \ \forall j\in \{1,\ldots,d\}\setminus \{k\}: \ x_{j}\leq x_{k}\}.
\end{equation}
We relabel coordinates in suitable way to reduce \eqref{Recall-split-n} to a sum of expressions of the form
\begin{equation}\label{Commutation-step-1-sec-4}
\mathbf{1}_{H_{L,d,k}}\Big(g\big(\mathbf{1}_{H_{k}}X\mathbf{1}_{H_{k}}\big)-g\big(\mathbf{1}_{H_{k-1}}X\mathbf{1}_{H_{k-1}}\big)\Big),
\end{equation}
with $H_{k}:=\R_+^{k}\times\R^{d-k}$, for a given $k\in\{1,\ldots,d\}$. The structure of these expressions makes them easier to analyse for monomials and we are now ready to establish the upper bound for the trace norm of the remaining operator \eqref{rem_op}. This is the content of the following
\begin{thm}\label{Theorem-Upper-Bound}
Let $X=\Op(\psi\mathcal{D})$ and $g:\R\rightarrow\C$ be an entire function with $g(0)=0$. Then there exists a constant $C>0$, independent of $L$ and $b$, such that for all $L\geq \max (b+1,2)$ the bound
\begin{equation}\label{Theorem-Upper-Bound-trace-norm}
\Big\|\sum_{V\in\mathcal{F}^{(0)}}\mathbf{1}_{H_{V,L}}X_{g}\Big\|_1\leq C \log L
\end{equation}
holds.
\end{thm}
\begin{proof}
Let $L\geq \max (b+1,2)$. By suitable rotation and translation, the analysis above and the triangle inequality, it suffices to consider expressions of the form
\begin{equation}\label{Theorem-Upper-Bound-step-1}
\Big\|\mathbf{1}_{H_{L,d,k}}\Big(g\big(\mathbf{1}_{H_{k}}X\mathbf{1}_{H_{k}}\big)-g\big(\mathbf{1}_{H_{k-1}}X\mathbf{1}_{H_{k-1}}\big)\Big)\Big\|_1.
\end{equation}
We first prove the required bound \eqref{Theorem-Upper-Bound-step-1} for monomial test function. Writing out the operator in \eqref{Theorem-Upper-Bound-step-1}, we see that every term contains at least one occurrence of the projection $\mathbf{1}_{H_{k-1}}-\mathbf{1}_{H_{k}}$, and by H\"older's inequality and the triangle inequality, it suffices to consider trace norms of the form 
\begin{equation}\label{Theorem-Upper-Bound-step-2}
\big\|\mathbf{1}_{H_{L,d,k}}\big(\mathbf{1}_{H_{k}}X\big)^p\big(\mathbf{1}_{H_{k-1}}-\mathbf{1}_{H_{k}}\big)\big\|_1,
\end{equation}
for $p\in\N$. In the case $p=1$, the desired bound follows immediately from Lemma \ref{log-L-Estimate}. For $p\geq 2$, we again introduce several sets "in between" $H_{L,d,k}$ and $H_{k-1}\setminus H_k$. To do so, we first define the sets 
\begin{equation}
M_\alpha:=\{x\in \R^d: \ \forall j\in\{1,\ldots,d\}\setminus\{k\}: \ |x_j|\leq \alpha|x_k|\},
\end{equation}
for real numbers $\alpha>0$. The required sets are then given by $V_j:=H_k\cap [-(j+1)L,(j+1)L]^d\cap M_{j+1}$, for $j\in\{1,\ldots,p-1\}$. We set $V_0:=H_{L,d,k}$. Clearly, the sets $V_j$, $j\in\{0,\ldots,p-1\}$, fulfil condition \eqref{log-L-Estimate-Req} with $C_M=p^d$. The sets $M=V_j$, $N=H_k\setminus V_{j+1}$, $j\in\{0,\ldots,p-2\}$, and the sets $M=V_{p-1}$, $N=H_{k-1}\setminus H_k$ also fulfil condition \eqref{Lemma-trace-class-sob-prep-req} with $C_\beta=C_\beta'p^d$ and $\beta=d$, where $C_\beta'>0$ is independent of $p$, $L$ and $b$. Therefore, Lemma \ref{log-L-Estimate}, in combination with H\"older's inequality and the triangle inequality, yields 
\begin{multline}
\big\|\mathbf{1}_{H_{L,d,k}}\big(\mathbf{1}_{H_{k}}X\big)^p\big(\mathbf{1}_{H_{k}}-\mathbf{1}_{H_{k-1}}\big)\big\|_1
\leq \sum_{j=0}^{p-2} \big\|\mathbf{1}_{V_j}X\big(\mathbf{1}_{H_{k}}-\mathbf{1}_{V_{j+1}}\big)\big\|_1 + \big\|\mathbf{1}_{V_{p-1}}X\big(\mathbf{1}_{H_{k-1}}-\mathbf{1}_{H_k}\big)\big\|_1 \\
\leq \sum_{j=0}^{p-1}C_j p^d\log L \leq Cp^{d+1}\log L,
\end{multline}
with a constant $C>0$ independent of $L$, $b$ and $p$. 
This proves the theorem for monomials. The extension to analytic functions works in the usual way, cf., for example, the proof of Lemma \ref{higher-order-analytic}.
\end{proof}

\subsection{Proof of Theorem \ref{Theorem-Asymptotics-Analytic}}
In order to prove one of our main results Theorem \ref{Theorem-Asymptotics-Analytic}, it just remains to collect the ingredients from Sections \ref{sec_higher_order} and \ref{sec_upper_bound}.
\begin{proof}[Proof of Theorem \ref{Theorem-Asymptotics-Analytic}]
By assumption the function $g$ is an entire function, with $g(0)=0$. We set $X:=\Op(\psi\mathcal{D})$, where $\psi$ and $\mathcal{D}$ are defined in \eqref{Operator-m=0-A3}. With all the requirements of Theorem \ref{theorem_higher_order_terms} being fulfilled, the theorem yields
\begin{multline}
\tr_{L^{2}(\R^{d})\otimes\C^n}\bigg[g\big(\mathbf{1}_{\Lambda_{L}}\Op(\psi\mathcal{D})\mathbf{1}_{\Lambda_{L}}\big)-\sum_{V\in\mathcal{F}^{(0)}}\mathbf{1}_{H_{V,L}}\Op(\psi\mathcal{D})_{LV,g}\bigg]\\
=\sum_{m=0}^{d-1}(2L)^{d-m} A_{m,g,b} 
+O(1),
\end{multline}
as $L\rightarrow\infty$. The coefficients $A_{m,g,b}$ are defined in \eqref{definition-A-coeff}. Therefore, it remains to consider the trace of the operator
\begin{equation}
\sum_{V\in\mathcal{F}^{(0)}}\mathbf{1}_{H_{V,L}}\Op(\psi\mathcal{D})_{LV,g}.
\end{equation}
By Theorem \ref{Theorem-Upper-Bound} we have
\begin{equation}
\tr_{L^{2}(\R^{d})\otimes\C^n}\bigg[\sum_{V\in\mathcal{F}^{(0)}}\mathbf{1}_{H_{V,L}}\Op(\psi\mathcal{D})_{LV,g}\bigg]\leq \Big\|\sum_{V\in\mathcal{F}^{(0)}}\mathbf{1}_{H_{V,L}}\Op(\psi\mathcal{D})_{LV,g}\Big\|_1\leq C\log L,
\end{equation}
with a constant $C$ independent of $L$ and $b$. This concludes the proof of the theorem.
\end{proof}

\section{Commutation in momentum space}
\label{sec_commutation}

With the upper bound being established, the remaining task is to extend the asymptotic expansion to the logarithmic term and to compute the corresponding coefficient. Similarly to the strategy used to obtain an enhanced are law in for example \cite{BM-Widom}, we divide this task into two steps. The first step is to separate the smooth cut-off functions $\psi$ from the projections $\Op(\mathcal{D})$ with discontinuous symbol and commute all occurrences of $\Op(\psi)$ to the right of the operator. The second step is to compute the asymptotic expansion for the resulting operator. This is done in Section \ref{sec_local_asymptotics}. 
As stated in the introduction, we are only able to do this last step in the case that the test function is a polynomial of degree three or less. Nonetheless, we are able to do the commutation for arbitrary entire functions $g$ with $g(0)=0$. As this would be necessary for a potential extension of Theorem \ref{Theorem-Asymptotics-Log} to entire test functions, we still carry out the proof in this general setting in this section, although Theorem \ref{Theorem-Asymptotics-Log} only requires the commutation results for polynomial of degree three or less.

The idea for the commutation of the smooth symbols $\psi$ is to mirror the analysis done in the case of an enhanced area law, while treating the projections $\Op(\mathcal{D})$ in the same way as the projections on a basic domain in the enhanced area law case (cf. \cite[Sec.~3]{BM-Widom} for an example of the enhanced area law case). The usual idea is to obtain trace-class bounds for the commutator of $\Op(\psi)$, for $\psi\in\C_{c}^\infty(\R^d)$, with the projection $\mathbf{1}_{\Omega_L}$ on some scaled subset $\Omega_L\subset\R^d$. But one can only hope to obtain bounds of order $L^{d-1}$ for this commutator. Therefore, we require an extended procedure, which makes use of the structure of the operator $X_{LV,g}$, here.
Nevertheless, it will be vital to use the estimates established in Section \ref{sec_upper_bound} for the smooth symbol $\psi$.
We summarise the result of this commutation procedure in the following theorem which we prove at the end of this section, when the necessary intermediate results are established.
\begin{thm}\label{Theorem-commutation}
Let $X=\Op(\psi\mathcal{D})$ and $g:\R\rightarrow\C$ be an entire function with $g(0)=0$. Then
\begin{equation}\label{Theorem-commutation-bound}
\tr_{L^{2}(\R^{d})\otimes\C^n}\bigg[\sum_{V\in\mathcal{F}^{(0)}}\mathbf{1}_{H_{V,L}}X_{LV,g}-\sum_{V\in\mathcal{F}^{(0)}}\mathbf{1}_{H_{V,L}}\Op\big(\mathcal{D}\big)_{LV,g}\mathbf{1}_{H_{V,L}}\Op\big(g(\psi\otimes\mathbb{1}_{n})\big)\bigg]=O(1),
\end{equation}
as $L\rightarrow\infty$, where we recall
\begin{equation}
Y_{F,g}=\sum_{k=0}^{d}(-1)^{k}\sum_{G\in\mathcal{F}_{F}^{(k)}}g(\mathbf{1}_{H_G}Y\mathbf{1}_{H_G})
\end{equation}
for $Y\in\{X,\Op(\mathcal{D})\}$.
\end{thm}
We begin by establishing the required bounds for polynomial test functions and extend this to entire test functions in Section \ref{subsec_proof_comm}.
\subsection{Commutation for monomial test functions}
\label{subsec_commutation}
As in \eqref{Commutation-step-1-sec-4} we reduce the structure of the operator $\mathbf{1}_{H_{V,L}}X_{g}$ to a sum of expressions of the following form
\begin{equation}\label{Commutation-step-1}
\mathbf{1}_{H_{L,d,k}}\Big(g\big(\mathbf{1}_{H_{k}}X\mathbf{1}_{H_{k}}\big)-g\big(\mathbf{1}_{H_{k-1}}X\mathbf{1}_{H_{k-1}}\big)\Big),
\end{equation}
where we recall $H_{k}=\R_+^{k}\times\R^{d-k}$, for a given $k\in\{1,\ldots,d\}$, as well as 
\begin{equation}
H_{L,d,k}=\{x\in H_{V,L}: \ \forall j\in \{1,\ldots,d\}\setminus \{k\}: \ x_{j}\leq x_{k}\}.
\end{equation} 
We now analyse the difference of the traces of the operator \eqref{Commutation-step-1}, with $X=\Op(\psi\mathcal{D})$ and $g$ given by a monomial, and the same operator with all occurrences of $\Op(\psi)$ commuted to the right. More precisely, we want to find a bound for
\begin{multline}\label{Commutation-section-goal}
\Big|\tr_{L^{2}(\R^{d})\otimes\C^n}\Big[\mathbf{1}_{H_{V,L}}\Big(\big(\mathbf{1}_{H_{k}}\Op(\psi\mathcal{D})\mathbf{1}_{H_{k}}\big)^p-\big(\mathbf{1}_{H_{k-1}}\Op(\psi\mathcal{D})\mathbf{1}_{H_{k-1}}\big)^p\Big)\mathbf{1}_{H_{L,d,k}} \\
 -\mathbf{1}_{H_{V,L}}\Big(\big(\mathbf{1}_{H_{k}}\Op(\mathcal{D})\mathbf{1}_{H_{k}}\big)^p-\big(\mathbf{1}_{H_{k-1}}\Op(\mathcal{D})\mathbf{1}_{H_{k-1}}\big)^p\Big)\mathbf{1}_{H_{L,d,k}}\Op(\psi^{p}\otimes\mathbb{1}_{n})\Big]\Big|,
\end{multline}
for given $k\in\{1,\ldots,d\}$ and $p\in\N$. This is the goal of the present section.

The proof of this bound is split into several lemmas contained in this section. As in Section \ref{sec_higher_order}, we start with Hilbert-Schmidt norm estimates, which present a large part of the technical challenge and then build an estimate for the absolute value of the trace from there.
The idea to prove the occurring Hilbert-Schmidt bounds is similar to the one employed in the proof of Lemma \ref{Lemma-HS-decay}, in the sense that we define a family of sets depending on power functions $\phi_m$, $m\in\{1,\ldots,p-1\}$, which will allow us to "transport" the decay of the integral kernel along multiple occurrences of the projections $\mathbf{1}_{H_{k}}$ respectively $\mathbf{1}_{H_{k-1}}-\mathbf{1}_{H_{k}}$. Although, in this case we need different sets. To define the relevant sets, we take a closer look at the boundary of $H_{k}$.
We note that 
\begin{equation}
\partial H_{k}=\{x\in H_{k}: \min_{1\leq j\leq k}x_j=0\}
\end{equation}
and we have that
\begin{equation}
\dist(x,\partial H_{k})\geq \min_{1\leq j\leq d}|x_j|,
\end{equation}
for all $x\in\R^d$. We now define a family of "thickened up" versions of the boundary of $H_k$. To do so, let $\phi:[0,\infty[\,\rightarrow[0,\infty[\,$ be a measurable function. We define
\begin{equation}
\big(\partial H_{k}\big)_{\phi}:=\{x\in \R^d :  \dist(x,\partial H_{k})<\phi(\|x\|_\infty)\}, 
\end{equation} 
where $\|x\|_\infty:=\max_{1\leq j\leq d}|x_j|$. During the proof of \eqref{Commutation-section-goal}, it is necessary to split the boundary of $H_k$ into two parts.
For real numbers $\alpha>0$, we recall the definition of the sets  
\begin{equation}
M_\alpha=\{x\in \R^d: \ \forall j\in\{1,\ldots,d\}\setminus\{k\}: \ |x_j|\leq \alpha|x_k|\},
\end{equation}
and split the set $\big(\partial H_{k}\big)_{\phi}$ into the two subsets
\begin{equation}
\big(\partial H_k\big)\cap M_\alpha \quad \text{and} \quad \big(\partial H_k\big)\cap M_\alpha^c.
\end{equation}
Comparing this with $H_{k-1}$, we see that, for every $\alpha>0$, we have $\partial_{H_{k-1}}H_{k}\subset M_\alpha^c$, where $\partial_{H_{k-1}}H_{k}$ is the boundary of $H_{k}$ in $H_{k-1}$. This will be crucial in the proof of Lemma \ref{Lemma-commutation-combination}.

We also need the same objects for the boundary of $H_{k-1}\setminus H_{k}$.
We define
\begin{equation}
\big(\partial (H_{k-1}\setminus H_{k})\big)_{\phi}:=\{x\in \R^d :  \dist(x,\partial (H_{k-1}\setminus H_{k})<\phi(\|x\|_\infty)\}
\end{equation} 
and note that $\partial_{H_{k-1}}(H_{k-1}\setminus H_k)=\partial_{H_{k-1}}H_{k}\subset M_\alpha^c$, for all $\alpha>0$. With these definitions at hand, we are ready to prove the Hilbert-Schmidt bounds required to commute $\Op(\psi)$ with $\mathbf{1}_{H_k}$ respectively $\mathbf{1}_{H_{k-1}\setminus H_k}$.

\begin{lem}\label{Lemma-commutation-HS}
Let $p\in\N$ and $X_1,\ldots, X_p$ be bounded translation-invariant integral operators, each satisfying the estimate \eqref{kernel_bound_distribution} and such that the operator norms $\|X_m\|$, $1\leq m \leq p$, are bounded from above by $1$. Let $k\in\{1,\ldots,d\}$ and let $\phi:[0,\infty[\,\rightarrow[0,\infty[\,$ be given by $\phi(x):=\tfrac{1}{8}\sqrt{\max(0,x-p)}$. For $m\in\{1,\ldots,p\}$, let $V_m\subseteq H_k$ be measurable. Then there exists a constant $C>0$, independent of all $V_m$ and $p$, such that, for every measurable $\Omega_1\subseteq \big(\partial H_{k}\big)_{\phi}\cap M_{3}$, we have
\begin{equation}\label{Lemma-commutation-HS-bound-1}
\big\|\mathbf{1}_{\Omega_1}X_1\Big(\prod_{m=1}^{p-1}\mathbf{1}_{V_m}X_{m+1}\Big)(\mathbf{1}_{H_{k-1}}-\mathbf{1}_{H_{k}})\big\|_2\leq Cd^{\tfrac{p}{2}} p^{\tfrac{d+2}{2}},
\end{equation}
and for every measurable $\Omega_2\subseteq \big(\partial (H_{k-1}\setminus H_{k})\big)_{\phi}\cap M_{3}$, we have
\begin{equation}\label{Lemma-commutation-HS-bound-2}
\big\|\mathbf{1}_{\Omega_2}X_1\Big(\prod_{m=1}^{p-1}(\mathbf{1}_{H_{k-1}}-\mathbf{1}_{H_{k}})X_{m+1}\Big)\mathbf{1}_{V_p}\big\|_2\leq Cd^{\tfrac{p}{2}} p^{\tfrac{d+2}{2}}.
\end{equation}
Furthermore, let $\Omega'\subseteq M_1$ be measurable and $P_{m}\in\{\mathbf{1}_{H_k},(\mathbf{1}_{H_{k-1}}-\mathbf{1}_{H_{k}})\}$, for $m\in\{1,\ldots,p-1\}$. Then there exists a constant $C'>0$, independent of $\Omega'$ and $p$, such that, for every measurable $\Omega_3\subseteq \big(\partial H_{k}\big)_{\phi}\cap M_{2}^c$, we have
\begin{equation}\label{Lemma-commutation-HS-bound-3}
\big\|\mathbf{1}_{\Omega_3}X_1\Big(\prod_{m=1}^{p-1}P_mX_{m+1}\Big)\mathbf{1}_{\Omega'}\big\|_2\leq C'd^{\tfrac{p}{2}} p^{\tfrac{d+2}{2}},
\end{equation}
and for every measurable $\Omega_4\subseteq \big(\partial (H_{k-1}\setminus H_{k})\big)_{\phi}\cap M_{2}^c$, we have
\begin{equation}\label{Lemma-commutation-HS-bound-4}
\big\|\mathbf{1}_{\Omega_4}X_1\Big(\prod_{m=1}^{p-1}P_mX_{m+1}\Big)\mathbf{1}_{\Omega'}\big\|_2\leq C'd^{\tfrac{p}{2}} p^{\tfrac{d+2}{2}}.
\end{equation}
\end{lem}
\begin{proof}
The proof employs a strategy similar to the one employed in the proof of Lemma \ref{Lemma-HS-decay}. Therefore, some of the steps are carried out in less detail.
We begin with the proof of \eqref{Lemma-commutation-HS-bound-1} and start with the case $p=1$. For every 
 $x\in \Omega_1$, we have that $x_k\geq 0$, as otherwise we have, using that $\phi$ is monotone,
\begin{equation}
 |x_k|\leq \dist(x,\partial H_k)<\phi(\|x\|_\infty)\leq \phi(3|x_k|)=\frac{\sqrt{\max(0,3|x_k|-1)}}{8},
\end{equation}
which leads to a contradiction. Therefore, for every $x\in \Omega_1$, we have that $\dist(x,H_{k-1}\setminus H_k)\geq x_k$. In particular we have $\dist(\Omega_1, H_{k-1}\setminus H_k)> \frac{1}{3}$, as in the case $x\in \Omega_1$ with $x_k\leq \frac{1}{3}$, we would have $\|x\|_\infty\leq 1$ and therefore $x\notin \big(\partial H_{k}\big)_{\phi}$. We apply Lemma \ref{Lemma_kernel_set_distribution} with $M=\Omega_1$, $N=H_{k-1}\setminus H_k$ and $\varphi$ given by $\varphi(x)=x_k$. We recall that $\dist(x,\partial H_k)\geq \min_{1\leq j \leq d}|x_j|$ and compute
\begin{multline}
\int_{\Omega_1}\frac{1}{x_k^{d}}\dd x =\sum_{l=1}^{d}\int_{\Omega_1\cap \{x\in\R^{d}: \ |x_l|= \min_{1\leq j\leq d}|x_j|\}}\frac{1}{x_k^{d}}\dd x \\
\leq (d-1) \int_{\tfrac{1}{3}}^\infty \frac{2^{d-1}(3x_k)^{d-2}\phi(3x_k)}{x_k^{d}}\dd x_k \leq \frac{C'}{2\sqrt{3}}\int_{\tfrac{1}{3}}^\infty \frac{x_k^{d-\tfrac{3}{2}}}{x_k^{d}}\dd x_k = C',
\end{multline}
where we used that $x_k\neq \min_{1\leq j\leq d}|x_j|$, as otherwise $x_k\leq \dist(x,\partial H_k)< \phi(\|x\|_\infty)\leq \phi(3x_k)$, which again leads to a contradiction. 
Therefore, Lemma \ref{Lemma_kernel_set_distribution} yields the desired constant $C>0$, such that
\begin{equation}
\big\|\mathbf{1}_{\Omega_1}X_{1}(\mathbf{1}_{H_{k-1}}-\mathbf{1}_{H_{k}})\big\|_2\leq C.
\end{equation}
This proves \eqref{Lemma-commutation-HS-bound-1} in the case $p=1$.
We now turn to the case $p\geq 2$. For $m\in\{1,\ldots,p-1\}$, we define the sets 
\begin{equation}
\Gamma_m:= V_{m}\cap\big(\partial H_{k}\big)_{\phi_m}\cap M_{\alpha_m}, 
\end{equation}
with $\alpha_m:=3+\tfrac{m}{p}$ and the functions $\phi_m:[0,\infty[\,\rightarrow [0,\infty[\,$ given by
\begin{equation}
\phi_m(x):=\tfrac{1}{8}\big(1+\tfrac{m}{p}\big)\max\big(0,x-p+m\big)^{1-\tfrac{1}{2d^{m}}+\epsilon},
\end{equation}
where we set $\epsilon:=\tfrac{1}{2d^{p}}$ for the remaining part of the proof.
We first consider the Hilbert-Schmidt norm of the operator $\mathbf{1}_{\Omega_1}X_1\mathbf{1}_{V_1\setminus\Gamma_1}$. Let $x\in\Omega_1$ and $y\in B_{\tfrac{\phi_1(\|x\|_\infty)}{4p}}(x)$, then we have
\begin{equation}
\|y\|_\infty > \|x\|_\infty-\tfrac{\phi_1(\|x\|_\infty)}{4p}\geq \|x\|_\infty - \tfrac{\|x\|_\infty-p+1}{16p} =\tfrac{(16p-1)\|x\|_\infty+(p-1)}{16p},
\end{equation}
where we used that $\|x\|_\infty>p$.
A short computation yields 
\begin{equation}
\frac{1+\tfrac{1}{p}}{8}\frac{16p-1}{16p}=\frac{16p+16-1-\tfrac{1}{p}}{128p}> \frac{1}{8}+\frac{1}{4p}\frac{1+\tfrac{1}{p}}{8}.
\end{equation}
With this, the monotonicity of $\phi_1$ and its definition, we obtain the bound
\begin{align}
\phi_1(\|y\|_\infty) &\geq \phi_1\big(\tfrac{(16p-1)\|x\|_\infty+(p-1)}{16p}\big) \nonumber \\
&\geq \Big(\tfrac{1}{8}+\tfrac{1}{4p}\tfrac{1+1/p}{8}\Big)\tfrac{16p}{16p-1}\max\big(0,\tfrac{(16p-1)\|x\|_\infty+(p-1)}{16p}-(p-1)\big)^{1-\tfrac{1}{2d}+\epsilon} \nonumber \\
&\geq \tfrac{1}{8}(\|x\|_\infty-p+1)^{1-\tfrac{1}{2d}+\epsilon}+\tfrac{1}{4p}\tfrac{1+1/p}{8}(\|x\|_\infty-p+1)^{1-\tfrac{1}{2d}+\epsilon}\nonumber \\
&> \phi(\|x\|_\infty)+\tfrac{\phi_1(\|x\|_\infty)}{4p}.
\end{align}
We conclude 
\begin{equation}
\dist(y,\partial H_k)\leq \dist(x,\partial H_k)+\tfrac{\phi_1(\|x\|_\infty)}{4p}< \phi(\|x\|_\infty)+\tfrac{\phi_1(\|x\|_\infty)}{4p}<\phi_1(\|y\|_\infty),
\end{equation}
i.e. $y\in \big(\partial H_{k}\big)_{\phi_1}$. We also have
\begin{multline}
\|y\|_\infty< \|x\|_\infty+\tfrac{\phi_1(\|x\|_\infty)}{4p}\leq \tfrac{16p+1}{16p}\|x\|_\infty\leq \tfrac{16p+4+1/p}{16p}\|x\|_\infty-(3+\tfrac{1}{p})\tfrac{\phi_1(\|x\|_\infty)}{4p} \\
\leq \tfrac{48p+12+3/p}{16p}|x_k|-(3+\tfrac{1}{p})\tfrac{\phi_1(\|x\|_\infty)}{4p}\leq (3+\tfrac{1}{p})\Big(|x_k|-\tfrac{\phi_1(\|x\|_\infty)}{4p}\Big)<(3+\tfrac{1}{p})|y_k|,
\end{multline}
i.e. $y\in M_{\alpha_1}$. Therefore, for every $x\in\Omega_1$, we have $\dist(x,V_1\setminus \Gamma_1)\geq \tfrac{\phi_1(\|x\|_\infty)}{4p}$. As $\|x\|_\infty>p$, for all $x\in \Omega_1$, we know in particular that $\dist(\Omega_1,V_1\setminus \Gamma_1)>\tfrac{1}{32p}$.  We now apply Lemma \ref{Lemma_kernel_set_distribution} with $M=\Omega_1$, $N=V_1\setminus \Gamma_1$ and $\varphi$ given by $\varphi(x)=\tfrac{\phi_1(\|x\|_\infty)}{4p}$. For the corresponding integral, we compute
\begin{align}\label{Lemma-commutation-integral-1}
\int_{\Omega_1}\frac{(4p)^d}{\phi_1(\|x\|_\infty)^{d}}&\dd x  \nonumber \\
&=\sum_{l=1}^{d}\sum_{\mu\in\{1,\ldots,d\}\setminus\{l\}}\int_{\Omega_1\cap \{x\in\R^{d}: \ |x_l|= \|x\|_\infty\}\cap \{x\in\R^{d}: \ |x_\mu|= \min_{1\leq j\leq d}|x_j|\}}\frac{(4p)^d}{\phi_1(\|x\|_\infty)^{d}}\dd x \nonumber \\
&\leq d(d-1) (64p)^{d}\int_{p}^\infty \frac{x_l^{d-2}\phi(x_l)}{(x_l-p+1)^{d-\tfrac{1}{2}+d\epsilon}}\dd x_l \nonumber \\
&\leq \frac{d(d-1) (64p)^{d}}{8}\int_{1}^\infty \frac{(x_l+p-1)^{d-2}\sqrt{x_l}}{x_l^{d-\tfrac{1}{2}+d\epsilon}}\dd x_l \nonumber \\
&\leq \frac{C'p^{d}}{\epsilon},
\end{align}
with a constant $C'>0$, independent of $V_1$ and $p$. Therefore, Lemma \ref{Lemma_kernel_set_distribution} yields a constant $C_1>0$, independent of $V_1$ and $p$, such that
\begin{equation}\label{Lemma-commutation-operator-HS-estimate-1}
\big\|\mathbf{1}_{\Omega_1}X_1\mathbf{1}_{V_1\setminus\Gamma_1}\big\|_2\leq C_1\sqrt{\frac{p^{d}}{\epsilon}}.
\end{equation}
We continue with the intermediate terms. Let $m\in \{1,\ldots,p-2\}$, then we want to find a bound for the Hilbert-Schmidt norm of the operators $\mathbf{1}_{\Gamma_m}X_{m+1}\mathbf{1}_{V_{m+1}\setminus\Gamma_{m+1}}$. Let $x\in\Gamma_m$ and $y\in B_{\tfrac{\phi_{m+1}(\|x\|_\infty)}{8p}}(x)$, then we have
\begin{equation}
\|y\|_\infty > \|x\|_\infty-\tfrac{\phi_{m+1}(\|x\|_\infty)}{8p}\geq \|x\|_\infty - \tfrac{\|x\|_\infty-p+m+1}{32p} =\tfrac{(32p-1)\|x\|_\infty+(p-m-1)}{32p},
\end{equation}
where we used that $\|x\|_\infty>p-m$.
A short computation yields 
\begin{equation}
\frac{1+\tfrac{m+1}{p}}{8}\frac{32p-1}{32p}=\frac{32p+32(m+1)-1-\tfrac{m+1}{p}}{256p}> \frac{1+\tfrac{m}{p}}{8}+\frac{1}{8p}\frac{1+\tfrac{m+1}{p}}{8}.
\end{equation}
With this, the monotonicity of $\phi_{m+1}$ and its definition, we obtain the bound
\begin{align}
\phi_{m+1}(\|y\|_\infty) &\geq \phi_{m+1}\big(\tfrac{(32p-1)\|x\|_\infty+(p-m-1)}{32p}\big)\nonumber\\
&\geq \Big(\tfrac{1+m/p}{8}+\tfrac{1}{8p}\tfrac{1+(m+1)/p}{8}\Big)\tfrac{32p}{32p-1}\Big(\tfrac{(32p-1)\|x\|_\infty+(p-m-1)}{32p}-(p-m-1)\Big)^{1-\tfrac{1}{2d^{m+1}}+\epsilon} \nonumber\\
&\geq \tfrac{1+m/p}{8}(\|x\|_\infty-p+m+1)^{1-\tfrac{1}{2d^{m+1}}+\epsilon}+\tfrac{1}{8p}\tfrac{1+(m+1)/p}{8}(\|x\|_\infty-p+m+1)^{1-\tfrac{1}{2d^{m+1}}+\epsilon}\nonumber\\ 
&> \phi_{m}(\|x\|_\infty)+\tfrac{\phi_{m+1}(\|x\|_\infty)}{8p}.
\end{align}
We conclude 
\begin{equation}
\dist(y,\partial H_k)\leq \dist(x,\partial H_k)+\tfrac{\phi_{m+1}(\|x\|_\infty)}{8p}< \phi_{m}(\|x\|_\infty)+\tfrac{\phi_{m+1}(\|x\|_\infty)}{8p}<\phi_{m+1}(\|y\|_\infty),
\end{equation}
i.e. $y\in \big(\partial H_{k}\big)_{\phi_{m+1}}$. We also have
\begin{align}
\|y\|_\infty &< \|x\|_\infty+\tfrac{\phi_{m+1}(\|x\|_\infty)}{8p}\leq \tfrac{32p+1}{32p}\|x\|_\infty\leq \tfrac{32p+4+(m+1)/p}{32p}\|x\|_\infty-(3+\tfrac{m+1}{p})\tfrac{\phi_{m+1}(\|x\|_\infty)}{8p} \nonumber\\
&\leq \tfrac{96p+12+3(m+1)/p+32m+4m/p+m(m+1)/p^2}{32p}|x_k|-(3+\tfrac{m+1}{p})\tfrac{\phi_{m+1}(\|x\|_\infty)}{8p}\nonumber\\&\leq (3+\tfrac{m+1}{p})\big(|x_k|-\tfrac{\phi_{m+1}(\|x\|_\infty)}{8p}\big)<(3+\tfrac{m+1}{p})|y_k|,
\end{align}
i.e. $y\in M_{\alpha_{m+1}}$. Therefore, for every $x\in\Gamma_m$, we have $\dist(x,V_{m+1}\setminus \Gamma_{m+1})\geq \tfrac{\phi_{m+1}(\|x\|_\infty)}{8p}$. As $\|x\|_\infty>p-m$, for all $x\in \Gamma_m$, we know in particular that $\dist(\Gamma_m,V_{m+1}\setminus \Gamma_{m+1})>\tfrac{1}{64p}$. We again apply Lemma \ref{Lemma_kernel_set_distribution}. Now with $M=\Gamma_m$, $N=V_{m+1}\setminus \Gamma_{m+1}$ and $\varphi$ given by $\varphi(x)=\tfrac{\phi_{m+1}(\|x\|_\infty)}{8p}$. For the corresponding integral, we compute
\begin{align}\label{Lemma-commutation-integral-2}
\int_{\Gamma_m}&\frac{(8p)^d}{\phi_{m+1}(\|x\|_\infty)^{d}}\dd x  \nonumber \\
&\leq\sum_{l=1}^{d}\sum_{\mu\in\{1,\ldots,d\}\setminus\{l\}}\int_{\big(\partial H_{k}\big)_{\phi_m}\cap \{x\in\R^{d}: \ |x_l|= \|x\|_\infty\}\cap \{x\in\R^{d}: \ |x_\mu|= \min_{1\leq j\leq d}|x_j|\}}\frac{(8p)^d}{\phi_{m+1}(\|x\|_\infty)^{d}}\dd x \nonumber \\
&\leq d(d-1) (128p)^{d}\int_{p-m}^\infty \frac{(x_l)^{d-2}\phi_m(x_l)}{(x_l-p+m+1)^{d-\tfrac{1}{2d^{m}}+d\epsilon}}\dd x_l \nonumber \\
&\leq \frac{d(d-1) (128p)^{d}}{4}\int_{1}^\infty \frac{(x_l+p-m-1)^{d-2}x_l^{1-\tfrac{1}{2d^{m}}+\epsilon}}{x_l^{d-\tfrac{1}{2d^{m}}+d\epsilon}}\dd x_l \nonumber \\
&\leq \frac{C'p^{d}}{\epsilon},
\end{align}
with a constant $C'>0$, independent of $V_{m}$, $V_{m+1}$ and $p$. Therefore, Lemma \ref{Lemma_kernel_set_distribution} yields a constant $C_{m+1}>0$, independent of $V_{m}$, $V_{m+1}$ and $p$, such that
\begin{equation}\label{Lemma-commutation-operator-HS-estimate-2}
\big\|\mathbf{1}_{\Gamma_{m}}X_{m+1}\mathbf{1}_{V_{m+1}\setminus\Gamma_{m+1}}\big\|_2\leq C_{m+1}\sqrt{\frac{p^{d}}{\epsilon}}.
\end{equation}
It remains to estimate the last operator $\mathbf{1}_{\Gamma_{p-1}}X_{p}(\mathbf{1}_{H_{k-1}}-\mathbf{1}_{H_{k}})$.
For every 
 $x\in \Gamma_{p-1}$, we have that $x_k\geq 0$, as $\Gamma_{p-1}\subset H_k$. Therefore,  we have, for every $x\in \Gamma_{p-1}$, that $\dist(x,H_{k-1}\setminus H_k)\geq x_k$. In particular we have $\dist(\Gamma_{p-1}, H_{k-1}\setminus H_k)> \frac{1}{4}$, as in the case $x\in \Gamma_{p-1}$ with $x_k\leq \frac{1}{4}$, we would have $\|x\|_\infty\leq 1$ and therefore $x\notin \big(\partial H_{k}\big)_{\phi_{p-1}}$. We apply Lemma \ref{Lemma_kernel_set_distribution} with $M=\Gamma_{p-1}$, $N=H_{k-1}\setminus H_k$ and $\varphi$ given by $\varphi(x)=x_k$. We recall that $\dist(x,\partial H_k)\geq \min_{1\leq j \leq d}|x_j|$ and compute
\begin{align}
\int_{\Gamma_{p-1}}\frac{1}{x_k^{d}}\dd x &= \sum_{l=1}^{d}\int_{\big(\partial H_{k}\big)_{\phi_{p-1}}\cap M_{4} \cap \{x\in\R^{d}: \ |x_l|= \min_{1\leq j\leq d}|x_j|\}}\frac{1}{x_k^{d}}\dd x \nonumber\\
&\leq (d-1) \int_{\tfrac{1}{4}}^\infty \frac{2^{d}(4x_k)^{d-2}\phi_{p-1}(4x_k)}{x_k^{d}}\dd x_k \nonumber\\
&\leq 8^{d}(d-1) \int_{\tfrac{1}{4}}^\infty x_k^{-1-\tfrac{1}{2d^{p-1}}+\epsilon}\dd x_k \leq C'\frac{1}{\epsilon},
\end{align}
where the constant $C'>0$ is independent of $V_{p-1}$ and $p$.
Therefore, Lemma \ref{Lemma_kernel_set_distribution} yields a constant $C_p>0$, independent of $V_{p-1}$ and $p$, such that
\begin{equation}\label{Lemma-commutation-operator-HS-estimate-3}
\big\|\mathbf{1}_{\Gamma_{p-1}}X_{p}(\mathbf{1}_{H_{k-1}}-\mathbf{1}_{H_{k}})\big\|_2\leq C_p\sqrt{\frac{1}{\epsilon}}.
\end{equation}
Repeated use of the triangle inequality, followed by H\"older's inequality, yields
\begin{align}\label{Lemma-commutation-Hölder}
\big\|&\mathbf{1}_{\Omega_1}X_1\Big(\prod_{m=1}^{p-1}\mathbf{1}_{V_m}X_{m+1}\Big)(\mathbf{1}_{H_{k-1}}-\mathbf{1}_{H_{k}})\big\|_2 \nonumber \\
&\leq \sum_{m=0}^{p-2}\Big\|\mathbf{1}_{\Omega_1}X_1\Big(\prod_{j=1}^{m}\mathbf{1}_{\Gamma_j}X_{j+1}\Big)\mathbf{1}_{V_{m+1}\setminus \Gamma_{m+1}}X_{m+2}\Big(\prod_{j=m+2}^{p-1}\mathbf{1}_{V_{j}}X_{j+1}\Big)(\mathbf{1}_{H_{k-1}}-\mathbf{1}_{H_{k}})\Big\|_2
\nonumber \\
& \qquad\quad  + \Big\|\mathbf{1}_{\Omega_1}X_1\Big(\prod_{j=1}^{p-1}\mathbf{1}_{\Gamma_{j}}X_{j+1}\Big)(\mathbf{1}_{H_{k-1}}-\mathbf{1}_{H_{k}})\Big\|_2 \nonumber \\
&\leq  \big\|\mathbf{1}_{\Omega_1}X_1\mathbf{1}_{V_{1}\setminus \Gamma_{1}}\big\|_2 +\sum_{m=1}^{p-2}\big\|\mathbf{1}_{\Gamma_m}X_{m+1}\mathbf{1}_{V_{m+1}\setminus \Gamma_{m+1}}\big\|_2+\big\|\mathbf{1}_{\Gamma_{p-1}}X_{p}(\mathbf{1}_{H_{k-1}}-\mathbf{1}_{H_{k}})\big\|_2  \nonumber\\
&\leq C d^{\tfrac{p}{2}} p^{\tfrac{d+2}{2}},
\end{align}
where we combined estimates \eqref{Lemma-commutation-operator-HS-estimate-1}, \eqref{Lemma-commutation-operator-HS-estimate-2} and \eqref{Lemma-commutation-operator-HS-estimate-3} and used the definition of $\epsilon$ in the last line. The constant $C$ is independent of $p$ and $V_m$, for every $m\in\{1,\ldots,p-1\}$. This proves \eqref{Lemma-commutation-HS-bound-1} for $p\geq 2$.

The bound \eqref{Lemma-commutation-HS-bound-2} reduces to \eqref{Lemma-commutation-HS-bound-1} through the reflection $S:\R^d\rightarrow\R^d$ along the hyperplane orthogonal to the $k$th coordinate, as we have $S(H_k)=H_{k-1}\setminus H_k$ and $S(M_\alpha)=M_\alpha$, for every $\alpha>0$.

We continue with the proof of \eqref{Lemma-commutation-HS-bound-3}. We again begin with the case $p=1$. Let $x\in \Omega_3\subset (M_2)^c$ and $y\in B_{\tfrac{\|x\|_\infty}{4}}(x)$. Then we have
\begin{equation}
\|y\|_\infty>\tfrac{3}{4}\|x\|_\infty=\tfrac{1}{2}\|x\|_\infty+\tfrac{\|x\|_\infty}{4}>|x_k|+\tfrac{\|x\|_\infty}{4}>|y_k|,
\end{equation}
i.e. $y\notin M_{1}$. Therefore, we have $\dist(x,\Omega')\geq \dist(x,M_{1})\geq \tfrac{\|x\|_\infty}{4}$. In particular we have $\dist(\Omega_3,\Omega')>\tfrac{1}{4}$.  We apply Lemma \ref{Lemma_kernel_set_distribution} with $M=\Omega_3$, $N=\Omega'$ and $\varphi$ given by $\varphi(x)=\tfrac{\|x\|_\infty}{4}$. We recall that $\dist(x,\partial H_k)\geq \min_{1\leq j \leq d}|x_j|$ and compute
\begin{multline}
\int_{\Omega_3}\frac{4^{d}}{\|x\|_\infty^{d}}\dd x =\sum_{l=1}^{d}\sum_{\mu\in\{1,\ldots,d\}\setminus\{l\}}\int_{\Omega_3\cap \{x\in\R^{d}: \ |x_l|= \|x\|_\infty\}\cap \{x\in\R^{d}: \ |x_\mu|= \min_{1\leq j\leq d}|x_j|\}}\frac{4^{d}}{\|x\|_\infty^{d}}\dd x \\
\leq d(d-1) \int_{1}^\infty \frac{4^{d}2^{d}x_l^{d-2}\phi(x_l)}{x_l^{d}}\dd x_l \leq \frac{C}{2}\int_{1}^\infty \frac{x_l^{d-\tfrac{3}{2}}}{x_l^{d}}\dd x_l = C,
\end{multline}
with a constant $C>0$, independent of $\Omega'$. Therefore, Lemma \ref{Lemma_kernel_set_distribution} yields the desired constant $C'>0$, independent of $\Omega'$, such that
\begin{equation}
\big\|\mathbf{1}_{\Omega_3}X_{1}\mathbf{1}_{\Omega'}\big\|_2\leq C'.
\end{equation}
This proves \eqref{Lemma-commutation-HS-bound-3} in the case $p=1$. We now turn to the case $p\geq 2$. For $m\in\{1,\ldots,p-1\}$, let $V_m\in\{H_k,H_{k-1}\setminus H_k\}$ be the set corresponding to the projection $P_m$. For $m\in\{1,\ldots,p-1\}$, we define the sets 
\begin{equation}
\Gamma_m:= V_{m}\cap\big(\partial H_{k}\big)_{\phi_m}\cap M_{\alpha_m}^c, 
\end{equation}
with $\alpha_m:=2-\tfrac{m}{2p}$ and $\phi_m:[0,\infty[\,\rightarrow [0,\infty[\,$ given by
\begin{equation}
\phi_m(x)=\tfrac{1}{8}\big(1+\tfrac{m}{p}\big)\max\big(0,x-p+m\big)^{1-\tfrac{1}{2d^{m}}+\epsilon},
\end{equation}
as in the proof of \eqref{Lemma-commutation-HS-bound-1}. We recall $\epsilon=\tfrac{1}{2d^{p}}$.
Again, we begin with the Hilbert-Schmidt norm of the operator $\mathbf{1}_{\Omega_3}X_1\mathbf{1}_{V_1\setminus\Gamma_1}$. Let $x\in\Omega_3$ and $y\in B_{\tfrac{\phi_1(\|x\|_\infty)}{4p}}(x)$, then we have, as in the proof of \eqref{Lemma-commutation-HS-bound-1}, that $y\in \big(\partial H_{k}\big)_{\phi_1}$. It remains to check that  also $y\in M_{\alpha_1}^c$. We estimate
\begin{multline}
\|y\|_\infty > \|x\|_\infty-\tfrac{\phi_1(\|x\|_\infty)}{4p}\geq \tfrac{(16p-1)\|x\|_\infty}{16p}\geq \tfrac{16p-3+1/(2p)}{16p}\|x\|_\infty+\big(2-\tfrac{1}{2p}\big)\tfrac{\phi_1(\|x\|_\infty)}{4p} \\
\geq \tfrac{32p-6+1/p}{16p}|x_k|+\big(2-\tfrac{1}{2p}\big)\tfrac{\phi_1(\|x\|_\infty)}{4p}\geq\big(2-\tfrac{1}{2p}\big)\big(|x_k|+\tfrac{\phi_1(\|x\|_\infty)}{4p}\big)\geq\big(2-\tfrac{1}{2p}\big)|y_k|,
\end{multline}
where we used that $\|x\|_\infty>p$.
Therefore, for every $x\in\Omega_3$, we have $\dist(x,V_1\setminus \Gamma_1)\geq \tfrac{\phi_1(\|x\|_\infty)}{4p}$. As $\|x\|_\infty>p$, for all $x\in \Omega_3$, we know in particular that $\dist(\Omega_3,V_1\setminus \Gamma_1)>\tfrac{1}{32p}$.  We now apply Lemma \ref{Lemma_kernel_set_distribution} with $M=\Omega_3$, $N=V_1\setminus \Gamma_1$ and $\varphi$ given by $\varphi(x)=\tfrac{\phi_1(\|x\|_\infty)}{4p}$. For the corresponding integral we compute, in the same way as in \eqref{Lemma-commutation-integral-1}, that
\begin{multline}
\int_{\Omega_3}\frac{(4p)^d}{\phi_1(\|x\|_\infty)^{d}}\dd x   \\
=\sum_{l=1}^{d}\sum_{\mu\in\{1,\ldots,d\}\setminus\{l\}}\int_{\Omega_3\cap \{x\in\R^{d}: \ |x_l|= \|x\|_\infty\}\cap \{x\in\R^{d}: \ |x_\mu|= \min_{1\leq j\leq d}|x_j|\}}\frac{(4p)^d}{\phi_1(\|x\|_\infty)^{d}}\dd x  
\leq \frac{Cp^{d}}{\epsilon},
\end{multline}
with a constant $C>0$, independent of $V_1$ and $p$. Therefore, Lemma \ref{Lemma_kernel_set_distribution} yields a constant $C_1'>0$, independent of $V_1$ and $p$, such that
\begin{equation}\label{Lemma-commutation-operator-HS-estimate-4}
\big\|\mathbf{1}_{\Omega_3}X_1\mathbf{1}_{V_1\setminus\Gamma_1}\big\|_2\leq C_1'\sqrt{\frac{p^{d}}{\epsilon}}.
\end{equation}
We continue with the intermediate terms. Let $m\in \{1,\ldots,p-2\}$, then we want to find a bound for the Hilbert-Schmidt norm of the operators $\mathbf{1}_{\Gamma_m}X_{m+1}\mathbf{1}_{V_{m+1}\setminus\Gamma_{m+1}}$. Let $x\in\Gamma_m$ and $y\in B_{\tfrac{\phi_{m+1}(\|x\|_\infty)}{8p}}(x)$, then we have, as in the proof of \eqref{Lemma-commutation-HS-bound-1}, that $y\in \big(\partial H_{k}\big)_{\phi_{m+1}}$. It remains to check that  also $y\in M_{\alpha_{m+1}}^c$. We estimate
\begin{multline}
\|y\|_\infty > \|x\|_\infty-\tfrac{\phi_{m+1}(\|x\|_\infty)}{8p}\geq \tfrac{(32p-1)\|x\|_\infty}{32p}\geq \tfrac{32p-3+(m+1)/(2p)}{32p}\|x\|_\infty+\big(2-\tfrac{m+1}{2p}\big)\tfrac{\phi_{m+1}(\|x\|_\infty)}{8p} \\
\geq \tfrac{64p-16m-6}{32p}|x_k|+\big(2-\tfrac{m+1}{2p}\big)\tfrac{\phi_{m+1}(\|x\|_\infty)}{8p} 
\geq\big(2-\tfrac{m+1}{2p}\big)\big(|x_k|+\tfrac{\phi_{m+1}(\|x\|_\infty)}{8p}\big)\geq\big(2-\tfrac{m+1}{2p}\big)|y_k|,
\end{multline}
where we used that $\|x\|_\infty>p-m$.
Therefore, for every $x\in\Gamma_m$, we have $\dist(x,V_{m+1}\setminus \Gamma_{m+1})\geq \tfrac{\phi_{m+1}(\|x\|_\infty)}{8p}$. As $\|x\|_\infty>p-m$, for all $x\in \Gamma_m$, we know in particular that $\dist(\Gamma_m,V_{m+1}\setminus \Gamma_{m+1})>\tfrac{1}{64p}$. We again apply Lemma \ref{Lemma_kernel_set_distribution}. Now with $M=\Gamma_m$, $N=V_{m+1}\setminus \Gamma_{m+1}$ and $\varphi$ given by $\varphi(x)=\tfrac{\phi_{m+1}(\|x\|_\infty)}{8p}$. We treat the corresponding integral as in \eqref{Lemma-commutation-integral-2} and obtain
\begin{equation}
\int_{\Gamma_m}\frac{(8p)^d}{\phi_{m+1}(\|x\|_\infty)^{d}}\dd x  
\leq \frac{Cp^{d}}{\epsilon},
\end{equation}
with a constant $C>0$, independent of $V_{m}$, $V_{m+1}$ and $p$. Therefore, Lemma \ref{Lemma_kernel_set_distribution} yields a constant $C_{m+1}'>0$, independent of $V_{m}$, $V_{m+1}$ and $p$, such that
\begin{equation}\label{Lemma-commutation-operator-HS-estimate-5}
\big\|\mathbf{1}_{\Gamma_{m}}X_{m+1}\mathbf{1}_{V_{m+1}\setminus\Gamma_{m+1}}\big\|_2\leq C_{m+1}'\sqrt{\frac{p^{d}}{\epsilon}}.
\end{equation}
It remains to estimate the last operator $\mathbf{1}_{\Gamma_{p-1}}X_{p}\mathbf{1}_{\Omega'}$.
Let $x\in \Gamma_{p-1}\subset \big(M_{3/2}\big)^c$ and $y\in B_{\tfrac{\|x\|_\infty}{6}}(x)$. Then we have
\begin{equation}
\|y\|_\infty>\tfrac{5}{6}\|x\|_\infty=\tfrac{4}{6}\|x\|_\infty+\tfrac{\|x\|_\infty}{6}>|x_k|+\tfrac{\|x\|_\infty}{6}>|y_k|,
\end{equation}
i.e. $y\notin M_{1}$. Therefore, we have $\dist(x,\Omega')\geq \dist(x,M_{1})\geq \tfrac{\|x\|_\infty}{6}$. In particular we have $\dist(\Gamma_{p-1},\Omega')>\tfrac{1}{6}$.  We apply Lemma \ref{Lemma_kernel_set_distribution} with $M=\Gamma_{p-1}$, $N=\Omega'$ and $\varphi$ given by $\varphi(x)=\tfrac{\|x\|_\infty}{6}$. We recall that $\dist(x,\partial H_k)\geq \min_{1\leq j \leq d}|x_j|$ and compute
\begin{multline}
\int_{\Gamma_{p-1}}\frac{6^{d}}{\|x\|_\infty^{d}}\dd x =\sum_{l=1}^{d}\sum_{\mu\in\{1,\ldots,d\}\setminus\{l\}}\int_{\Gamma_{p-1}\cap \{x\in\R^{d}: \ |x_l|= \|x\|_\infty\}\cap \{x\in\R^{d}: \ |x_\mu|= \min_{1\leq j\leq d}|x_j|\}}\frac{6^{d}}{\|x\|_\infty^{d}}\dd x \\
\leq d(d-1) \int_{1}^\infty \frac{12^{d}x_l^{d-2}\phi_{p-1}(x_l)}{x_l^{d}}\dd x_l \leq d(d-1)12^{d}\int_{1}^\infty x_l^{-1-\tfrac{1}{2d^{p-1}}+\epsilon}\dd x_l \leq C\frac{1}{\epsilon},
\end{multline}
with a constant $C>0$, independent of $V_{p-1}$, $\Omega'$ and $p$.
Therefore, Lemma \ref{Lemma_kernel_set_distribution} yields a constant $C_p'>0$, independent of $V_{p-1}$, $\Omega'$ and $p$, such that
\begin{equation}\label{Lemma-commutation-operator-HS-estimate-6}
\big\|\mathbf{1}_{\Gamma_{p-1}}X_{p}(\mathbf{1}_{H_{k-1}}-\mathbf{1}_{\Omega'})\big\|_2\leq C_p'\sqrt{\frac{1}{\epsilon}}.
\end{equation}
As in the proof of \eqref{Lemma-commutation-HS-bound-1}, we combine the estimates \eqref{Lemma-commutation-operator-HS-estimate-4}, \eqref{Lemma-commutation-operator-HS-estimate-5} and \eqref{Lemma-commutation-operator-HS-estimate-6} and use the definition of $\epsilon$ to obtain the bound \eqref{Lemma-commutation-HS-bound-3}. This proves \eqref{Lemma-commutation-HS-bound-3} for $p\geq 2$.

As in the proof of the bound \eqref{Lemma-commutation-HS-bound-2}, the bound \eqref{Lemma-commutation-HS-bound-4} reduces to \eqref{Lemma-commutation-HS-bound-3} through the reflection along the hyperplane orthogonal to the $k$th coordinate. This finishes the proof of the lemma.
\end{proof}

Now we combine the trace-class bound from Lemma \ref{Lemma-trace-class-sob} with the Hilbert-Schmidt bounds we just obtained and make use of the structure of the operator in question in order to commute the first $p-1$ occurrences of $\Op(\psi)$ to the right.
\begin{lem} \label{Lemma-commutation-combination}
Let $L\geq 1$ and $p\in\N$. Then there exists a constant $C>0$, independent of $L$ and $p$, such that
\begin{align}
\Big|\tr_{L^{2}(\R^{d})\otimes\C^n}\Big[\mathbf{1}_{H_{L,d,k}}\Big(&\big(\mathbf{1}_{H_{k}}\Op(\psi\mathcal{D})\mathbf{1}_{H_{k}}\big)^p-\big(\mathbf{1}_{H_{k-1}}\Op(\psi\mathcal{D})\mathbf{1}_{H_{k-1}}\big)^p \nonumber \\
&\ -\big(\mathbf{1}_{H_{k}}\Op(\mathcal{D})\mathbf{1}_{H_{k}}\big)^{p-1}\Op(\psi^{p}\mathcal{D})\mathbf{1}_{H_{k}} \nonumber \\
 &\ +\big(\mathbf{1}_{H_{k-1}}\Op(\mathcal{D})\mathbf{1}_{H_{k-1}}\big)^{p-1}\Op(\psi^{p}\mathcal{D})\mathbf{1}_{H_{k-1}}\Big)\Big]\Big|\leq C(2\sqrt{d})^{p-1}(p-1)^{5d+1}.
\end{align}
\end{lem}
\begin{proof} 
For $p=1$, there is nothing to show.
Similarly as in the proof of Lemma \ref{higher-order-monom}, we write
\begin{multline}
\mathbf{1}_{H_{L,d,k}}\Big(\big(\mathbf{1}_{H_{k-1}}\Op(\psi\mathcal{D})\mathbf{1}_{H_{k-1}}\big)^p-\big(\mathbf{1}_{H_{k}}\Op(\psi\mathcal{D})\mathbf{1}_{H_{k}}\big)^p\Big)\mathbf{1}_{H_{L,d,k}}\\
=\sum_{\pi=(\pi_{1},\ldots,\pi_{p-1})\in \{0,1\}^{p-1}, \pi\neq 0}\mathbf{1}_{H_{L,d,k}}\Big(\prod_{j=1}^{p-1}\Op(\psi\mathcal{D})P_{\pi(j)}\Big)\Op(\psi\mathcal{D})\mathbf{1}_{H_{L,d,k}},
\end{multline}
with $P_0=\mathbf{1}_{H_{k}}$ and $P_1=\mathbf{1}_{H_{k-1}\setminus H_{k}}$. For the remaining part of the proof, we will only consider one of these terms for a given $\pi\in \{0,1\}^{p-1}$ with $\pi\neq 0$. The estimates for the sum follow by the triangle inequality. The fact that $\pi\neq 0$ guarantees that there is at least one occurrence of the projection $\mathbf{1}_{H_{k-1}\setminus H_{k}}$. For the given $\pi$, we define $p_{1,\pi}:=\min \{j\in\{1,\ldots,p-1\}: \ \pi(j)=1\}$ and $p_{2,\pi}:=\max \{j\in\{1,\ldots,p-1\}: \ \pi(j)=1\}$.
The first step is to commute the first occurrence of $\psi$ to the right, i.e. we want to estimate 
\begin{equation}\label{Lemma-commutation-goal-1}
\Big|\tr_{L^{2}(\R^{d})\otimes\C^n}\Big[\mathbf{1}_{H_{L,d,k}}\Big(\Op(\psi\mathcal{D})P_{\pi(1)}-\Op(\mathcal{D})P_{\pi(1)}\Op(\psi)\Big)\Big(\prod_{j=2}^{p-1}\Op(\psi\mathcal{D})P_{\pi(j)}\Big)\Op(\psi\mathcal{D})\Big]\Big|.
\end{equation}
We write
\begin{align}
\Op(\psi\mathcal{D})P_{\pi(1)}&-\Op(\mathcal{D})P_{\pi(1)}\Op(\psi)\nonumber\\
&= \Op(\mathcal{D})\big[(\mathbf{1}-P_{\pi(1)}+P_{\pi(1)})\Op(\psi)P_{\pi(1)}-P_{\pi(1)}\Op(\psi)(\mathbf{1}-P_{\pi(1)}+P_{\pi(1)})\big]\nonumber\\
&= \Op(\mathcal{D})\big[(\mathbf{1}-P_{\pi(1)})\Op(\psi)P_{\pi(1)}-P_{\pi(1)}\Op(\psi)(\mathbf{1}-P_{\pi(1)})\big],
\end{align}
where $\mathbf{1}:=\mathbf{1}_{\R^d}$ is the corresponding multiplication operator on $L^2(\R^d)\otimes\C^n$.
With this at hand, we bound \eqref{Lemma-commutation-goal-1} from above by
\begin{align}\label{Lemma-commutation-goal-2}
&\,\Big|\tr_{L^{2}(\R^{d})\otimes\C^n}\Big[\mathbf{1}_{H_{L,d,k}}\Op(\mathcal{D})(\mathbf{1}-P_{\pi(1)})\Op(\psi)P_{\pi(1)}\Big(\prod_{j=2}^{p-1}\Op(\psi\mathcal{D})P_{\pi(j)}\Big)\Op(\psi\mathcal{D})\Big]\Big| \nonumber \\
+&\, \Big|\tr_{L^{2}(\R^{d})\otimes\C^n}\Big[\mathbf{1}_{H_{L,d,k}}\Op(\mathcal{D})P_{\pi(1)}\Op(\psi)(\mathbf{1}-P_{\pi(1)})\Big(\prod_{j=2}^{p-1}\Op(\psi\mathcal{D})P_{\pi(j)}\Big)\Op(\psi\mathcal{D})\Big]\Big|.
\end{align}
We will continue to estimate the first trace in \eqref{Lemma-commutation-goal-2}. The second trace works in the same way.
To do so, we choose the measurable function $\phi:[0,\infty[\,\rightarrow[0,\infty[\,$ with $\phi(x):=\tfrac{1}{8}\sqrt{\max(0,x-p)}$ and define the projections $P_{0,\phi}:=\mathbf{1}_{\big(\partial H_{k}\big)_{\phi}\cap H_k}$ and $P_{1,\phi}:=\mathbf{1}_{\big(\partial (H_{k-1}\setminus H_{k})\big)_{\phi}\cap (H_{k-1}\setminus H_{k})}$. Then Corollary \ref{Lemma-trace-class-sob} with $\beta=2d$ and $C_\beta=C_\beta'p^d$, with $C_\beta'>0$ independent of $p$, yields that 
\begin{equation}
\|(\mathbf{1}-P_{\pi(1)})\Op(\psi)\big(P_{\pi(1)}-P_{\pi(1),\phi}\big)\|_1\leq C_0 p^{d}
\end{equation}
with the constant $C_0>0$ being independent of $L$ and $p$. This allows us to replace the projection $P_{\pi(1)}$ on the right-hand side of $\Op(\psi)$ in the first line of \eqref{Lemma-commutation-goal-2} with the projection $P_{\pi(1),\phi}$ up to an error term of constant order. 
As explained earlier in this section, we split the sets $\big(\partial H_{k}\big)_{\phi}$ and $\big(\partial (H_{k-1}\setminus H_{k})\big)_{\phi}$ into two subsets, which are disjoint up to sets of measure zero. We choose the parameter $\alpha=\tfrac{5}{2}$ and define the corresponding projections
\begin{equation}
P_{\pi(1),\phi}^{(1)}:=P_{\pi(1),\phi}\mathbf{1}_{M_{5/2}} \qquad \text{and} \qquad P_{\pi(1),\phi}^{(2)}:=P_{\pi(1),\phi}\mathbf{1}_{(M_{5/2})^c}.
\end{equation}
We also define the projections
\begin{align}
&P_{0,\phi}^{(3)}:=\mathbf{1}_{\big(\partial H_{k}\big)_{\phi}}\mathbf{1}_{(H_k)^c}\mathbf{1}_{M_{3}}, \qquad \qquad \qquad \ \ P_{0,\phi}^{(4)}:=\mathbf{1}_{\big(\partial H_{k}\big)_{\phi}}\mathbf{1}_{(H_k)^c}\mathbf{1}_{(M_{2})^c}, \nonumber \\
&P_{1,\phi}^{(3)}:=\mathbf{1}_{\big(\partial (H_{k-1}\setminus H_{k})\big)_{\phi}}\mathbf{1}_{(H_{k-1}\setminus H_{k})^c}\mathbf{1}_{M_{3}}, \qquad P_{1,\phi}^{(4)}:=\mathbf{1}_{\big(\partial (H_{k-1}\setminus H_{k})\big)_{\phi}}\mathbf{1}_{(H_{k-1}\setminus H_{k})^c}\mathbf{1}_{(M_{2})^c}
\end{align}
and claim that the trace norms of the operators
\begin{align}\label{Lemma-commutation-smooth-decay-operators}
&\big(\mathbf{1}_{(H_{k})^c}-P_{0,\phi}^{(3)}\big)\Op(\psi)P_{0,\phi}^{(1)}, \qquad \qquad \qquad \big(\mathbf{1}_{(H_{k})^c}-P_{0,\phi}^{(4)}\big)\Op(\psi)P_{0,\phi}^{(2)}, \nonumber \\
&\big(\mathbf{1}_{(H_{k-1}\setminus H_{k})^c}-P_{1,\phi}^{(3)}\big)\Op(\psi)P_{0,\phi}^{(1)}, \qquad \qquad \big(\mathbf{1}_{(H_{k-1}\setminus H_{k})^c}-P_{1,\phi}^{(4)}\big)\Op(\psi)P_{1,\phi}^{(2)},
\end{align}
are bounded by a constant times $p^{d}$, which is independent of $L$ and $p$. We will only check this for the first operator, the other cases work in the same way. In order to apply Corollary \ref{Lemma-trace-class-sob}, we need to check the condition \eqref{Lemma-trace-class-sob-prep-req} for the sets $M=(H_k)^c\setminus \big(\big(\partial H_{k}\big)_{\phi}\cap(H_k)^c\cap M_{3}\big)$ and $N=M_{5/2} \cap \big(\partial H_{k}\big)_{\phi}\cap H_k$. We only do so in the case $\rho=1$, as another choice of $\rho$ would only yield an additional dependence on $\rho=b+1$ in the constant $C_\beta$ but we are not concerned with the dependence on $b$ in the present lemma.
Let $x\in (H_k)^c\setminus \big(\big(\partial H_{k}\big)_{\phi}\cap(H_k)^c\cap M_{3}\big)$, then we consider two cases. In the first case we have $x\in (H_k)^c$ and $x\notin \big(\partial H_{k}\big)_{\phi}$ and therefore
\begin{equation}\label{Lemma-commutation-smooth-decay-2}
\dist(x,M_{5/2} \cap \big(\partial H_{k}\big)_{\phi}\cap H_k)\geq\dist(x,H_k)=\dist(x,\partial H_k)\geq\phi(\|x\|_\infty)=\frac{\sqrt{\max(0,\|x\|_\infty-p)}}{8}.
\end{equation}
In the second case we have $x\notin M_3$. Suppose now that $y\in B_{\tfrac{\|x\|_\infty}{21}}(x)$, then we have 
\begin{equation}
\|y\|_\infty>\tfrac{20}{21}\|x\|_\infty=\tfrac{5}{6}\|x\|_\infty+\tfrac{5}{2}\tfrac{\|x\|_\infty}{21}>\tfrac{5}{2}\big(|x_k|+\tfrac{\|x\|_\infty}{21}\big)>\tfrac{5}{2}|y_k|,
\end{equation}
i.e. $y\notin M_{5/2}$. Therefore, in this case  we have
\begin{equation}\label{Lemma-commutation-smooth-decay-3}
\dist(x,M_{5/2} \cap \big(\partial H_{k}\big)_{\phi}\cap H_k)\geq\dist(x,M_{5/2})\geq\tfrac{\|x\|_\infty}{21}.
\end{equation}
Combining \eqref{Lemma-commutation-smooth-decay-2} and \eqref{Lemma-commutation-smooth-decay-3}, we find a constant $C_\beta'>0$, independent of $L$ and $p$, such that the first operator in \eqref{Lemma-commutation-smooth-decay-operators} satisfies the requirements of Corollary \ref{Lemma-trace-class-sob} with $\beta=2d$ and $C_\beta=C_\beta'p^d$. Checking this for the remaining operators in \eqref{Lemma-commutation-smooth-decay-operators}, Corollary \ref{Lemma-trace-class-sob} yields that the trace norms of the operators in \eqref{Lemma-commutation-smooth-decay-operators} are indeed bounded by a constant times $p^{d}$ independent of $L$ and $p$.
With this at hand, we estimate the first trace in \eqref{Lemma-commutation-goal-2} by 
\begin{multline}\label{Lemma-commutation-goal-3}
\sum_{l=1}^{2}\Big|\tr_{L^{2}(\R^{d})\otimes\C^n}\Big[\mathbf{1}_{H_{L,d,k}}\Op(\mathcal{D})P_{\pi(1),\phi}^{(l+2)}\Op(\psi)P_{\pi(1),\phi}^{(l)}\Big(\prod_{j=2}^{p-1}\Op(\psi\mathcal{D})P_{\pi(j)}\Big)\Op(\psi\mathcal{D})\Big]\Big| \\
+C_0'p^{d}.
\end{multline}
with $C_0'$, independent of $L$ and $p$.

We now estimate the individual terms in \eqref{Lemma-commutation-goal-3} by using the cyclic property of the trace to rewrite them as products of Hilbert-Schmidt norms and then applying Lemma \ref{Lemma-commutation-HS}. Here, we have to distinguish between several different cases. In order to apply Lemma \ref{Lemma-commutation-HS}, we note that we have $\big(\partial H_{k}\big)_{\phi}\subset \big(\partial H_{k}\big)_{\phi_m}$ and $\big(\partial (H_{k-1}\setminus H_{k})\big)_{\phi}\subset \big(\partial (H_{k-1}\setminus H_{k})\big)_{\phi_m}$ for all $m\in\{1,\ldots,p-1\}$, where $\phi_m:[0,\infty[\,\rightarrow[0,\infty[\,$ with $\phi_m(x):=\tfrac{1}{8}\sqrt{\max(0,x-m)}$. We also recall the definitions $p_{1,\pi}=\min \{j\in\{1,\ldots,p-1\}: \ \pi(j)=1\}$ and $p_{2,\pi}=\max \{j\in\{1,\ldots,p-1\}: \ \pi(j)=1\}$.
\begin{enumerate}[(i)]
\item If $l=1$ and $\pi(1)=0$, we estimate the corresponding term in \eqref{Lemma-commutation-goal-3} by
\begin{multline}
\Big\|P_{\pi(p_{2,\pi})}\Big(\prod_{j=p_{2,\pi}+1}^{p-1}\Op(\psi\mathcal{D})P_{\pi(j)}\Big)\Op(\psi\mathcal{D})\mathbf{1}_{H_{L,d,k}}\Op(\mathcal{D})P_{0,\phi}^{(3)}\Big\|_2 \\ 
\times \quad \Big\|P_{0,\phi}^{(1)}
\Big(\prod_{j=2}^{p_{1,\pi}-1}\Op(\psi\mathcal{D})P_{\pi(j)}\Big)\Op(\psi\mathcal{D})P_{\pi(p_{1,\pi})}\Big\|_2
\end{multline}
and apply \eqref{Lemma-commutation-HS-bound-1} to obtain the bound
\begin{equation}
C_1 d^{\tfrac{p-1}{2}}(p-1)^{d+2},
\end{equation}
where the constant $C_1>0$ is independent of $L$ and $p$.
\item If $l=2$ and $\pi(1)=0$, we estimate the corresponding term in \eqref{Lemma-commutation-goal-3} by
\begin{equation}
\big\|\mathbf{1}_{H_{L,d,k}}\Op(\mathcal{D})P_{0,\phi}^{(4)}\big\|_2  \quad \times \quad \Big\|P_{0,\phi}^{(2)}
\Big(\prod_{j=2}^{p-1}\Op(\psi\mathcal{D})P_{\pi(j)}\Big)\Op(\psi\mathcal{D})\mathbf{1}_{H_{L,d,k}}\Big\|_2
\end{equation}
and apply \eqref{Lemma-commutation-HS-bound-3} to obtain the same bound as in (i).
\item If $l=1$ and $\pi(1)=1$, we estimate the corresponding term in \eqref{Lemma-commutation-goal-3} by
\begin{equation}
\big\|\mathbf{1}_{H_{L,d,k}}\Op(\mathcal{D})P_{1,\phi}^{(3)}\big\|_2  \quad \times \quad \Big\|P_{1,\phi}^{(1)}
\Big(\prod_{j=2}^{p_{1,\pi,0}-1}\Op(\psi\mathcal{D})P_{\pi(j)}\Big)\Op(\psi\mathcal{D})P_{\pi(p_{1,\pi,0})}\Big\|_2,
\end{equation}
where $p_{1,\pi,0}$ is the smallest $j\in \{2,\ldots,p-1\}$ with $\pi(p_{1,\pi,0})=0$, in the case that it exists. Otherwise we set $p_{1,\pi,0}=p$ and $P_{p_{1,\pi,0}}=\mathbf{1}_{H_{L,d,k}}$. Then we apply \eqref{Lemma-commutation-HS-bound-2} to obtain the same bound as in (i).
\item If $l=2$ and $\pi(1)=1$, we estimate the corresponding term in \eqref{Lemma-commutation-goal-3} by
\begin{equation}
\big\|\mathbf{1}_{H_{L,d,k}}\Op(\mathcal{D})P_{1,\phi}^{(4)}\big\|_2  \quad \times \quad \Big\|P_{1,\phi}^{(2)}
\Big(\prod_{j=2}^{p-1}\Op(\psi\mathcal{D})P_{\pi(j)}\Big)\Op(\psi\mathcal{D})\mathbf{1}_{H_{L,d,k}}\Big\|_2
\end{equation}
and apply \eqref{Lemma-commutation-HS-bound-4} to obtain the same bound as in (i).
\end{enumerate}
As we obtain the same bound in all of the cases, we see that \eqref{Lemma-commutation-goal-3}, and with it \eqref{Lemma-commutation-goal-1}, are bounded from above by 
\begin{equation}
C_1' d^{\tfrac{p-1}{2}}(p-1)^{d+2},
\end{equation}
with $C_1'>0$ independent of $L$ and $p$. 
The commutation to the right of the following $p-2$ occurrences of $\Op(\psi^n)$, for $n\in\{2,\ldots,p-1\}$ works in an analogous way. 
Although the application of Corollary \ref{Lemma-trace-class-sob} now yields a higher power in $p$, as we now apply it to the function $\psi^j$ for $j\in\{2,\ldots,p-1\}$.
Therefore, we obtain the bound
\begin{equation}\label{Lemma-commutation-goal-4}
\Big|\tr_{L^{2}(\R^{d})\otimes\C^n}\Big[\mathbf{1}_{H_{L,d,k}}\Big(\prod_{j=1}^{p-1}\Op(\mathcal{D})P_{\pi(j)}\Big)\Op(\psi^{p}\mathcal{D})\Big]\Big|\leq C d^{\tfrac{p-1}{2}}(p-1)^{5d+1},
\end{equation}
with a constant $C>0$, independent of $L$ and $p$. Summing up the contributions of all $\pi\in \{0,1\}^{p-1}$ with $\pi\neq 0$, we obtain a constant $C>0$, independent of $L$ and $p$, such that
\begin{align}
\Big|\tr_{L^{2}(\R^{d})\otimes\C^n}\Big[\mathbf{1}_{H_{L,d,k}}\Big(&\big(\mathbf{1}_{H_{k}}\Op(\psi\mathcal{D})\mathbf{1}_{H_{k}}\big)^p-\big(\mathbf{1}_{H_{k-1}}\Op(\psi\mathcal{D})\mathbf{1}_{H_{k-1}}\big)^p \nonumber \\
&\ -\big(\mathbf{1}_{H_{k}}\Op(\mathcal{D})\mathbf{1}_{H_{k}}\big)^{p-1}\Op(\psi^{p}\mathcal{D})\mathbf{1}_{H_{k}} \nonumber \\
 &\ +\big(\mathbf{1}_{H_{k-1}}\Op(\mathcal{D})\mathbf{1}_{H_{k-1}}\big)^{p-1}\Op(\psi^{p}\mathcal{D})\mathbf{1}_{H_{k-1}}\Big)\Big]\Big|\leq C(2\sqrt{d})^{p-1}(p-1)^{5d+1}.
\end{align}
This concludes the proof of the lemma.
\end{proof}
In order to find an upper bound for \eqref{Commutation-section-goal}, it remains to commute $\Op(\psi^p)$ with the projection onto the set $H_{L,d,k}$. To do so, we will first split the set $H_{L,d,k}=\bigcup_{x_k\in [0,L]} [0,x_{k}]^{k-1}\times \{x_{k}\}\times [0,x_{k}]^{d-k}$ into the two disjoint parts
\begin{equation}
H_{L,d,k}= \{x\in H_{L,d,k}: x_k\leq \tfrac{L}{2}\}\cup \{x\in H_{L,d,k}: x_k>\tfrac{L}{2}\}=:H_{L,d,k}^{(1)}\cup H_{L,d,k}^{(2)}.
\end{equation}
We note that the set $H_{L,d,k}^{(2)}$ has a measure of order $L^{d}$ and a distance of order $L$ to the set $H_{k-1}\setminus H_{k}$. Therefore, we later deal with it using Lemma \ref{Lemma-HS-distance-L}. For the set  $H_{L,d,k}^{(1)}$, we define, given a measurable function $\phi:[0,\infty[\,\rightarrow[0,\infty[\,$, a "thickened up" version of the boundary of its infinite version, i.e.~of the set $H_{\infty,d,k}$, by 
\begin{align}
\big(\partial H_{\infty,d,k}\big)_{\phi}:= &\bigcup_{x_k\in [0,\infty[\,}\,]-\phi(x_{k}),x_k+\phi(x_{k})[^{k-1}\times \{x_{k}\}\times \,]-\phi(x_{k}),x_k+\phi(x_{k})[^{d-k}\nonumber \\
&\quad\setminus \bigcup_{x_k\in [0,\infty[\,}\,]\phi(x_{k}),x_k-\phi(x_{k})[^{k-1}\times \{x_{k}\}\times \,]\phi(x_{k}),x_k-\phi(x_{k})[^{d-k}.
\end{align}
We define the part of this "thickened up" boundary associated to $H_{L,d,k}^{(1)}$ by 
\begin{equation}
\big(\partial H_{L,d,k}\big)_{\phi}:=\big\{x\in\big(\partial H_{\infty,d,k}\big)_{\phi}:\ x_k\in [0,\tfrac{L}{2}]\big\}.
\end{equation}
Next, we provide an estimate for a Hilbert-Schmidt norm involving this set.

\begin{lem}\label{Lemma-commutation-last}
Let $p\in\N$ and $X_1,\ldots, X_p$ be bounded translation-invariant integral operators, each satisfying the estimate \eqref{kernel_bound_distribution} and such that $\|X_m\|\leq 1$ for each $m\in\{1,\ldots ,p\}$. Let $k\in\{1,\ldots,d\}$. Further let $\phi:[0,\infty[\,\rightarrow[0,\infty[\,$ be given by $\phi(x):=\tfrac{1}{p^2}\sqrt{\max(0,px-p)}$. Then, for every measurable $\Omega\subseteq \big(\partial H_{\infty,d,k}\big)_{\phi}$, there exists a constant $C>0$, independent of $\Omega$ and $p$, such that
\begin{equation}\label{Lemma-commutation-last-bound}
\big\|\mathbf{1}_{\Omega}X_1\Big(\prod_{m=1}^{p-1}\mathbf{1}_{H_k}X_{m+1}\Big)(\mathbf{1}_{H_{k-1}}-\mathbf{1}_{H_{k}})\big\|_2\leq Cd^{\tfrac{p}{2}} p^{\tfrac{3d+2}{2}}.
\end{equation}
\end{lem}
\begin{proof}
The proof is similar to the one of Lemma \ref{Lemma-commutation-HS}.
As usual we start with the case $p=1$. For every 
$x\in \Omega$, we clearly have $x_k\geq 0$ and therefore $\dist(x,H_{k-1}\setminus H_k)\geq x_k$. In particular we have $\dist(\Omega, H_{k-1}\setminus H_k)> 1$, as in the case $x\in \Omega$ with $x_k\leq 1$, we would have $x\notin \big(\partial H_{\infty,d,k}\big)_{\phi}$. We apply Lemma \ref{Lemma_kernel_set_distribution} with $M=\Omega$, $N=H_{k-1}\setminus H_k$ and $\varphi$ given by $\varphi(x)=x_k$. We note that we have $\phi(x_k)\leq \sqrt{x_k}$ and compute
\begin{equation}
\int_{\Omega}\frac{1}{x_k^{d}}\dd x 
\leq C_0\int_{1}^\infty \frac{x_k^{d-2}\phi(x_k)}{x_k^{d}}\dd x_k = C_0',
\end{equation}
where the constants $C_0$, $C_0'$ are independent of $\Omega$.
Therefore, Lemma \ref{Lemma_kernel_set_distribution} yields a constant $C>0$, independent of $\Omega$, such that
\begin{equation}
\big\|\mathbf{1}_{\Omega}X_{1}(\mathbf{1}_{H_{k-1}}-\mathbf{1}_{H_{k}})\big\|_2\leq C.
\end{equation}
This proves the lemma in the case $p=1$.
We now turn to the case $p\geq 2$. For $m\in\{1,\ldots,p-1\}$, we define the sets 
\begin{equation}
\Gamma_m:= H_k\cap\big(\partial H_{\infty,d,k}\big)_{\phi_m}, 
\end{equation}
with $\phi_m$ given by
\begin{equation}
\phi_m(x):=\frac{m+1}{p^2}\max\big(0,px-p+m\big)^{1-\tfrac{1}{2d^{m}}+\epsilon},
\end{equation} 
where we set $\epsilon:=\tfrac{1}{2d^{p}}$ for the remaining part of the proof.
We first consider the Hilbert-Schmidt norm of the operator $\mathbf{1}_{\Omega}X_1\mathbf{1}_{H_k\setminus\Gamma_1}$. Let $x\in\Omega$ and $y\in B_{\tfrac{\phi_1(x_k)}{2p}}(x)$, then we have
\begin{equation}
y_k > x_k-\tfrac{\phi_1(x_k)}{2p}\geq x_k - \tfrac{px_k-p+1}{2p^2} =\tfrac{(2p^2-p)x_k+(p-1)}{2p^2},
\end{equation}
where we used that $x_k>1$.
A short computation yields 
\begin{equation}
\tfrac{2}{p^2}\tfrac{2p^2-p}{2p^2}=\tfrac{4p^2-2p}{2p^4}> \tfrac{1}{p^2}+\tfrac{2}{2p^3}.
\end{equation}
With this, the monotonicity of $\phi_1$ and its definition, we obtain the bound
\begin{multline}
\phi_1(y_k) \geq \phi_1\big(\tfrac{(2p^2-p)x_k+(p-1)}{2p^2}\big)\geq \big(\tfrac{1}{p^2}+\tfrac{2}{2p^3}\big)\tfrac{2p^2}{2p^2-p}\max\big(0,\tfrac{(2p^2-p)x_k+(p-1)}{2p}-(p-1)\big)^{1-\tfrac{1}{2d}+\epsilon} \\
\geq \tfrac{1}{p^2}(px_k-p+1)^{1-\tfrac{1}{2d}+\epsilon}+\tfrac{2}{2p^3}(px_k-p+1)^{1-\tfrac{1}{2d}+\epsilon}\geq \phi(x_k)+\tfrac{\phi_1(x_k)}{2p}.
\end{multline}
We conclude 
\begin{equation}
\dist(y,\partial H_{\infty,d,k})\leq \dist(x,\partial  H_{\infty,d,k})+\tfrac{\phi_1(x_k)}{2p}< \phi(x_k)+\tfrac{\phi_1(x_k)}{2p}<\phi_1(y_k),
\end{equation}
i.e. $y\in \big(\partial H_{\infty,d,k}\big)_{\phi_1}$. Therefore, for every $x\in\Omega$, we have $\dist(x,H_k\setminus \Gamma_1)\geq \tfrac{\phi_1(x_k)}{2p}$. As $x_k>1$ for all $x\in \Omega$, we know in particular that $\dist(\Omega,H_k\setminus \Gamma_1)>\tfrac{1}{p^3}$.  We now apply Lemma \ref{Lemma_kernel_set_distribution} with $M=\Omega$, $N=H_k\setminus \Gamma_1$ and $\varphi$ given by $\varphi(x)=\tfrac{\phi_1(x_k)}{2p}$. Again we have $\phi(x_k)\leq \sqrt{x_k}$ and we estimate the corresponding integral by
\begin{equation}
\int_{\Omega}\frac{(2p)^d}{\phi_1(x_k)^{d}}\dd x  \leq C' p^d\int_{1}^\infty \frac{x_k^{d-2}\phi(x_k)}{\phi_1(x_k)^{d}}\dd x_k 
\leq C' p^{3d}\int_{1}^\infty x_k^{d-\tfrac{3}{2}-d+\tfrac{1}{2}-d\epsilon}\dd x_k  
 \leq  C''\frac{p^{3d}}{\epsilon},
\end{equation}
with constants $C',C''>0$, independent of $\Omega$ and $p$. Therefore, Lemma \ref{Lemma_kernel_set_distribution} yields a constant $C_1>0$, independent of $\Omega$ and $p$, such that
\begin{equation}\label{Lemma-commutation-last-operator-HS-estimate-1}
\big\|\mathbf{1}_{\Omega}X_1\mathbf{1}_{H_k\setminus\Gamma_1}\big\|_2\leq C_1\sqrt{\frac{p^{3d}}{\epsilon}}.
\end{equation}
We continue with the intermediate terms. Let $m\in \{1,\ldots,p-2\}$, then we want to find a bound for the Hilbert-Schmidt norm of the operators $\mathbf{1}_{\Gamma_m}X_{m+1}\mathbf{1}_{H_k\setminus\Gamma_{m+1}}$. Let $x\in\Gamma_m$ and $y\in B_{\tfrac{\phi_{m+1}(x_k)}{2p}}(x)$, then we have
\begin{equation}
y_k > x_k-\tfrac{\phi_{m+1}(x_k)}{2p}\geq x_k - \tfrac{px_k-p+m+1}{2p^2} =\tfrac{(2p^2-p)x_k+(p-m-1)}{2p^2},
\end{equation}
where we used that $x_k>\tfrac{p-m}{p}$.
A short computation yields 
\begin{equation}
\tfrac{m+2}{p^2}\tfrac{2p^2-p}{2p^2}=\tfrac{2(m+2)p^2-(m+2)p}{2p^4}>\tfrac{2(m+1)p^2+(m+2)p}{2p^4} = \tfrac{m+1}{p^2}+\tfrac{m+2}{2p^3}.
\end{equation}
With this, the monotonicity of $\phi_{m+1}$ and its definition, we obtain the bound
\begin{align}
\phi_{m+1}(y_k) &\geq \phi_{m+1}\big(\tfrac{(2p^2-p)x_k+(p-m-1)}{2p^2}\big) \nonumber \\
&\geq \Big(\tfrac{m+1}{p^2}+\tfrac{m+2}{2p^3}\Big)\tfrac{2p^2}{2p^2-p}\max\big(0,\tfrac{(2p^2-p)x_k+(p-m-1)}{2p}-(p-m-1)\big)^{1-\tfrac{1}{2d^{m+1}}+\epsilon} \nonumber\\
&\geq \tfrac{m+1}{p^2}(px_k-p+m+1)^{1-\tfrac{1}{2d^{m+1}}+\epsilon}+\tfrac{m+2}{2p^3}(px_k-p+m+1)^{1-\tfrac{1}{2d^{m+1}}+\epsilon}\nonumber\\
&> \phi_{m}(x_k)+\tfrac{\phi_{m+1}(x_k)}{2p}.
\end{align}
We conclude 
\begin{equation}
\dist(y,\partial H_{\infty,d,k})\leq \dist(x,\partial  H_{\infty,d,k})+\tfrac{\phi_{m+1}(x_k)}{2p}< \phi_{m}(x_k)+\tfrac{\phi_{m+1}(x_k)}{2p}<\phi_{m+1}(y_k),
\end{equation}
i.e. $y\in \big(\partial H_{\infty,d,k}\big)_{\phi_{m+1}}$. Therefore, for every $x\in\Gamma_m$, we have $\dist(x,H_k\setminus \Gamma_{m+1})\geq \tfrac{\phi_{m+1}(x_k)}{2p}$. As $x_k>\tfrac{p-m}{p}$, for all $x\in \Gamma_m$, we know in particular that $\dist(\Gamma_m,H_k\setminus \Gamma_{m+1})>\tfrac{1}{p^3}$.
We again apply Lemma \ref{Lemma_kernel_set_distribution}. Now with $M=\Gamma_m$, $N=H_k\setminus \Gamma_{m+1}$ and $\varphi$ given by $\varphi(x)=\tfrac{\phi_{m+1}(x_k)}{2p}$. We estimate the corresponding integral 
\begin{multline}
\int_{\Gamma_m}\frac{(2p)^d}{\phi_{m+1}(x_k)^{d}}\dd x  \leq C' p^d\int_{\tfrac{p-m}{p}}^\infty \frac{x_k^{d-2}\phi_{m}(x_k)}{\phi_{m+1}(x_k)^{d}}\dd x_k 
\leq C' p^{3d}\int_{\tfrac{p-m}{p}}^\infty x_k^{d-1-\tfrac{1}{2d^m}+\epsilon-d+\tfrac{1}{2d^m}-d\epsilon}\dd x_k \\ 
\leq C''\frac{p^{3d}}{\epsilon},
\end{multline}
with constants $C',C''>0$, independent of $\Omega$ and $p$. Therefore, Lemma \ref{Lemma_kernel_set_distribution} yields a constant $C_{m+1}>0$, independent of $\Omega$ and $p$, such that
\begin{equation}\label{Lemma-commutation-last-operator-HS-estimate-2}
\big\|\mathbf{1}_{\Gamma_{m}}X_{m+1}\mathbf{1}_{V_{m+1}\setminus\Gamma_{m+1}}\big\|_2\leq C_{m+1}\sqrt{\frac{p^{3d}}{\epsilon}}.
\end{equation}
It remains to estimate the last operator $\mathbf{1}_{\Gamma_{p-1}}X_{p}(\mathbf{1}_{H_{k-1}}-\mathbf{1}_{H_{k}})$.

For every 
 $x\in \Gamma_{p-1}$, we have $x_k\geq 0$ and therefore $\dist(x,H_{k-1}\setminus H_k)\geq x_k$. In particular we have $\dist(\Gamma_{p-1}, H_{k-1}\setminus H_k)> \tfrac{1}{p}$, as in the case $x\in \Gamma_{p-1}$ with $x_k\leq \tfrac{1}{p}$, we would have $x\notin \big(\partial H_{\infty,d,k}\big)_{\phi_{p-1}}$. We apply Lemma \ref{Lemma_kernel_set_distribution} with $M=\Gamma_{p-1}$, $N=H_{k-1}\setminus H_k$ and $\varphi$ given by $\varphi(x)=x_k$. We compute
\begin{equation}
\int_{\Gamma_{p-1}}\frac{1}{x_k^{d}}\dd x 
\leq C'\int_{\tfrac{1}{p}}^\infty \frac{x_k^{d-2}\phi_{p-1}(x_k)}{x_k^{d}}\dd x_k \leq C''\frac{1}{\epsilon},
\end{equation}
with constants $C',C''>0$, independent of $\Omega$ and $p$.
Therefore, Lemma \ref{Lemma_kernel_set_distribution} yields a constant $C_p>0$, independent of $\Omega$ and $p$, such that
\begin{equation}\label{Lemma-commutation-last-operator-HS-estimate-3}
\big\|\mathbf{1}_{\Gamma_{p-1}}X_{p}(\mathbf{1}_{H_{k-1}}-\mathbf{1}_{H_{k}})\big\|_2\leq C_p\sqrt{\frac{1}{\epsilon}}.
\end{equation}
As in \eqref{Lemma-commutation-Hölder}, we combine the estimates \eqref{Lemma-commutation-last-operator-HS-estimate-1}, \eqref{Lemma-commutation-last-operator-HS-estimate-2} and \eqref{Lemma-commutation-last-operator-HS-estimate-3} and use the definition of $\epsilon$ to obtain the bound \eqref{Lemma-commutation-last-bound}. This proves the lemma for $p\geq 2$.
\end{proof}
In a similar way as in Lemma \ref{Lemma-commutation-combination}, we now use the derived Hilbert-Schmidt bound  to also commute the last occurrence of $\Op(\psi^p)$ to the right.
\begin{lem}\label{Lemma-commutation-last-combination}
 Let $L\geq 1$ and $p\in\N$. Then there exists a constant $C>0$, independent of $L$ and $p$, such that
\begin{align}
\Big|\tr_{L^{2}(\R^{d})\otimes\C^n}\Big[&\mathbf{1}_{H_{V,L}}\Big(\big(\mathbf{1}_{H_{k}}\Op(\psi\mathcal{D})\mathbf{1}_{H_{k}}\big)^p-\big(\mathbf{1}_{H_{k-1}}\Op(\psi\mathcal{D})\mathbf{1}_{H_{k-1}}\big)^p\Big)\mathbf{1}_{H_{L,d,k}} \nonumber\\
& -\mathbf{1}_{H_{V,L}}\Big(\big(\mathbf{1}_{H_{k}}\Op(\mathcal{D})\mathbf{1}_{H_{k}}\big)^p-\big(\mathbf{1}_{H_{k-1}}\Op(\mathcal{D})\mathbf{1}_{H_{k-1}}\big)^p\Big)\mathbf{1}_{H_{L,d,k}}\Op(\psi^{p}\otimes\mathbb{1}_{n})\Big]\Big|\nonumber\\
& \qquad\qquad\qquad\qquad\qquad\qquad\qquad\qquad\qquad\qquad \leq C(2\sqrt{d})^{p-1}p^{7d+1}.
\end{align}
\end{lem}
\begin{proof}
For $p=1$, there is nothing to show.
After an application of Lemma \ref{Lemma-commutation-combination} it only remains to commute the remaining occurrence of $\Op(\psi^d)$ with $\mathbf{1}_{H_{L,d,k}}$, i.e. we want to find a bound for
\begin{multline}\label{Lemma-commutation-last-combination-goal-1}
\Big|\tr_{L^{2}(\R^{d})\otimes\C^n}\Big[\Big(\big(\mathbf{1}_{H_{k}}\Op(\mathcal{D})\mathbf{1}_{H_{k}}\big)^{p-1}\Op(\psi^{p}\mathcal{D}) -\big(\mathbf{1}_{H_{k-1}}\Op(\mathcal{D})\mathbf{1}_{H_{k-1}}\big)^{p-1}\Op(\psi^{p}\mathcal{D})\Big)\mathbf{1}_{H_{L,d,k}} \\
 -\mathbf{1}_{H_{V,L}}\Big(\big(\mathbf{1}_{H_{k}}\Op(\mathcal{D})\mathbf{1}_{H_{k}}\big)^p-\big(\mathbf{1}_{H_{k-1}}\Op(\mathcal{D})\mathbf{1}_{H_{k-1}}\big)^p\Big)\mathbf{1}_{H_{L,d,k}}\Op(\psi^{p}\otimes\mathbb{1}_{n})\Big]\Big|.
\end{multline}
As in the proof of Lemma \ref{Lemma-commutation-combination} we write
\begin{multline}
\mathbf{1}_{H_{V,L}}\Big(\big(\mathbf{1}_{H_{k-1}}\Op(\mathcal{D})\mathbf{1}_{H_{k-1}}\big)^{p-1}-\big(\mathbf{1}_{H_{k}}\Op(\mathcal{D})\mathbf{1}_{H_{k}}\big)^{p-1}\Big)\Op(\psi^{p}\mathcal{D})\mathbf{1}_{H_{L,d,k}}\\
=\sum_{\pi=(\pi_{1},\ldots,\pi_{p-1})\in \{0,1\}^{p-1}, \pi\neq 0}\mathbf{1}_{H_{V,L}}\Big(\prod_{j=1}^{p-1}\Op(\mathcal{D})P_{\pi(j)}\Big)\Op(\psi^p\mathcal{D})\mathbf{1}_{H_{L,d,k}},
\end{multline}
with $P_0=\mathbf{1}_{H_{k}}$ and $P_1=\mathbf{1}_{H_{k-1}\setminus H_{k}}$. For the remaining part of the proof, we will only consider one of these terms for a given $\pi\in \{0,1\}^{p-1}$ with $\pi\neq 0$. The estimates for the sum follow by the triangle inequality. The fact that $\pi\neq 0$ guarantees that there is at least one occurrence of the projection $\mathbf{1}_{H_{k-1}\setminus H_{k}}$. For the given $\pi$, we define $p_{1,\pi}:=\min \{j\in\{1,\ldots,p-1\}: \ \pi(j)=1\}$ and $p_{2,\pi}:=\max \{j\in\{1,\ldots,p-1\}: \ \pi(j)=1\}$. In order to commute the operator $\Op(\psi^d)$ to the right for the given $\pi$ we need to estimate
\begin{equation}\label{Lemma-commutation-last-combination-goal-2}
\Big|\tr_{L^{2}(\R^{d})\otimes\C^n}\Big[\mathbf{1}_{H_{V,L}}\Big(\prod_{j=1}^{p-1}\Op(\mathcal{D})P_{\pi(j)}\Big)\Op(\mathcal{D})\Big(\Op(\psi^{p}\otimes\mathbb{1}_{n})\mathbf{1}_{H_{L,d,k}}-\mathbf{1}_{H_{L,d,k}}\Op(\psi^{p}\otimes\mathbb{1}_{n})\Big)\Big]\Big|.
\end{equation}
We write
\begin{align}
\Op(\psi^p &\mathcal{D})\mathbf{1}_{H_{L,d,k}}-\Op(\mathcal{D})\mathbf{1}_{H_{L,d,k}}\Op(\psi^{p}\otimes\mathbb{1}_{n})\nonumber\\
&= \Op(\mathcal{D})\big[(\mathbf{1}-\mathbf{1}_{H_{L,d,k}})\Op(\psi^{p}\otimes\mathbb{1}_{n})\mathbf{1}_{H_{L,d,k}}-\mathbf{1}_{H_{L,d,k}}\Op(\psi^{p}\otimes\mathbb{1}_{n})(\mathbf{1}-\mathbf{1}_{H_{L,d,k}})\big].
\end{align}
By the triangle inequality it suffices to only study one of these terms, as the other one works in an analogous way. Therefore, we continue by a finding a bound for
\begin{equation}\label{Lemma-commutation-last-combination-goal-3}
\Big|\tr_{L^{2}(\R^{d})\otimes\C^n}\Big[\mathbf{1}_{H_{V,L}}\Big(\prod_{j=1}^{p-1}\Op(\mathcal{D})P_{\pi(j)}\Big)\Op(\mathcal{D})(\mathbf{1}-\mathbf{1}_{H_{L,d,k}})\Op(\psi^{p}\otimes\mathbb{1}_{n})\mathbf{1}_{H_{L,d,k}}\Big]\Big|.
\end{equation}
We now replace the projections $\mathbf{1}_{H_{L,d,k}}$ and $(\mathbf{1}-\mathbf{1}_{H_{L,d,k}})$ by projections which correspond to "thickened up" versions of $\partial H_{L,d,k}$ with the help of Corollary \ref{Lemma-trace-class-sob}. To do so, let $\phi:=\phi_p:[0,\infty[\,\rightarrow[0,\infty[\,$ be given by $\phi(x):=\phi_p(x):=\tfrac{1}{p^2}\sqrt{\max(0,px-p)}$. We define the projections
\begin{equation}
P_\phi^{(1)}:= \mathbf{1}_{H_{L,d,k}}-\mathbf{1}_{\{x\in H_{L,d,k}: \   x\in (\partial H_{L,d,k})_{\phi} \ \text{or} \ L/2<x_k\}}
\end{equation}
and
\begin{equation}
P_\phi^{(2)}:= \mathbf{1}_{(H_{L,d,k})^c}-\mathbf{1}_{\{x\in (H_{L,d,k})^c: \   x\in (\partial H_{L,d,k})_{\phi} \ \text{or} \ L/2<x_k<2L\}}.
\end{equation}
Then Corollary \ref{Lemma-trace-class-sob} yields a constant $C_1$, independent of $L$ and $p$, such that
\begin{align}\label{Lemma-commutation-last-combination-trace-bounds}
&\|(\mathbf{1}-\mathbf{1}_{H_{L,d,k}})\Op(\psi^{p}\otimes\mathbb{1}_{n})P_\phi^{(1)}\|_1\leq C_1p^{7d+1}, \nonumber \\
 &\|P_\phi^{(2)}\Op(\psi^{p}\otimes\mathbb{1}_{n})(\mathbf{1}_{H_{L,d,k}}-P_\phi^{(1)})\|_1\leq C_1p^{7d+1}.
\end{align}
We will only check this for the first trace norm. The second one works in a similar way. As in the proof of Lemma \ref{Lemma-commutation-combination}, we only verify the condition \eqref{Lemma-trace-class-sob-prep-req} in the case $\rho=1$. Let $x\in H_{L,d,k}\setminus \big(\partial H_{L,d,k}\big)_{\phi}$ such that $x_k\leq\tfrac{L}{2}$. Then we have 
\begin{equation}
\dist(x,(H_{L,d,k})^c)=\dist(x,\partial H_{L,d,k})\geq \phi(x_k)=\tfrac{1}{p^2}\sqrt{\max(0,px_k-p)}.
\end{equation}
Suppose now that $r> 0$ and that $\dist(x,(H_{L,d,k})^c)\leq r$, then we have
\begin{equation}
\max(0,px_k-p)\leq p^4r^2 \ \Rightarrow \ x_k\leq 1+p^3r^2
\end{equation}
As $x\in  H_{L,d,k}$, we have $x_k=\|x\|_\infty$ and we find a constant $C_\beta=C_\beta'p^{3d}$ such that the requirements of Corollary \ref{Lemma-trace-class-sob} hold with $\beta=2d$. The desired bound in \eqref{Lemma-commutation-last-combination-trace-bounds} follows.
We define the projections
\begin{equation}
P_\phi^{(3)}:=  \mathbf{1}_{H_{L,d,k}}-P_\phi^{(1)} = \mathbf{1}_{\{x\in H_{L,d,k}: \   x\in (\partial H_{L,d,k})_{\phi} \ \text{or} \ L/2<x_k\}}
\end{equation}
and
\begin{equation}
P_\phi^{(4)}:=  \mathbf{1}_{(H_{L,d,k})^c}-P_\phi^{(2)}= \mathbf{1}_{\{x\in (H_{L,d,k})^c: \   x\in (\partial H_{L,d,k})_{\phi} \ \text{or} \ L/2<x_k<2L\}}.
\end{equation}
With this at hand, we estimate
 \eqref{Lemma-commutation-last-combination-goal-3} by
\begin{equation}\label{Lemma-commutation-last-combination-goal-4}
\Big|\tr_{L^{2}(\R^{d})\otimes\C^n}\Big[\mathbf{1}_{H_{V,L}}\Big(\prod_{j=1}^{p-1}\Op(\mathcal{D})P_{\pi(j)}\Big)\Op(\mathcal{D})P_\phi^{(4)}\Op(\psi^{p}\otimes\mathbb{1}_{n})P_\phi^{(3)}\Big]\Big|+C_1p^{7d+1}.
\end{equation}
Using the cyclic property of the trace, we estimate the trace in \eqref{Lemma-commutation-last-combination-goal-4} by the following product of Hilbert-Schmidt norms
\begin{equation}\label{Lemma-commutation-last-combination-goal-5}
\Big\|P_\phi^{(3)}\Big(\prod_{j=1}^{p_{1,\pi}-1}\Op(\mathcal{D})P_{0}\Big)\Op(\mathcal{D})P_{1}\Big\|_2 \quad \times \quad \Big\|P_{1}\Big(\prod_{j=p_{1,\pi}+1}^{p-1}\Op(\mathcal{D})P_{0}\Big)\Op(\mathcal{D})P_\phi^{(4)}\Big\|_2.
\end{equation}
We split the sets corresponding to the projections $P_\phi^{(3)}$ and $P_\phi^{(4)}$ into two parts. One of distance $\tfrac{L}{2}$ to $H_{k-1}\setminus H_k$, the set corresponding to $P_1$. We estimate this part with Lemma \ref{Lemma-HS-distance-L}. The other part is contained in $\big(\partial H_{L,d,k}\big)_{\phi}\subset \big(\partial H_{L,d,k}\big)_{\phi_m}$, for $m\in\{1,\ldots,p-1\}$, and we estimate it with Lemma \ref{Lemma-commutation-last}. In total we obtain a constant $C_2>0$, independent of $L$ and $p$, such that \eqref{Lemma-commutation-last-combination-goal-5} is bounded from above by
\begin{equation}
C_2 d^{\tfrac{p-1}{2}}(p-1)^{3d+2}.
\end{equation}
With this we obtain a constant $C>0$, independent of $L$ and $p$, such that \eqref{Lemma-commutation-last-combination-goal-4} and with it \eqref{Lemma-commutation-last-combination-goal-2} are bounded from above by
\begin{equation}
C d^{\tfrac{p-1}{2}}p^{7d+1}.
\end{equation}
It remains to collect the contributions from all $\pi\in \{0,1\}^{p-1}$ with $\pi\neq 0$ and the contribution from the application of Lemma \ref{Lemma-commutation-combination} to obtain the desired bound
\begin{multline}
\Big|\tr_{L^{2}(\R^{d})\otimes\C^n}\Big[\mathbf{1}_{H_{V,L}}\Big(\big(\mathbf{1}_{H_{k}}\Op(\psi\mathcal{D})\mathbf{1}_{H_{k}}\big)^p-\big(\mathbf{1}_{H_{k-1}}\Op(\psi\mathcal{D})\mathbf{1}_{H_{k-1}}\big)^p\Big)\mathbf{1}_{H_{L,d,k}} \\
 -\mathbf{1}_{H_{V,L}}\Big(\big(\mathbf{1}_{H_{k}}\Op(\mathcal{D})\mathbf{1}_{H_{k}}\big)^p-\big(\mathbf{1}_{H_{k-1}}\Op(\mathcal{D})\mathbf{1}_{H_{k-1}}\big)^p\Big)\mathbf{1}_{H_{L,d,k}}\Op(\psi^{p}\otimes\mathbb{1}_{n})\Big]\Big|\\
 \leq C(2\sqrt{d})^{p-1}p^{7d+1}.
\end{multline}
This concludes the proof of the lemma.
\end{proof}
\subsection{Proof of Theorem \ref{Theorem-commutation}}
\label{subsec_proof_comm}
We are now ready to prove Theorem \ref{Theorem-commutation} by extending the results from the previous section to entire functions.
\begin{proof}[Proof of Theorem \ref{Theorem-commutation}]
Let $L\geq 1$. We denote the $p$th monomial by $g_p$ and choose a natural number $N_{0}\in\N$ such that $ p^{7d+1}\leq 2^{p}$ for all $p\geq N_{0}$. As described in the beginning of Section \ref{subsec_commutation}, we write
\begin{multline}
\tr_{L^{2}(\R^{d})\otimes\C^n}\bigg[\sum_{V\in\mathcal{F}^{(0)}}\mathbf{1}_{H_{V,L}}X_{LV,g}\bigg]  \\
= 2^{d}\tr_{L^{2}(\R^{d})\otimes\C^n}\bigg[\sum_{k=1}^{d}c_k\mathbf{1}_{H_{V,L}}\Big(g\big(\mathbf{1}_{H_{k}}X\mathbf{1}_{H_{k}}\big)-g\big(\mathbf{1}_{H_{k-1}}X\mathbf{1}_{H_{k-1}}\big)\Big) \mathbf{1}_{H_{L,d,k}}\bigg],
\end{multline}
with the constants $c_k:=(-1)^{d-k}d\frac{(d-1)!}{(k-1)!(d-k)!}$.
In the same way we write
\begin{multline}
\tr_{L^{2}(\R^{d})\otimes\C^n}\bigg[\sum_{V\in\mathcal{F}^{(0)}}\mathbf{1}_{H_{V,L}}\Op\big(\mathcal{D}\big)_{LV,g}\mathbf{1}_{H_{V,L}}\Op\big(g(\psi\otimes\mathbb{1}_{n})\big)\bigg] \\
=2^{d}\tr_{L^{2}(\R^{d})\otimes\C^n}\Big[\sum_{k=1}^{d}c_k\mathbf{1}_{H_{V,L}}\Big(g\big(\mathbf{1}_{H_{k}}\Op(\mathcal{D})\mathbf{1}_{H_{k}}\big)-g\big(\mathbf{1}_{H_{k-1}}\Op(\mathcal{D})\mathbf{1}_{H_{k-1}}\big)\Big)\\
\times\ \mathbf{1}_{H_{L,d,k}}\Op\big(g(\psi\otimes\mathbb{1}_{n})\big)\Big].
\end{multline}
By the triangle inequality, it suffices to estimate the individual terms, i.e. the terms
\begin{multline}
\mathbf{1}_{H_{V,L}}\Big(g\big(\mathbf{1}_{H_{k}}X\mathbf{1}_{H_{k}}\big)-g\big(\mathbf{1}_{H_{k-1}}X\mathbf{1}_{H_{k-1}}\big)\Big)\mathbf{1}_{H_{L,d,k}}\\
-\mathbf{1}_{H_{V,L}}\Big(g\big(\mathbf{1}_{H_{k}}\Op(\mathcal{D})\mathbf{1}_{H_{k}}\big)
-g\big(\mathbf{1}_{H_{k-1}}\Op(\mathcal{D})\mathbf{1}_{H_{k-1}}\big)\Big)\mathbf{1}_{H_{L,d,k}}\Op\big(g(\psi)\big)=:T_g.
\end{multline}
The bound for polynomial test functions is given in Lemma \ref{Lemma-commutation-last-combination}. For the analytic function $g$, we obtain:
\begin{equation}
\big|\tr_{L^{2}(\R^{d})\otimes\C^n}\big[T_g\big]\big| 
= \bigg|\tr_{L^{2}(\R^{d})\otimes\C^n}\bigg[\sum_{p=1}^{\infty}\omega_{p}T_{g_p}\bigg]\bigg|
\leq  C_{1}+C_{2}\sum_{p=N_{0}}^{\infty}|\omega_{p}|(4\sqrt{d})^p\leq C_{3},
\end{equation}
with constants $C_1, C_2, C_3>0$, independent of $L$. Taking the limit $L\rightarrow\infty$, we obtain \eqref{Theorem-commutation-bound}. This concludes the proof of the theorem.
\end{proof}

\section{Local asymptotic formula}
\label{sec_local_asymptotics}

After the commutation in the last section it remains to analyse the trace of the operator
\begin{equation}\label{op_loc}
\sum_{V\in\mathcal{F}^{(0)}}\mathbf{1}_{H_{V,L}}\Op\big(\mathcal{D}\big)_{LV,g}\mathbf{1}_{H_{V,L}}\Op\big(g(\psi\otimes\mathbb{1}_{n})\big).
\end{equation}
As before, all terms corresponding to the different vertices $V$ reduce to the case $V=V_0=\{0\}$ and $H_{V,L}=[0,L]^d$, up to suitable rotation and translation where we write $\Op\big(\mathcal{D}\big)_{LV_0,g}=\Op\big(\mathcal{D}\big)_{g}$, see \eqref{def_op_designated_corner}. It will be convenient to slightly modify this operator by replacing the second projection onto $H_{V,L}$ with the projection onto the slightly smaller set $(B_+)_L$, where $B_+:=\{y\in \,]0,\infty[^d: \ |y|< 1 \}$. As the set $\,]0,L[^d\setminus (B_+)_L$ has distance $L$ to the negative quadrant $\,]-\infty,0]^d$, the structure of the operator $\Op\big(\mathcal{D}\big)_{g}$ can be used to deduce, from Lemma \ref{Lemma-HS-distance-L}, that this replacement only yields an error of constant order. Due to the similarity of the proof to the one of Lemma \ref{Lemma-localisation}, we will omit a proof of this fact. We also replace the first projection onto $[0,L]^d$ by the projection onto its interior $\,]0,L[^d$.

In order to establish an asymptotic formula, we interpret the operator 
\begin{equation}
\mathbf{1}_{\,]0,L[^d}\Op\big(\mathcal{D}\big)_{g}\mathbf{1}_{(B_+)_L}\Op\big(g(\psi\otimes\mathbb{1}_{n})\big)
\end{equation}
as a multi-dimensional version of the one-dimensional localised operator studied in \cite{Widom1982} and try to adapt the one-dimensional proof to our case. In order to see the similarity of the operators, we note that both the symbols $\mathcal{D}$ and $\psi$ only depend on momentum space. Therefore, by the unitary dilatation $U_L$ on $L^2(\R^d)\otimes \C^n$, with $(U_L u)(x):=L^{\frac{d}{2}} u(Lx)$ for all $u\in L^2(\R^d)\otimes \C^n$, we have
\begin{multline}
\tr_{L^{2}(\R^{d})\otimes\C^n}\big[\mathbf{1}_{\,]0,L[^d}\Op\big(\mathcal{D}\big)_{g}\mathbf{1}_{(B_+)_L}\Op\big(g(\psi\otimes\mathbb{1}_{n})\big)\big] \\
=\tr_{L^{2}(\R^{d})\otimes\C^n}\big[\mathbf{1}_{\,]0,1[^d}\Op\big(\mathcal{D}\big)_{g}\mathbf{1}_{B_+}\Op_L\big(g(\psi\otimes\mathbb{1}_{n})\big)\big],
\end{multline}
where we used the fact that the symbol $\mathcal{D}$ is homogeneous of degree $0$ and therefore $\Op_L(\mathcal{D})=\Op(\mathcal{D})$, where $\Op_L$ denotes the operator with semi-classical parameter $L$, see e.g. \cite[Sec.~2.3]{BM-Widom}. Following the proof of the one-dimensional case in \cite{Widom1982}, the first step would be to determine the integral kernel of the operator $\Op\big(\mathcal{D}\big)_{g}$. In \cite{Widom1982} the kernel is derived explicitly with the use of the Mellin transform and its interaction with the Hilbert transform, which is the integral kernel of $\Op(\mathcal{D})$ in dimension $1$. Unfortunately, it seems difficult to generalise this step to our multi-dimensional case. This is on the one hand due to the far more complicated structure of $\Op\big(\mathcal{D}\big)_{g}$ in higher dimensions and on the other hand due to the fact that, instead of simply being the Hilbert transform, the integral kernel of $\Op(\mathcal{D})$ is then composed of its more complicated multi-dimensional generalisations, the Riesz transforms. The alternative chosen here is to establish the properties of the integral kernel, which are required for the second step, the local asymptotic formula, by hand. Unfortunately, we are only able to do this in the special case that $g$ is given by $g_2$ the monomial of second degree. This is the reason for the restriction on $g$ in Theorem \ref{Theorem-Asymptotics-Log}. Establishing these properties for $g_2$ is the content of the following section.

\subsection{Integral kernel properties for the second moment}
\begin{lem}\label{Lemma-kernel-prop}
Denote by $g_2$ the second monomial and let $K:\R^d\times\R^d\rightarrow\C^{n}$ be the corresponding integral kernel of the operator $\Op\big(\mathcal{D}\big)_{g_2}$. Then $K$ is continuous on $\,]0,\infty[^d\times \,]0,\infty[^d $, homogeneous of degree $-d$, has non-vanishing trace on $\,]0,\infty[^d\times \,]0,\infty[^d $ and satisfies the bound
\begin{equation}\label{polynomial-kernel-hölder}
\big|\tr_{\C^{n}}[K(z,z)-K(x,z)]\big|\leq C_d\sqrt{|x-z|} \ \big|\tr_{\C^{n}}[K(x,z)]\big|,
\end{equation}
for every $z\in\,]0,\infty[^d$ with $|z|=1$, $x\in  [z-\delta_z,z+\delta_z]^d$, with $\delta_z:=\min_{j=1,\ldots,d}\big(\frac{z_j}{2}\big)^2$, and a constant $C_d>0$, which only depends on the dimension $d$.
Furthermore, $K$ satisfies the bound
\begin{equation}\label{polynomial-kernel-bound}
\|K(x,z)\|_2\leq C_K (|x||z|)^{-\tfrac{d}{2}},
\end{equation}
for every $x,z\in \,]0,\infty[^d$, with a constant $C_K$ independent of $x$ and $z$.
\end{lem}
\begin{proof}
We begin with the homogeneity of $K$. 
We write out the operator $\mathbf{1}_{\,]0,\infty[^d}\Op\big(\mathcal{D}\big)_{g_2}\mathbf{1}_{\,]0,\infty[^d}$
\begin{multline}
\mathbf{1}_{\R_+^d}\Op\big(\mathcal{D}\big)_{g_2}\mathbf{1}_{\,]0,\infty[^d}=\sum_{k=0}^{d}(-1)^{k}\sum_{\mathcal{M}\subseteq \{1,\ldots,d\}\ : \ |\mathcal{M}|=k}\mathbf{1}_{\,]0,\infty[^d}\big(\mathbf{1}_{H_{\mathcal{M}}}\Op(\mathcal{D})\mathbf{1}_{H_{\mathcal{M}}}\big)^2\mathbf{1}_{\,]0,\infty[^d} \\
=(-1)^d\mathbf{1}_{\,]0,\infty[^d}\Op(\mathcal{D})\mathbf{1}_{]-\infty,0[^d}\Op(\mathcal{D})\mathbf{1}_{\,]0,\infty[^d}.
\end{multline}
Therefore, the kernel at some point $(x,z)\in \,]0,\infty[^d\times \,]0,\infty[^d$ can be written as
\begin{equation}\label{formula-kernel-p=2}
K(x,z)=(-1)^d\int_{\,]-\infty,0[^d}K_0(x,y)K_0(y,z) \dd y,
\end{equation}
with $K_0$ being the kernel of $\Op(\mathcal{D})$. The homogeneity of $K$ follows from the homogeneity of $K_0$ and substitution in the variable $y$.

With the definition of $K_0$, equation \eqref{formula-kernel-p=2} immediately yields that the trace of $K$ is either strictly positive or strictly negative on $\,]0,\infty[^d\times\,]0,\infty[^d$.

We continue with the continuity of the kernel. We again write
\begin{equation}
K(x,z)=(-1)^d\int_{\,]-\infty,0[^d}K_0(x,y)K_0(y,z) \dd y.
\end{equation}
The function $K_0(x,y)K_0(y,z)$ is continuous on $\,]0,\infty[^d\times \,]0,\infty[^d$ for every $y\in\,]-\infty,0[^d$ and integrable for fixed $(x_0,z_0)\in \,]0,\infty[^d\times\, ]0,\infty[^d$. We choose $\frac{\min(|x_0|,|z_0|)}{2}>\delta>0$ such that additionally $B_\delta(x_0,z_0)\subset\,]0,\infty[^d\times ]0,\infty[^d$.
With the bound \eqref{kernel_bound_distribution}, we obtain 
\begin{equation}\label{polynomial-kernel-cont-step-1}
\|K_0(x,y)K_0(y,z)\|_2\leq C_{K_0}^2\frac{1}{|x-y|^d|y-z|^d}.
\end{equation}
Let $(x,z)\in B_\delta(x_0,z_0)$. Then, by the choice of $\delta$, we have
\begin{equation}
\frac{|x_0-y|}{|x-y|}\leq 1+\frac{|x-x_0|}{|x-y|}\leq 2,
\end{equation}
for arbitrary $y\in\, ]-\infty,0[^d$.
We bound $\frac{|y-z_0|}{|y-z|}$ in the same way and obtain the following bound for the right-hand side of \eqref{polynomial-kernel-cont-step-1}
\begin{equation}
4^dC_{K_0}^2\frac{1}{|x_0-y|^d|y-z_0|^d}.
\end{equation}
As this is integrable on $]-\infty,0]^d$, the continuity of the kernel $K$ follows.

We now turn towards the proof of the bound \eqref{polynomial-kernel-hölder}. We estimate
\begin{align}\label{polynomial-kernel-hölder-step-1}
\big|\tr_{\C^{n}}[K(z,z)-K(x,z)]\big| &\leq \big(\tfrac{c_d}{2}\big)^2 \sum_{j=1}^d \int_{\,]-\infty,0]^d}\bigg|\frac{z_j-y_j}{|z-y|^{d+1}} -\frac{x_j-y_j}{|x-y|^{d+1}}\bigg| \frac{z_j-y_j}{|y-z|^{d+1}} \dd y \nonumber\\ 
&= \big(\tfrac{c_d}{2}\big)^2 \sum_{j=1}^d \int_{\,]-\infty,0]^d}\bigg|\frac{(z_j-y_j)|x-y|^{d+1}}{(x_j-y_j)|z-y|^{d+1}} -1\bigg|\frac{x_j-y_j}{|x-y|^{d+1}} \frac{z_j-y_j}{|y-z|^{d+1}} \dd y,
\end{align}
where we used that $x_j-y_j>0$, for all $y\in\, ]-\infty,0[^d$ and $j\in\{1,\ldots,d\}$, as $x\in  [z-\delta_z,z+\delta_z]^d\subset \,]0,\infty[^d$. We recall $c_{d}=\frac{\Gamma[(d+1)/2]}{\pi^{(d+1)/2}}$.
With the triangle inequality we obtain
\begin{equation}
\tfrac{z_j-y_j}{x_j-y_j}\in \big[1-\tfrac{|x_j-z_j|}{x_j-y_j}, 1+\tfrac{|x_j-z_j|}{x_j-y_j} \big]\subseteq \big[1-\tfrac{\sqrt{\delta_z|x-z|}}{\sqrt{\delta_z}}, 1+\tfrac{\sqrt{\delta_z|x-z|}}{\sqrt{\delta_z}} \big],
\end{equation}
where, in the last step, we used that $x\in  [z-\delta_z,z+\delta_z]^d$, $|x_j-z_j|\leq |x-z|$ and $x_j-y_j \geq \sqrt{\delta_z}$, for all $y\in\, ]-\infty,0[^d$, by the definition of $\delta_z$. The triangle inequality also yields
\begin{equation}
\tfrac{|x-y|}{|z-y|} \in \big[1-\tfrac{|x-z|}{|z-y|}, 1+\tfrac{|x-z|}{|z-y|} \big]\subseteq \big[1-|x-z|, 1+|x-z| \big],
\end{equation}
where we used that $|z-y|\geq 1$ for all $y\in\, ]-\infty,0[^d$, as $|z|=1$. With these two estimates, we find a constant $C_d>0$, only depending on the dimension $d$, such that the right-hand side of \eqref{polynomial-kernel-hölder-step-1} is bounded from above by
\begin{equation}
 C_d\sqrt{|x-z|}\big(\tfrac{c_d}{2}\big)^2 \sum_{j=1}^d \int_{\,]-\infty,0[^d}\frac{x_j-y_j}{|x-y|^{d+1}} \frac{z_j-y_j}{|y-z|^{d+1}} \dd y= C_d\sqrt{|x-z|} \ \big|\tr_{\C^{n}}[K(x,z)]\big|.
\end{equation}
This concludes the proof of the bound \eqref{polynomial-kernel-hölder}.

It remains to verify the bound \eqref{polynomial-kernel-bound}. We again have
\begin{equation}
K(x,z)=(-1)^d\int_{\,]-\infty,0[^d}K_0(x,y)K_0(y,z) \dd y.
\end{equation}
With the bound \eqref{kernel_bound_distribution}, we obtain 
\begin{equation}\label{polynomial-kernel-bound-step-1}
\|K(x,z)\|_2\leq C_{K_0}^2\int_{\,]-\infty,0[^d}\frac{1}{|x-y|^d|y-z|^d} \dd y=C_{K_0}^2\int_{\,]0,\infty[^d}\frac{1}{|x+y|^d|y+z|^d} \dd y.
\end{equation}
Using the fact that $x,y,z$ are all in the same quadrant, as well as the inequality $\sqrt{a+b}\geq \frac{1}{\sqrt{2}} (\sqrt{a}+\sqrt{b})$ for $a,b>0$, we see that
\begin{equation}
|x+y||y+z|\geq \frac{1}{2} (|x|+|y|)(|y|+|z|).
\end{equation}
Therefore, the right-hand side of \eqref{polynomial-kernel-bound-step-1} is bounded from above by
\begin{equation}
2^d C_{K_0}^2\int_{\,]0,\infty[^d}\frac{1}{[(|x|+|y|)(|y|+|z|)]^d}\dd y\leq 2^d C_{K_0}^2 \int_{\,]0,\infty[^d}\frac{1}{(|x||z|+|y|^2)^d}\dd y.
\end{equation}
Introducing spherical coordinates, we write this as
\begin{equation}
C_{K_0}^2 |S^{d-1}|\int_{0}^\infty\frac{r^{d-1}}{(|x||z|+r^2)^d}\dd r 
\leq C_{K_0}^2 |S^{d-1}|\int_{0}^\infty\frac{r}{(|x||z|+r^2)^{\tfrac{d+2}{2}}}\dd r  \leq C_K (|x||z|)^{-\tfrac{d}{2}},
\end{equation}
where the constant $C_K>0$ is independent of $x$ and $z$. This concludes the proof of the bound \eqref{polynomial-kernel-bound} and with it the proof of the lemma.
\end{proof}

\subsection{A local asymptotic formula}
We now prove a local asymptotic formula in a similar way as done in the one-dimensional case in \cite{Widom1982}. We employ a similar scaling argument as in \cite{Widom1982} and use the homogeneity of the integral kernel. But instead of reducing the positive half-axis $\,]0,\infty[\,$ to the single point $\{1\}$, we reduce the positive quadrant $\,]0,\infty[^d$ to the $(d-1)$-dimensional set $S^{d-1}_+=\{y\in \,]0,\infty[\,^d: \ |y|=1 \}$. The higher-dimensional nature of this reduced set adds some additional challenges to the proof and leads to more extensive requirements. In particular, we require that the trace of the integral kernel of the operator does not vanish. This requirement could be easily avoided in the one-dimensional case.

\begin{thm}\label{Local-asymptotic-formula}
Let $X$ be an integral operator on $L^2(\R^d)\otimes\C^{n}$. Let its integral kernel $K$ be continuous on $\,]0,\infty[^d\times\,]0,\infty[^d$, homogeneous of degree $-d$, have non-vanishing trace on $\,]0,\infty[^d\times\,]0,\infty[^d$ and satisfy the bounds \eqref{polynomial-kernel-hölder} and \eqref{polynomial-kernel-bound}. Furthermore, let the symbol $\sigma\in C_c^\infty(\R^d)$ be smooth and compactly supported. Then we have 
\begin{equation}\label{Local-asymptotic-formula-trace}
\tr_{L^{2}(\R^{d})\otimes\C^n}\big[\mathbf{1}_{\,]0,1[^d}X\mathbf{1}_{B_+}\Op_L(\sigma\otimes \mathbb{1}_{n})\big]=\log L \ \sigma(0)\int_{S^{d-1}_+} \tr_{\C^{n}}[K(y,y)]\dd y+O(1),
\end{equation}
as $L\rightarrow\infty$.
\end{thm}
\begin{proof}
The kernel of $\mathbf{1}_{\,]0,1[^d}X\mathbf{1}_{B_+}\Op_L(\sigma\otimes \mathbb{1}_{n})$ at a point $(x,y)\in \,]0,1[^d\times \R^d$ is given by
\begin{equation}
\frac{L^d}{(2\pi)^d}\int_{B_+}\int_{\R^d}K(x,z)e^{\i L\xi(z-y)}\sigma(\xi)\dd \xi \dd z.
\end{equation}
As this is continuous, we compute the trace on the left-hand side of \eqref{Local-asymptotic-formula-trace} by integrating this kernel along the diagonal, i.e. we have
\begin{multline}\label{Local-asymptotic-formula-step-1}
\tr_{L^{2}(\R^{d})\otimes\C^n}\big[\mathbf{1}_{\,]0,1[^d}X\mathbf{1}_{B_+}\Op_L(\sigma\otimes \mathbb{1}_{n})\big] \\
=\frac{L^d}{(2\pi)^d}\int_{\,]0,1[^d}\int_{B_+}\int_{\R^d}\tr_{\C^{n}}[K(x,z)]e^{\i L\xi(z-x)}\sigma(\xi)\dd \xi \dd z \dd x.
\end{multline}
We now need to evaluate this asymptotically as $L\rightarrow\infty$. The integrand on the right-hand side of \eqref{Local-asymptotic-formula-step-1} is clearly integrable. Hence, we can change the order of integration by Fubini's Theorem. We first introduce spherical coordinates in the variable $z$ and rewrite the right-hand side of \eqref{Local-asymptotic-formula-step-1} as
\begin{equation}
\frac{L^d}{(2\pi)^d}\int_{S^{d-1}_+}\int_{\,]0,1[^d}\int_{0}^1\int_{\R^d}r^{d-1}\tr_{\C^{n}}[K(x,ry)]e^{\i L\xi(ry-x)}\sigma(\xi)\dd \xi \dd r \dd x \dd y.
\end{equation}
Let now $\phi_y:\R^d\rightarrow\R$ be measurable functions, for every $y\in S^{d-1}_+$. We begin by estimating the expression
\begin{equation}\label{Local-asymptotic-formula-step-2}
\frac{L^d}{(2\pi)^d}\int_{S^{d-1}_+}\int_{\,]0,1[^d}\int_{0}^{1}\int_{\R^d}r^{d-1}\tr_{\C^{n}}[K(x,ry)]\big[1-\phi_y\big(\tfrac{x}{r}\big)\big]e^{\i L\xi(ry-x)}\sigma(\xi)\dd \xi \dd r \dd x \dd y.
\end{equation}
As we have $\sigma\in C_c^\infty(\R^d)$, integrating by parts $d+1$ times yields the bound
\begin{equation}
\bigg|\int_{\R^d}e^{\i L\xi(ry-x)}\sigma(\xi)\dd \xi\bigg|\leq C_\sigma (1+L|ry-x|)^{-d-1},
\end{equation}
with a constant $C_\sigma>0$ independent of $L$. With this, the absolute value of \eqref{Local-asymptotic-formula-step-2} is bounded from above by
\begin{equation}
\frac{L^d C_\sigma}{(2\pi)^d}\int_{S^{d-1}_+}\int_{\,]0,\infty[^d}\int_{0}^1 r^{d-1}\big|\tr_{\C^{n}}[K(x,ry)]\big|\big|1-\phi_y\big(\tfrac{x}{r}\big)\big|(1+L|ry-x|)^{-d-1} \dd r \dd x \dd y.
\end{equation}
With the substitution $x\mapsto rx$ and the homogeneity of $K$ this is equal to
\begin{equation}\label{Local-asymptotic-formula-step-3}
\frac{L^d C_\sigma}{(2\pi)^d}\int_{S^{d-1}_+}\int_{\,]0,\infty[^d}\int_{0}^1 r^{d-1}\big|\tr_{\C^{n}}[K(x,y)]\big|\big|1-\phi_y(x)\big|(1+Lr|y-x|)^{-d-1} \dd r \dd x \dd y.
\end{equation}
Carrying out the integration in the variable $r$, we obtain
\begin{multline}
\int_{0}^1 r^{d-1}(1+Lr|y-x|)^{-d-1} \dd r \leq \int_{0}^\infty r^{d-1}(1+Lr|y-x|)^{-d-1} \dd r \\
=\frac{r^d(Lr|y-x|+1)^{-d}}{d}\ \bigg|_{r=0}^{r=\infty}= \frac{1}{dL^d|y-x|^d}
\end{multline}
and therefore \eqref{Local-asymptotic-formula-step-3} is bounded from above by
\begin{equation}\label{Local-asymptotic-formula-step-4}
\frac{C_\sigma}{d(2\pi)^d}\int_{S^{d-1}_+}\int_{\,]0,\infty[^d}\big|\tr_{\C^{n}}[K(x,y)]\big|\big|1-\phi_y(x)\big||y-x|^{-d} \dd x \dd y.
\end{equation}
Due to the bound \eqref{polynomial-kernel-bound}, we have $\big|\tr_{\C^{n}}[K(x,y)]\big|\leq C_K|x|^{\tfrac{-d}{2}}$ for all $x\in \,]0,\infty[^d$. Choosing $\delta_y:=\big(\frac{\min_{1\leq j\leq d}y_j}{2}\big)^2$, we see that $ [y-\delta_y,y+\delta_y]^d\subset \,]0,\infty[^d$ and therefore $\tr_{\C^{n}}[K(x,y)]\neq 0$ for all $x\in [y-\delta_y,y+\delta_y]^d$, by assumption.
We choose the functions 
\begin{equation}
\phi_y(x):=\omega_y(x)\frac{\tr_{\C^{n}}[K(y,y)]}{\tr_{\C^{n}}[K(x,y)]}:=1_{[y-\delta_y,y+\delta_y]^d}(x)\frac{\tr_{\C^{n}}[K(y,y)]}{\tr_{\C^{n}}[K(x,y)]}.
\end{equation}
With these functions and the bound \eqref{polynomial-kernel-hölder}, the integral in the variable $x$ in \eqref{Local-asymptotic-formula-step-4} is bounded by a constant, which is independent of $y$, times
\begin{multline}
C_K\int_{\,]0,1/2]^d}|x|^{\tfrac{-d}{2}} \dd x+ \int_{\,]0,2]^d\setminus ([y-\delta_y,y+\delta_y]^d\cup \,]0,1/2]^d)}|y-x|^{-d}\dd x  \\
 + C_d\int_{[y-\delta_y,y+\delta_y]^d}|y-x|^{-d+1/2}\dd x + C_K\int_{\,]0,\infty[^d\setminus \,]0,2]^d}|x|^{\tfrac{-3d}{2}}\dd x .
\end{multline}
This is bounded from above by $C_1+C_2|\log(\delta_y)|$, where the constants $C_1, C_2>0$ are independent of $y$ and $L$. Carrying out the remaining integration in the variable $y$ in \eqref{Local-asymptotic-formula-step-4}, we see that \eqref{Local-asymptotic-formula-step-4} is bounded independently of $L$.

Hence, we reduced the trace in \eqref{Local-asymptotic-formula-trace} to
\begin{equation}\label{Local-asymptotic-formula-step-5}
\frac{L^d}{(2\pi)^d}\int_{S^{d-1}_+}\int_{\,]0,1[^d}\int_{0}^{1}\int_{\R^d}r^{d-1}\tr_{\C^{n}}[K(x,ry)]\phi_y\big(\tfrac{x}{r}\big) e^{\i L\xi(ry-x)}\sigma(\xi)\dd \xi \dd r \dd x \dd y
\end{equation}
up to an error term of constant order. 
With the homogeneity of $K$ we obtain
\begin{equation}
\tr_{\C^{n}}[K(x,ry)]\phi_y\big(\tfrac{x}{r}\big)=r^{-d}\omega_y\big(\tfrac{x}{r}\big)\tr_{\C^{n}}[K(y,y)]
\end{equation}
and with this, \eqref{Local-asymptotic-formula-step-5} reads
\begin{equation}
\frac{L^d}{(2\pi)^d}\int_{S^{d-1}_+}\tr_{\C^{n}}[K(y,y)]\int_{\R^d}\int_{0}^{1}r^{-1}\omega_y\big(\tfrac{x}{r}\big)\int_{\R^d} \sigma(\xi)e^{\i L\xi (ry-x)} \dd \xi  \dd r  \dd x\dd y.
\end{equation}
Here, we used that $\supp \phi_y\subset \,]0,1[^d$ and $r\leq 1$. Translating the variable $x$ by $ry$ and using the definition of $\omega_y$, this is equal to
\begin{equation}
\frac{L^d}{(2\pi)^d}\int_{S^{d-1}_+}\tr_{\C^{n}}[K(y,y)]\int_{\R^d}\int_{0}^{1}r^{-1}1_{[-\delta_y r,\delta_y r]^d}(x)\int_{\R^d} \sigma(\xi)e^{-\i L\xi x} \dd \xi  \dd r  \dd x\dd y.
\end{equation}
Substituting $x\mapsto \tfrac{x}{L}$ and carrying out the integration in the variable $\xi$, this is equal to
\begin{equation}\label{Local-asymptotic-formula-step-6}
(2\pi)^{-\tfrac{d}{2}}\int_{S^{d-1}_+}\tr_{\C^{n}}[K(y,y)]\int_{\R^d}\hat{\sigma}(x)\int_{0}^{1}1_{[-L\delta_y r,L\delta_y r]^d}(x)  \frac{\dd r}{r}  \dd x\dd y.
\end{equation}
We now consider the integral in the variable $r$. We see that it vanishes for $\|x\|_\infty\geq L\delta_y$. Therefore, we have
\begin{align}
\int_{0}^{1}1_{[-L\delta_y r,L\delta_y r]^d}(x)  \frac{\dd r}{r} &=1_{\{x\in\R^d \: \ \|x\|_\infty<L\delta_y \}}\int_{\frac{\|x\|_\infty}{L\delta_y}}^{1}\frac{\dd r}{r} \nonumber \\
&=1_{\{x\in\R^d \: \ \|x\|_\infty<L\delta_y \}}[\log L + \log \delta_y -\log\|x\|_\infty].
\end{align}
With this \eqref{Local-asymptotic-formula-step-6} reads
\begin{equation}
(2\pi)^{-\tfrac{d}{2}}\int_{S^{d-1}_+}\tr_{\C^{n}}[K(y,y)]\int_{\R^d}\hat{\sigma}(x)1_{\{x\in\R^d \: \ \|x\|_\infty<L\delta_y \}}[\log L + \log \delta_y -\log\|x\|_\infty]  \dd x\dd y.
\end{equation}
Rewriting this as
\begin{align}\label{Local-asymptotic-formula-step-7}
\log L &\ \sigma(0)\int_{S^{d-1}_+}\tr_{\C^{n}}[K(y,y)]\dd y \nonumber \\
&+  \sigma(0)\int_{S^{d-1}_+}\log \delta_y\tr_{\C^{n}}[K(y,y)]\dd y \nonumber \\
&- (2\pi)^{-\tfrac{d}{2}} \int_{S^{d-1}_+}\tr_{\C^{n}}[K(y,y)]\int_{\R^d}\hat{\sigma}(x)1_{\{x\in\R^d \: \ \|x\|_\infty<L\delta_y \}}\log\|x\|_\infty  \dd x\dd y\nonumber \\
&-(2\pi)^{-\tfrac{d}{2}} \int_{S^{d-1}_+}[\log L +\log \delta_y]\tr_{\C^{n}}[K(y,y)]\int_{\R^d}\hat{\sigma}(x)1_{\{x\in\R^d \: \ \|x\|_\infty\geq L\delta_y \}}  \dd x\dd y,
\end{align}
we have found the desired asymptotic coefficient. It only remains to verify that the other terms in \eqref{Local-asymptotic-formula-step-7} are of sufficiently low order. By the definition of $\delta_y$ and the bound \eqref{polynomial-kernel-bound}, the term in the second line of \eqref{Local-asymptotic-formula-step-7} is finite and it is clearly of constant order. As $\hat{\sigma}$ is a Schwartz function, the term in the third line of \eqref{Local-asymptotic-formula-step-7} is also finite and of constant order. For the terms in the fourth line of \eqref{Local-asymptotic-formula-step-7}, we also use that $\hat{\sigma}$ is Schwartz to obtain the bound
\begin{equation}
\bigg|\int_{\R^d}\hat{\sigma}(x)1_{\{x\in\R^d \: \ \|x\|_\infty\geq L\delta_y \}}  \dd x\bigg|\leq \frac{C_\sigma'}{\big(1+L\delta_y\big)^{\alpha}},
\end{equation}
with $C_\sigma'>0$, independent of $L$ and $\delta_y$, and  $0<\alpha<\tfrac{1}{2}$. This guarantees that the terms in the fourth line of \eqref{Local-asymptotic-formula-step-7} are both finite and of lower than constant order. With this, we see that \eqref{Local-asymptotic-formula-step-7} equals
\begin{equation}
\log L \ \sigma(0)\int_{S^{d-1}_+}\tr_{\C^{n}}[K(y,y)]\dd y+O(1),
\end{equation}
as $L\rightarrow\infty$. This concludes the proof of the theorem.
\end{proof}
\subsection{Proof of Theorem \ref{Theorem-Asymptotics-Log}}
\label{subsec_proof_asym_log}
We are now in a position to prove our second main result Theorem \ref{Theorem-Asymptotics-Log} by combining the results of Section \ref{sec_local_asymptotics} with the previous sections. For polynomials of degree $3$, we make use of the matrix structure of the Dirac matrices.
\begin{proof}[Proof of Theorem \ref{Theorem-Asymptotics-Log}]
In the case that $g$ is of degree $1$, there is nothing to show. We start with the case $g=g_2$. By Theorem \ref{theorem_higher_order_terms} we have
\begin{align}
\tr_{L^{2}(\R^{d})\otimes\C^n}\big[g_2\big(\mathbf{1}_{\Lambda_{L}}\Op(\psi\mathcal{D})\mathbf{1}_{\Lambda_{L}}\big)\big]=&\sum_{m=0}^{d-1}(2L)^{d-m}A_{m,g_2,b} \nonumber \\
&+\tr_{L^{2}(\R^{d})\otimes\C^n}\bigg[\sum_{V\in\mathcal{F}^{(0)}}\mathbf{1}_{H_{V,L}}\Op(\psi\mathcal{D})_{LV,g_2}\bigg] +O(1), 
\end{align}
as $L\rightarrow\infty$. It remains to compute the trace on the right-hand side. By Theorem \ref{Theorem-commutation} and the discussion at the beginning of Section \ref{sec_local_asymptotics} we have for $V=V_0=\{0\}$
\begin{equation}
\tr_{L^{2}(\R^{d})\otimes\C^n}\big[\mathbf{1}_{H_{V_0,L}}\Op(\psi\mathcal{D})_{LV_0,g_2}\big]=\tr_{L^{2}(\R^{d})\otimes\C^n}\big[\mathbf{1}_{\,]0,1[^d}\Op(\mathcal{D})_{g_2}\mathbf{1}_{B_+}\Op_L(\psi^2\otimes \mathbb{1}_{n})\big]+O(1),
\end{equation}
as $L\rightarrow\infty$. By Lemma \ref{Lemma-kernel-prop}, the operator $\Op(\mathcal{D})_{g_2}$ fulfils the requirements of Theorem \ref{Local-asymptotic-formula}. An application of Theorem \ref{Local-asymptotic-formula} yields
\begin{equation}
\tr_{L^{2}(\R^{d})\otimes\C^n}\big[\mathbf{1}_{\,]0,1[^d}\Op(\mathcal{D})_{g_2}\mathbf{1}_{B_+}\Op_L(\psi^2\otimes \mathbb{1}_{n})\big]=\log L \ \int_{S^{d-1}_+} \tr_{\C^{n}}[K_{g_2}(y,y)]\dd y+O(1),
\end{equation}
as $L\rightarrow\infty$, where $K_{g_2}$ is the integral kernel of the operator $\Op(\mathcal{D})_{g_2}$. Collecting the contributions for all $V\in\mathcal{F}^{(0)}$ concludes the proof of the Theorem in the case $g=g_2$.

We now turn to the third moment, i.e. the case $g=g_3$. The results from Sections \ref{sec_higher_order} and \ref{sec_commutation} apply to analytic functions, in particular also in the case $g=g_3$. Therefore, it only remains to justify an application of the results from Section \ref{sec_local_asymptotics}. As the Dirac matrices have vanishing trace, only the terms where the identity matrix occurs once or thrice contribute to the trace. Due to the structure of the operator $\Op\big(\mathcal{D}\big)_{g_3}$, the case where the identity occurs thrice vanishes. Therefore, we have
\begin{multline}
\tr_{L^{2}(\R^{d})\otimes\C^n}\big[\mathbf{1}_{\,]0,1[^d}\Op\big(\mathcal{D}\big)_{g_3}\mathbf{1}_{B_+}\Op_L(\psi^3\otimes \mathbb{1}_{n})\big] \\
=
\frac{3}{2}\tr_{L^{2}(\R^{d})\otimes\C^n}\big[\mathbf{1}_{\,]0,1[^d}\Op\big(\mathcal{D}\big)_{g_2}\mathbf{1}_{B_+}\Op_L(\psi^3\otimes \mathbb{1}_{n})\big].
\end{multline}
As $\psi^3(0)=\psi^2(0)=1$, an application of Theorem \ref{Local-asymptotic-formula} yields
\begin{equation}
\tr_{L^{2}(\R^{d})\otimes\C^n}\big[\mathbf{1}_{\,]0,1[^d}\Op(\mathcal{D})_{g_3}\mathbf{1}_{B_+}\Op_L(\psi^3\otimes \mathbb{1}_{n})\big]=\frac{3}{2}\log L \ \int_{S^{d-1}_+} \tr_{\C^{n}}[K_{g_2}(y,y)]\dd y+O(1),
\end{equation}
as $L\rightarrow\infty$. Using the structure of the Dirac matrices again, we see that 
\begin{equation}
\frac{3}{2}\tr_{\C^{n}}[K_{g_2}(y,y)]=\tr_{\C^{n}}[K_{g_3}(y,y)],
\end{equation}
which concludes the proof of the Theorem in the case $g=g_3$. The result for arbitrary polynomials of degree at most three follows by linearity.
\end{proof}


\section*{Acknowledgements}
The author is very grateful to his Ph.D. advisor Peter Müller for illuminating discussions on the topic, as well as helpful remarks on the first two sections of the present article. The author is also very grateful to Alexander Sobolev for bringing to his attention the article \cite{Pfirsch}.  
This work was partially supported by the Deutsche Forschungsgemeinschaft (DFG, German Research Foundation) 
	-- TRR 352 ``Mathematics of Many-Body Quantum Systems and Their Collective Phenomena" -- Project-ID 470903074.
	


\providecommand{\noopsort}[1]{} \providecommand{\singleletter}[1]{#1}
  \providecommand{\noopsort}[1]{} \providecommand{\singleletter}[1]{#1}

\end{document}